\documentclass[twoside,16pt]{elsarticle}
\usepackage{amssymb,amsmath,natbib,graphicx,lineno,setspace,amsthm,float}
\usepackage{setspace}
\usepackage{rotating}
\usepackage[ruled]{algorithm2e}
\usepackage{threeparttable}
\biboptions{numbers,sort&compress}
\bibliography{mybib}
\newtheorem{theorem}{Theorem}
\newtheorem{definition}{Definition}
\newtheorem{lemma}{Lemma}
\newtheorem{remark}{Remark}
\newtheorem{example}{Example}
\newtheorem{assumption}{Assumption}[section]
\newtheorem{corollary}{Corollary}[section]
\newtheorem{proposition}{Proposition}[section]

\usepackage[CJKbookmarks,colorlinks,driverfallback=dvipdfm,linkcolor=red]{hyperref}
\usepackage{mathrsfs}
\journal{arXiv}
\numberwithin{equation}{section}
\numberwithin{theorem}{section}
\numberwithin{lemma}{section}
\numberwithin{definition}{section}
\numberwithin{example}{section}
\numberwithin{assumption}{section}
\begin{document}
	\begin{frontmatter}
		\title{Stochastic generalized Kolmogorov systems with small diffusion: II. Explicit approximations for periodic solutions in distribution}
		\author[a,b] {Baoquan Zhou}
        \author[c] {Hao Wang}
        \ead{hao8@ualberta.ca}
        \author[c] {Tianxu Wang}
		\author[b] {Daqing Jiang\corref{cor1}}
	    \ead{jiangdq\_mail@163.com}
		\cortext[cor1]{Corresponding author.}
		\address[a] {School of Mathematics and Statistics, Key Laboratory of Applied Statistics of MOE, Northeast Normal University, Changchun 130024, P.R. China}
        \address[b] {College of Science, China University of Petroleum (East China), Qingdao 266580, P.R. China}
        \address[c] {Department of Mathematical and Statistical Sciences, University of Alberta, Edmonton, Alberta T6G 2G1, Canada}
		\begin{abstract}

  This paper is Part II of a two-part series on coexistence states study in stochastic generalized Kolmogorov systems under small diffusion. Part I provided a complete characterization for approximating invariant probability measures and density functions, while here, we focus on explicit approximations for periodic solutions in distribution. Two easily implementable methods are introduced: periodic normal approximation (PNOA) and periodic log-normal approximation (PLNA). These methods offer unified algorithms to calculate the mean and covariance matrix, and verify positive definiteness, without additional constraints like non-degenerate diffusion. Furthermore, we explore essential properties of the covariance matrix, particularly its connection under periodic and non-periodic drift coefficients. Our new approximation methods significantly relax the minimal criteria for positive definiteness of the solution of the discrete-type Lyapunov equation. Some numerical experiments are provided to support our theoretical results.
		\end{abstract}
	\end{frontmatter}
	\section{Introduction}
To describe the dynamics of interacting populations, Kolmogorov systems (a class of deterministic systems) have been widely used in the modeling of ecological and biological processes, which are in general form given by
 \begin{equation}\label{1.1}
 \dot{x}_i(t)=x_i(t)b_i(t,x_1(t),...,x_n(t)),\ \ \ \ i=1,...,n,
 \end{equation}
 where the vector field $(b_i(t,\cdot))$ is $\theta$-periodic in $t$ for some $\theta>0$. The periodic time-dependence in (\ref{1.1}) is frequently used to model seasonal variations or time recurrence \cite{S01}. 

However, in most cases, capturing real-world processes through (\ref{1.1}) is challenging due to inevitable environmental noises. 
 While random perturbations in the studied system are often small compared to the deterministic component \cite{S02,S03,S04}, they can still have a significant destabilizing impact on the asymptotic behavior of these systems \cite{S05}.
  Thus, investigating the long-term properties of (\ref{1.1}) under small perturbations is a fundamental issue of both practical and theoretical significance.
 To present a wide range of possibilities for applications, we now incorporate small perturbations into a generalized version of (\ref{1.1}), yielding the following stochastic generalized Kolmogorov systems:
 
 \begin{equation}\label{1.2}
 \begin{cases}
 dX_{\epsilon,i}(t)=f_i(t,\mathbf{X}_{\epsilon}(t))dt+\sqrt{\epsilon}X_{\epsilon,i}(t)\sum\limits_{j=1}^{N}g_{ij}(t,\mathbf{X}_{\epsilon}(t))dW_j(t),\ \ \ \ i=1,...,n,\\
 \mathbf{X}_{\epsilon}(0)=\mathbf{x}_0\in\mathbb{R}^n,
 \end{cases}
 \end{equation}
 where $\epsilon>0$ is a small parameter, $\mathbf{X}_{\epsilon}(t)=(X_{\epsilon,1}(t),...,X_{\epsilon,n}(t))^{\top}_{t\ge 0}$ (the superscript ``$\top$" stands for the transpose), and $(W_1(t),...,W_N(t))^{\top}:=\mathbf{W}(t)$ is an $N$-dimensional vector of independent standard Brownian motions. $G_c=(x_ig_{ij})$ is an $n\times N$ matrix-valued function on $\mathbb{R}^n$, called the $\mathit{Kolmogorov\ noise\ matrix}$. Moreover, $G_c(t,\cdot)$ and the drift coefficient $(f_i(t,\cdot))$ are $\theta$-periodic in $t$. The formulation (\ref{1.2}) shows wider application and covers the most common dynamical systems in the literature, such as epidemic models, Lotka--Volterra models and their variants, chemostat models, algal growth models, etc.

 
A longstanding central concern in ecology and biology revolves around the stable coexistence of interacting populations \cite{S06}. For periodic systems, the focus lies on the existence of positive periodic solutions. While much progress has been made for deterministic systems \cite{S07,S08,S09}, yielding tools like fixed point theorems and degree theory, finding periodic solutions for stochastic differential equations (SDEs) still remains a formidable challenge. One key difficulty arises in defining $\mathit{stochastic\ periodic\ solutions}$ (SPSs). Meyn and Tweedie \cite{S01} initiated the study by introducing the concept of recurrence for Markov processes. Khasminskii \cite{S11} subsequently defined SPSs in the sense of periodic Markov processes and obtained an existence theorem through Lyapunov functional methods. However, recurrence is not precise enough to capture stochastic periodicity \cite{S12}. In recent decades, scholars have proposed and explored pathwise SPSs (e.g., \cite{S01,S02,S03}). Nonetheless, due to the complexities and unpredictability of real-world random perturbations, requiring the recurrence of solution processes under diffusion to specific sample paths, such as the evolution of the annual sunspot number, is inappropriate. More precisely, randomness and periodicity become intertwined in periodic SDEs. Mellah and Fitte \cite{S16} further demonstrated that there are no periodic solutions in the sense of probability or moment for SDEs with almost periodic coefficients.


  Despite the high randomness of the solutions of SDEs with small diffusion coefficients, they may still exhibit stable statistical properties over a large time span, such as expectations and covariances \cite{S17}. This fact inspires us to consider SPSs in distribution (SPSD for abbreviation) for (\ref{1.2}), i.e., find a solution $(\mathbf{X}_{\epsilon}(t))$ such that $\mathbf{X}_{\epsilon}(t+\theta)$ and $\mathbf{X}_{\epsilon}(t)$ are identically distributed, $\forall\ t\ge 0$. This definition naturally captures the periodicity and randomness of the solution of periodic SDEs. Most recently, some efforts (in particular, technology level) have been made for the existence of SPSD. For example, Chen et al. \cite{S18} developed a weak Halanay-like criterion using Skorokhod theorems. Moreover, Ji et al. \cite{S19} further studied this issue for SDEs with irregular coefficients (including non-degenerate and degenerate diffusions) under the existence of unbounded Lyapunov functions.
	
Apart from the existence of SPSD, the associated explicit probability distributions and density functions are also required for a complete characterization of the coexistence state of periodic SDEs. Such probability density is determined by a corresponding Kolmogorov--Fokker--Planck (KFP) equation. Thus, we expect to obtain a periodic solution of the following KFP equation associated with (\ref{1.2}):
	\begin{equation}\label{1.3}
	\begin{cases}
	\dfrac{\partial}{\partial t}\mathcal{P}_{\epsilon}(t,\mathbf{x})=-\sum\limits_{i=1}^{n}\dfrac{\partial}{\partial x_i}\bigl(f_i(t,\mathbf{x})\mathcal{P}_{\epsilon}(t,\mathbf{x})\bigr)+\dfrac{\epsilon}{2}\sum\limits_{i,j=1}^{n}\dfrac{\partial^2}{\partial x_i\partial x_j}\bigl((G_cG_c^{\top})_{ij}\mathcal{P}_{\epsilon}(t,\mathbf{x})\bigr)\\
	\ \ \ \ \ \ \ \ \ \ \ \ \ :=\mathscr{L}_{\epsilon,\theta}\mathcal{P}_{\epsilon}(t,\mathbf{x}),\ \ \ \forall\ (t,\mathbf{x})\in[0,\infty)\times\mathbb{R}^n,\\
	\int_{\mathbb{R}^n}\mathcal{P}_{\epsilon}(t,\mathbf{x})d\mathbf{x}=1,\ \ \ \mathcal{P}_{\epsilon}(t,\mathbf{x})\ge 0,\\
	\mathcal{P}_{\epsilon}(t+\theta,\mathbf{x})=\mathcal{P}_{\epsilon}(t,\mathbf{x}),\ \ \ \forall\ (t,\mathbf{x})\in[0,\infty)\times\mathbb{R}^n,
	\end{cases}
	\end{equation}
	where $\mathscr{L}_{\epsilon,\theta}$ is the $\theta$-periodic KFP operator. The desired result can follow from $\int_{\mathbb{A}}\mathcal{P}_{\epsilon}(t,\mathbf{x})d\mathbf{x}$, $\forall\ \mathbb{A}\subset\mathbb{R}^n$.

 Due to its potential prevalence in applications, the study of SPSD has gained increasing attention in recent years, mainly focusing on central issues such as existence, uniqueness, and convergence in mean \cite{S01,S05,S18,S19,S20}. However, the explicit characterization of SPSD remains an open question. Currently, there are virtually no available approaches to solve KFP equations analytically, as most stochastic dynamical systems (especially (\ref{1.2})) are highly complex and nonlinear. Existing numerical methods for solving KFP equations, including techniques like Monte Carlo simulation and numerical PDE methods (e.g., \cite{S21,S22,S23,S24}), are still underdeveloped and only focus on autonomous systems. In fact, applying such techniques to (\ref{1.3}) is almost impossible because of two significant challenges.
 The first challenge comes from the need for high-precision local solutions and the management of large numerical domains. We only consider the autonomous case of (\ref{1.2}) for convenience, i.e., $\mathcal{P}_{\epsilon}(t,\mathbf{x})$ is time-independent ($\mathcal{P}_{\epsilon}(\mathbf{x})$ for short). It follows from the Freidlin--Wentzell theory \cite{S21} that the density function $\mathcal{P}_{\epsilon}(\mathbf{x})$ under small diffusion is concentrated on an $O(\epsilon)$-neighborhood of all attractors of the deterministic system of (\ref{1.2}) in a large probability. The grid size of discretization is then required to be sufficiently small for such a local numerical solution with high precision. Meanwhile, (\ref{1.3}) is defined on $\mathbb{R}^n$, lacking a well-defined boundary condition. The common practice is to let the numerical domain (with zero-value boundary condition) cover the $O(\epsilon)$-neighborhood of the attractors with sufficient margin \cite{S25}. This inevitably incurs significant computational efforts. Moreover, the local boundary condition of such $O(\epsilon)$-neighborhood is unknown, making the issue even more challenging. The second difficulty arises from stochastic periodicity. For periodic cases, the aforementioned $O(\epsilon)$-neighborhood will be no longer time-independent. Instead, it is replaced by an $O(\epsilon,t)$-neighborhood family on $t\in[0,\theta)$, which requires higher computational costs. Furthermore, the discrete equation of (\ref{1.3}) may not have the numerical solution in the presence of the periodicity of $\mathcal{P}_{\epsilon}(t,\mathbf{x})$. These difficulties and challenges force us to only approximate the probability distributions for an explicit characterization of the SPSD of (\ref{1.2}). 
	
 
 The probability distributions of SPSD in the time-independent case degenerate into their autonomous analogs, called $\mathit{invariant\ probability\ measures}$ (IPMs). Recently, a few numerical methods for approximating IPMs have been developed including stochastic theta method \cite{S26}, the truncated Euler--Maruyama (EM) scheme \cite{S27,S28}, and the backward EM scheme \cite{S29,S30}. Although these numerical methods have demonstrated their ability to approximate the underlying IPMs, there are two limitations: (i) The approximate expressions of their invariant probability density functions (IPDFs) cannot be obtained, and (ii) such methods are sample-path based, so whether they can be applied to the periodic case is unclear. Even if applicable, it would require at least a multiple-fold increase in computational effort and convergence analysis. By contrast, an interesting idea is to consider special continuous probability distributions to approximate IPMs \cite{S03}, and the expressions of IPDFs can then be explicitly approximated. Our study in the series is along this line. Let us briefly recall the associated important work to date. The early classical assumption is that the IPM can be approximated by a Gibbs measure, which takes the form
	$$\int_{\mathbb{A}}\frac{1}{K}e^{-H(\mathbf{y})}d\mathbf{y},\ \ \ \forall\ \mathbb{A}\subset\mathbb{R}^n,$$
 where $H(\cdot)$ is the quasi-potential function \cite{S31}. However, this assumption is only applicable to gradient systems and cannot be generally verified due to the high regularity requirements of $H(\cdot)$ \cite{S32}. In contrast, Li et al. \cite{S33} initiated a study by introducing small perturbations to a biochemical system with a unique stable equilibrium $\mathbf{x}^*$. Under non-degenerate diffusion (i.e., the diffusion matrix $\aleph$ is positive definite), they pointed out that the IPM around $\mathbf{x}^*$ may be approximately normally distributed, with the covariance matrix $\Sigma$ satisfying the continuous-type Lyapunov equation $\Sigma\boldsymbol{\varpi}^{\top}+\boldsymbol{\varpi}\Sigma+\aleph=\mathbb{O}$ ($\Im_c(\Sigma,\boldsymbol{\varpi},\aleph)=\mathbb{O}$ for short), where $\mathbb{O}$ is zero matrix. Zhou et al. \cite{S34} further developed their work by combining superposition principle and matrix algebra approach to obtain an approximate expression of the local IPDF of a stochastic avian influenza model. By now, such normal approximation method has been relatively well used for stochastic epidemic models in dimension$\le 5$ under non-degenerate diffusion \cite{S35,S36,S37,S38}. One of its key ideas is to study the ``standard $L_0$-algebraic equation" (e.g., \cite[Lemma 3]{S34} and \cite[Lemma 3.1]{S38}, actually a special class of the continuous-type Lyapunov equation), which helps to derive the explicit form of $\Sigma$. 
 

It should be mentioned that there is a big technique leap between the relevant approximation analysis of specific epidemic models and that of system (\ref{1.2}).  The objective of our series is to bridge this gap through two intermediary research steps, denoted as $(r_1)$ and $(r_2)$. Research $(r_1)$ aims to introduce novel mathematical techniques that extend the existing normal approximation from low-dimensional models to our general setting of (\ref{1.2}) under autonomous case. Meanwhile, Research $(r_2)$ is devoted to the extension from autonomous case to periodic case (see Fig. 1).
  
 	\begin{figure}[H]
 		\begin{center}
 			\resizebox{15cm}{1.5cm}{\includegraphics{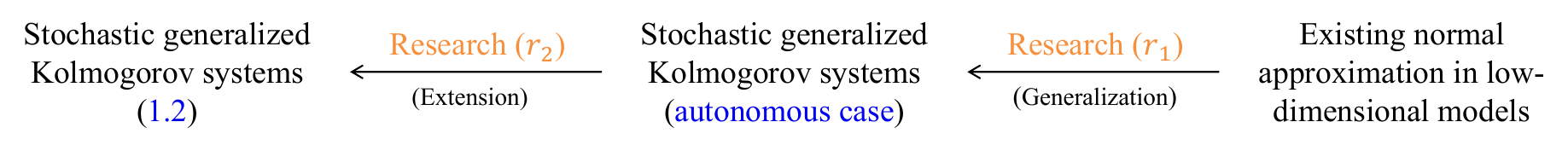}}
 		\end{center}
 		\makeatletter\def\@captype{figure}\makeatother \caption{The road map for the study of our two-part series.}
 	\end{figure}
   Recall the main issues that have been addressed in Research $(r_1)$ (see part I of the series \cite{S39}). Existing analysis of standard $L_0$-algebraic equation is limited to the positive definiteness of solutions, which is achieved through direct calculation (e.g., \cite[Appendix B]{S38}) and thus becomes increasingly challenging as equation dimension grows. We used the residue theorem and linear control theory to provide a complete characterization of the fundamental properties of arbitrary dimensional $L_0$-algebraic equation, including specific form and positive definiteness of solution and structure of eigenvalues. Additionally, in some stochastic risk-adjusted volatility models, the IPMs are long right-tailed (see \cite[Fig.4]{S27}) and cannot use regular normal approximation methods. To address this, we introduced two approaches for approximating the IPM and IPDF: (i) \textit{log-normal approximation} (LNA) and (ii) \textit{updated normal approximation} (uNA) and proposed new theoretical algorithms for calculating the expression of the covariance matrix of LNA (or uNA) approach. For our new approach to be available in degenerate diffusion, we substantially relax the classical conditions for ensuring positive definiteness of the solution of general Lyapunov equation $\Im_c(\Sigma,\boldsymbol{\varpi},\aleph)=\mathbb{O}$ by using matrix algebra approach and Routh--Hurwitz criterion.

    Then a question naturally arises whether the techniques and theories used in \cite{S39} can be applicable to Research $(r_2)$. Regrettably, the answer is not as positive as desired. In one respect, the study of the continuous-type Lyapunov equation (in particular, standard $L_0$-algebraic equation) is a key issue to obtain the covariance matrix $\Sigma$ of the approximation method in Research $(r_1)$. But when stochastic periodicity is taken into account, the IPMs will no longer exist, replaced by a time-dependent probability distribution. Accordingly, such covariance matrix $\Sigma$ becomes time-varying, denoted as $\Sigma(t)$, no longer satisfying the continuous-type Lyapunov equation. To further ensure the periodicity of the underlying probability distribution, we prove that $\Sigma(0)$ is determined by the discrete-type Lyapunov equation, as shown in Section 3. Thus for Research $(r_2)$, the fundamental properties of general discrete-type Lyapunov equation are the main focus. However, the associated analysis is even more challenging. In another respect, when periodic variations are incorporated, the global attractor of the deterministic system of (\ref{1.2}) will shift from several stable equilibria to periodic orbits and may even can not be calculated, which implies that Hurwitz matrix together with Routh--Hurwitz criterion \cite{S40} is unavailable and the explicit form of $\Sigma(t)$ is more difficult to obtained. Some generalized concepts and results are urgently developed.

  These challenges motivate our current work. Inspired by the Gaussian-like elimination idea in part I of the series \cite{S39}, together with new transformation techniques and more delicate analysis, we aim to develop a unified algorithm framework for explicitly approximating the probability distribution of the SPSD of (\ref{1.2}). Notably, although (\ref{1.2}) is more general and realistic in applications, the analysis of such systems is much more difficult. To advance this kind work, a canonical form of the discrete-type Lyapunov equation called ``standard $\mathbb{L}_0$-algebraic equation" is introduced and its asymptotic features are completely studied. The greatest challenge stems from obtaining the positive definiteness of general discrete-type Lyapunov function under degenerate diffusion. By transforming each sub-equation as much as possible into an equivalent form with a higher dimensional standard $\mathbb{L}_0$-algebraic equation structure, along with the application of the matrix algebra approach, linear control theory, and Gaussian-like elimination method, we effectively address this.

    Our main contributions are listed as follows:
	\begin{itemize}
		\item Two easily implementable explicit methods for approximating the probability distribution and density function of the SPSD of (\ref{1.2}) are developed, including (i) \textit{periodic normal approximation} (PNOA) for roughly symmetric measure, and (ii) \textit{periodic log-normal approximation} (PLNA) for right-skewed measure. A more biologically reasonable stochastic modeling assumption that falls into our general setting is provided; see Section 5 in details. 
  \item
  New theoretical algorithms for calculating the expression of the covariance matrix of PNOA (or PLNA) are proposed. A novelty of these algorithms is that their positive definiteness can be verified simultaneously. We further provide two modified approximation algorithms under slightly complex diffusions, and examine basic properties of such covariance matrix without periodic drift coefficients. Notably, our matrix transformations differ from the classical command "dlyap($\cdot,\cdot$)" in MATLAB software, offering simpler computational structure and wider applicability.
  \item A complete characterization of general standard $\mathbb{L}_0$-algebraic equation is provided. The minimal criteria for guaranteeing positive definiteness of the solution of general discrete-type Lyapunov equation are substantially relaxed. Our results can be regarded as a generalization of Routh--Hurwitz criterion involving periodicity.
	\end{itemize}
	
	The outline of this paper is as follows. Section 2 begins with necessary notations, lemma and mathematical definitions. Then some fundamental properties of standard $\mathbb{L}_0$-algebraic equations are given. Section 3 develops a complete framework of PNOA method for the probability distribution and density function of the SPSD of (\ref{1.2}), including basic formulation, unified explicit algorithms, and characterization of nondegeneracy of the approximated probability distribution. The corresponding framework of PLNA method is provided in Section 4. Section 5 presents an application of our main results in stochastic population models. An Appendix, which contains the proofs of key auxiliary propositions in Section 2, is given at the end of the paper.
	\section{Preliminaries}
	Throughout the paper, we work on a complete filtered probability space $\{\Omega,\mathscr{F},\{\mathscr{F}_t\}_{t\ge 0},\mathbb{P}\}$, where $\{\mathscr{F}_t\}_{t\ge 0}$ is a filtration satisfying the usual conditions \cite{S41}. Let $\mathbb{E}$ be the expectation related to $\mathbb{P}$, and $\mathbf{W}(t)$ is adapted to $\{\mathscr{F}_t\}_{t\ge 0}$. To help the reading, Table 1 and Definitions \ref{2:1}-\ref{2:4} show the mathematical notations and several important classes of matrices used in this paper.  
	\begin{table}[hthp]
		\centering
		\caption{A glossary of mathematical notations used in the paper}
		\label{table:para}
		\begin{small}
		\begin{center}
		\begin{threeparttable}
		\begin{tabular}{ll}
		\noalign{\smallskip}\hline\noalign{\smallskip}
		$\mathbf{I}_l$ & $l$-dimensional identity matrix\\
		$\mathbf{e}_l$ & $:=(1,0,0,...,0)^{\top}\in\mathbb{R}^l$\tnote{\textcolor{blue}{$\ddag1$}}\\
		$\boldsymbol{\beta}_l$ & $:=(0,0,...,0,1)\in\mathbb{R}^{1\times l}$\\
		$\mathbb{O}_{l,q}$ (resp., $\mathbb{O}_{l}$)\tnote{\textcolor{blue}{$\ddag2$}} & $l\times q$ (resp., $l\times l$)-dimensional zero matrix\\
		$\ln\mathbf{X}$ & $:=(\ln X_1,...,\ln X_l)^{\top}$, where $\mathbf{X}=(X_1,...,X_l)^{\top}$\\
		$\mathbb{S}_l^k$\tnote{\textcolor{blue}{$\ddag1$}} & index set $\{k+1,k+2,...,l\}\ (-1\le k\le l)$\\
		$\Im_d(\Sigma,\boldsymbol{\varpi},\aleph)=\mathbb{O}$ & the discrete-type Lyapunov equation $\boldsymbol{\varpi}\Sigma\boldsymbol{\varpi}^{\top}-\Sigma+\aleph=\mathbb{O}$, where $\Sigma$ is real symmetric\\
		$|\cdot|$ & the Euclidean norm (or determinant of a matrix)\\
		$\boldsymbol{a}^{\langle k\rangle}$ & subvector formed by the first $k$ part of vector $\boldsymbol{a}$\\
		$A^{(k)}$ (resp., $\boldsymbol{a}^{(k)}$) & submatrix formed by the first $k\times k$ part of matrix $A$ (resp., the $k$th element of vector $\boldsymbol{a}$)\\
		$A^{-1}$ & inverse matrix of $A$\\
		$A^{\top}$ & transpose of $A$\\
		$\psi_A(\cdot)$ & the eigenpolynomial of matrix $A$\\
		$\mathscr{G}_{c,A}$ & the contraction-related matrix of $A$ (see Definition \ref{2:3})\\
		$A\succ B$ & $A-B$ is a positive definite matrix \\
		$A\succeq B$ & $A-B$ is a positive semidefinite matrix \\
		$\text{diag}\{A_1,...,A_l\}$ & The generalized diagonal matrix with the sub-block matrices $A_i\ (i\in\mathbb{S}_l^0)$ \\
		$\overline{\mathbf{CM}}(l)$ & $:=\{A\in\mathbb{R}^{l\times l}|\ \{|\lambda|:\psi_A(\lambda)=0\}\in(0,1)\}$\\
		$\mathscr{T}(l)$ & set of all $l$-dimensional \underline{standard $\mathbf{CM}$ matrices} in Definition \ref{2:1}\\
		$\mathcal{U}(l)$ & set of all $l$-dimensional nonsingular upper triangular matrices\\
		$\mathcal{U}_{cm}(l)$ & set of all $l$-dimensional \underline{upper $\mathbf{CM}$--Hessenberg matrices} in Definition \ref{2:2}\\
		$\amalg_{l,j}$ & $:=\text{diag}\{\underbrace{0,0,...,0}_{j-1\ \rm{term}},1,0,...,0\}\in\mathbb{R}^{l\times l}$\\
		\noalign{\smallskip}\hline
		\end{tabular}
		
		\begin{tablenotes}
			\footnotesize
			\item[\textcolor{blue}{$\ddag1$}] $\mathbb{R}^l:=\mathbb{R}^{l\times 1}$; $\mathbb{S}_l^l=\emptyset$; $\mathbb{R}_+^l=(0,\infty)^l:=\{(x_1,...,x_l)^{\top}\in\mathbb{R}^l|x_i>0,\ \forall\ i\in\mathbb{S}_l^0\}$. In particular, let $\mathbb{S}_{\infty}^0:=\{1,2,3,...\}$.
			\item[\textcolor{blue}{$\ddag2$}] If there is no ambiguity in theoretical analysis, $\mathbb{O}_{l,q}$ (or $\mathbb{O}_{l}$) can be simplified as $\mathbb{O}$.  
		\end{tablenotes}	
		\end{threeparttable}
		\end{center}
		\end{small}
	\end{table}
	\begin{definition}\label{2:1}
		The equation $\Im_d(\Sigma,\boldsymbol{\alpha},\aleph)=\mathbb{O}$ is called a standard $\mathbb{L}_0$-algebraic equation if $\boldsymbol{\alpha}\in\mathscr{S}(k)$ and $\aleph=\amalg_{k,1}\ (k=1,2,...)$.
	\end{definition}
	\begin{definition}\label{2:2}
		$C$ is called an $l$-dimensional standard $\mathbf{CM}$ matrix if $C\in\overline{\mathbf{CM}}(l)$ with the form
		\begin{equation*}
		C=\left(\begin{array}{cc}
		-\mathbf{c}^{\langle l-1\rangle} & -c_l\\ 
		\mathbf{I}_{l-1} & \mathbb{O}
		\end{array}\right),
		\end{equation*}
		where $\mathbf{c}=(c_1,c_2,...,c_l)$.
	\end{definition}
	\begin{definition}\label{2:3}
		$C=(c_{ij})_{l\times l}$ is called an $l$-dimensional upper $\mathbf{CM}$--Hessenberg matrix if $C\in\overline{\mathbf{CM}}(l)$ and it satisfies $c_{j+1,j}\neq 0$ and $c_{ij}=0$ for any $j\in\mathbb{S}_{l-1}^0;i\in\mathbb{S}_l^{j+1}$.
	\end{definition}
	\begin{definition}\label{2:4}
		Let $\psi_A(\lambda)=\lambda^{l}+\sum_{i=1}^{l}a_i\lambda^{l-i}$, then $\mathscr{G}_{c,A}$ is called the contraction-related matrix of $A$ if (\ref{2.1}) holds:
		\begin{equation}\label{2.1}
		\begin{cases}
		\mathscr{G}_{c,A}(1,1)=1-\sum\limits_{k=1}^{l}a_k^2,\\ \mathscr{G}_{c,A}(1,j)=-\sum\limits_{k=1}^{l}a_k(a_{k+1-j}+a_{k+j-1}):=\wp_{c,j},\ \ \forall\ j\in\mathbb{S}_l^1,\\
		\mathscr{G}_{c,A}(i,1)=a_{i-1},\ \ \ \mathscr{G}_{c,A}(i,i)=1+a_{2i-2},\ \ \forall\ i\in\mathbb{S}_l^2,\\
		\mathscr{G}_{c,A}(i,j)=a_{i-j}+a_{i+j-2},\ \ \forall\ i,j\in\mathbb{S}_l^1;i\neq j,
		\end{cases}
		\end{equation}
		where $a_j=0$ for any $j\notin\mathbb{S}_l^0$. That is,
		\begin{equation*}
		\mathscr{G}_{c,A}=\left(\begin{array}{ccccccc}
		1-\sum\limits_{k=1}^{l}a_k^2 & \wp_{c,2} & \wp_{c,3} & \wp_{c,4} & \cdots & \wp_{c,l-1} & \wp_{c,l} \\ 
		a_1 & 1+a_2 & a_3 & a_4 & \cdots & a_{l-1} & a_l \\ 
		a_2 & a_1+a_3 & 1+a_4 & a_5 & \cdots & a_l & 0 \\ 
		a_3 & a_2+a_4 & a_1+a_5 & 1+a_6 & \cdots & 0 & 0 \\ 
		\vdots & \vdots & \vdots & \vdots & \begin{sideways}$\ddots$\end{sideways} & \vdots & \vdots \\ 
		a_{l-1} & a_{l-2} & a_{l-3} & a_{l-4} & \cdots & a_1 & 1
		\end{array}\right). 
		\end{equation*}
	\end{definition}
	Below we introduce a simple result for the discrete-type Lyapunov equations (Lemma \ref{2,1}). 
	\begin{lemma}\label{2,1}
		(\cite{S42}) For any given matrix $A\in\mathbb{R}^{l\times l}$, if all the eigenvalues $s_i$ of $A$ satisfy $s_is_j\neq 1$, $\forall\ i,j\in\mathbb{S}_l^0$, then for any $C\in\mathbb{R}^{l\times l}$, there exists a unique symmetric matrix $B$ such that $\Im_d(B,A,C)=\mathbb{O}$.	
	\end{lemma}	   	
	\begin{proposition}\label{2&1}
		Consider an $l$-dimensional standard $\mathbb{L}_0$-algebraic equation
		\begin{equation}\label{2.2}
		\Im_d(\Xi_l,A,\amalg_{l,1})=\mathbb{O},
		\end{equation}
		where $A\in\mathscr{T}(l)$. Then
		\begin{itemize}
		\item[$\rm{(i)}$] $\rm{(Positive\ definiteness)}$ $\Xi_l$ is unique and $\Xi_l\succ\mathbb{O}$.
		\item[$\rm{(ii)}$] $\rm{(Expression)}$ $\Xi_l$ takes the form
		\begin{equation}\label{2.3}
		\Xi_l=\left(\begin{array}{cccccc}
		\zeta_1 & \zeta_2 & \zeta_3 & \zeta_4 & \cdots & \zeta_l \\ 
		\zeta_2 & \zeta_1 & \zeta_2 & \zeta_3 & \cdots & \zeta_{l-1} \\ 
		\zeta_3 & \zeta_2 & \zeta_1 & \zeta_2 & \cdots & \zeta_{l-2} \\ 
		\zeta_4 & \zeta_3 & \zeta_2 & \zeta_1 & \cdots & \zeta_{l-3} \\ 
		\vdots & \vdots & \vdots & \vdots & \begin{sideways}$\ddots$\end{sideways} & \vdots \\ 
		\zeta_l & \zeta_{l-1} & \zeta_{l-2} & \zeta_{l-3} & \cdots & \zeta_1
		\end{array}\right),
		\end{equation}
		where $(\zeta_1,\zeta_2,...,\zeta_l)^{\top}:=\boldsymbol{\zeta}$ is determined by equation $\mathscr{G}_{c,A}\boldsymbol{\zeta}=\mathbf{e}_{l}$.
		\end{itemize}
	\end{proposition}
	\begin{proposition}\label{2&2}
		For any matrix $C\in\mathcal{U}_{cm}(l)$, let $\mathcal{D}^{\top}=((\boldsymbol{\beta}_lC^{l-1})^{\top},(\boldsymbol{\beta}_lC^{l-2})^{\top},...,\boldsymbol{\beta}_l^{\top})$, then $\mathcal{D}\in\mathcal{U}(l)$ and $\mathcal{D}C\mathcal{D}^{-1}\in\mathscr{T}(l)$.
	\end{proposition}
	The proofs of Propositions \ref{2&1} and \ref{2&2} are shown in Appendix A.

	In order for the PNOA (resp., PLNA) approach available in general setting of (\ref{1.2}), we impose the following Assumption \ref{2;1} (resp., \ref{2;2}).
	\begin{assumption}\label{2;1}
		The following conditions are satisfied:	
		\begin{itemize} 
			\item[\rm{(a)}] The deterministic model of (\ref{1.2}) has a unique $\theta$-periodic solution $\mathbf{X}^{\star}(t)$ on $t\ge 0$.
			\item[\rm{(b)}] For any $\epsilon>0$ and $\mathbf{x}_0\in\mathbb{R}^n$, system (\ref{1.2}) has a unique global solution $(\mathbf{X}_{\epsilon}(t))$ which will stay forever in $\mathbb{R}^n$ with probability $1$ (a.s.). 
			\item[\rm{(c)}] The fundamental matrix $\boldsymbol{\Phi}_{C^{\star}}(\theta)\in\overline{\mathbf{CM}}(n)$, where
			$C^{\star}(t)=(\frac{\partial f_i(t,\mathbf{X}^{\star}(t))}{\partial x_j})_{n\times n}$.
		\end{itemize}
	\end{assumption} 
	\begin{assumption}\label{2;2}
		The following conditions hold:	
		\begin{itemize}
			\item[\rm{(1)}] There exists $\epsilon_0>0$ such that for any $\epsilon\in[0,\epsilon_0)$ (including $\epsilon_0=\infty$), system (\ref{2.4}) has a unique $\theta$-periodic solution $\boldsymbol{\Psi}_{\epsilon}^{\star}(t)$ on $t\in[0,\infty)$, where
			\begin{equation}\label{2.4}
			d\Psi_{\epsilon,i}(t)=\Big(e^{-\Psi_{\epsilon,i}(t)}f_i(t,e^{\boldsymbol{\Psi}_{\epsilon}(t)})-\dfrac{\epsilon}{2}\sum_{j=1}^{N}g_{ij}^2(t,e^{\boldsymbol{\Psi}_{\epsilon}(t)})\Big)dt,\ \ \ \forall\ i\in\mathbb{S}_n^0.
			\end{equation} 
			\item[\rm{(2)}] For any $\mathbf{x}_0\in\mathbb{R}_+^n$ and $\epsilon>0$, system (\ref{1.2}) has a unique solution $(\mathbf{X}_{\epsilon}(t))$, and it will remain in $\mathbb{R}_+^n$ a.s.
			\item[\rm{(3)}] There is $\epsilon_1>0$ such that for any $\epsilon\in(0,\epsilon_1)$, the fundamental matrix $\boldsymbol{\Phi}_{D^{\star}}(\theta)\in\overline{\mathbf{CM}}(n)$, where $$D^{\star}(t)=\Big(\frac{\partial F_i(t,\mathbf{x})}{\partial(\ln x_j)}\Big)_{n\times n}\bigl|_{\mathbf{x}=e^{\boldsymbol{\Psi}^{\star}_{\epsilon}(t)}},\ \ \ F_i(t,\mathbf{x})=\frac{f_i(t,\mathbf{x})}{x_i}-\frac{\epsilon}{2}\sum_{j=1}^{N}g_{ij}^2(t,\mathbf{x}).$$
		\end{itemize}
	\end{assumption}
	\begin{remark}\label{11}
		{\rm{Assumption \ref{2;1}(c) (resp., Assumption \ref{2;2}(3)) guarantees that our PNOA (resp., PLNA) approach can work in degenerate diffusion.}}
	\end{remark}
\section{Periodic normal approximation (PNOA)}
This section aims at providing the PNOA approach for locally characterizing the SPSD of (\ref{1.2}). To start, we define 
$$\xi=\text{rank}\Big(\int_{0}^{\theta}(\boldsymbol{\Phi}_{C^{\star}}(\theta)\boldsymbol{\Phi}^{-1}_{C^{\star}}(t)\Gamma^{\star}(t))(\boldsymbol{\Phi}_{C^{\star}}(\theta)\boldsymbol{\Phi}^{-1}_{C^{\star}}(t)\Gamma^{\star}(t))^{\top}dt\Big),$$
where $\Gamma^{\star}(t)=G_c(t,\mathbf{X}^{\star}(t))$.

Let $\lambda_k^{+}\ (k\in\mathbb{S}_{\xi}^0)$ be all the nonzero eigenvalues of $\int_{0}^{\theta}(\boldsymbol{\Phi}_{C^{\star}}(\theta)\boldsymbol{\Phi}^{-1}_{C^{\star}}(t)\Gamma^{\star}(t))(\boldsymbol{\Phi}_{C^{\star}}(\theta)\boldsymbol{\Phi}^{-1}_{C^{\star}}(t)\Gamma^{\star}(t))^{\top}dt$. In view of $\int_{0}^{\theta}(\boldsymbol{\Phi}^{-1}_{C^{\star}}(t)\Gamma^{\star}(t))(\boldsymbol{\Phi}^{-1}_{C^{\star}}(t)\Gamma^{\star}(t))^{\top}dt\succeq\mathbb{O}$, one has $\lambda_k^{+}>0$, $\forall\ k\in\mathbb{S}_{\xi}^0$. As a result, there is an orthogonal matrix $\mathcal{G}_0$ such that
\begin{equation}\label{3.1}
\mathcal{G}_0\bigg[\int_{0}^{\theta}\big(\boldsymbol{\Phi}_{C^{\star}}(\theta)\boldsymbol{\Phi}^{-1}_{C^{\star}}(t)\Gamma^{\star}(t)\big)\big(\boldsymbol{\Phi}_{C^{\star}}(\theta)\boldsymbol{\Phi}^{-1}_{C^{\star}}(t)\Gamma^{\star}(t)\big)^{\top}dt\bigg]\mathcal{G}_0^{\top}=\sum_{k=1}^{\xi}\lambda_k^{+}\amalg_{n,\phi_k},
\end{equation}
where $\phi_i<\phi_j$, $\forall\ i<j$.

By the definition of $\mathbf{X}^{\star}(t)$, we have
	\begin{equation}\label{3.2}
	d(X_{\epsilon,i}(t)-X_i^{\star}(t))=\bigl(f_i(t,\mathbf{X}_{\epsilon}(t))-f_i(t,\mathbf{X}^{\star}(t))\bigr)dt+\sqrt{\epsilon}X_{\epsilon,i}(t)\sum\limits_{j=1}^{N}g_{ij}(t,\mathbf{X}_{\epsilon}(t))dW_j(t),\ \ \ \ i\in\mathbb{S}_n^0.
	\end{equation}
	Combining Taylor expansion, the linearized equations of (\ref{3.2}) near $\mathbf{X}^{\star}(t)$ is
	\begin{equation}\label{3.3}
	\begin{cases}
	d\mathbf{Z}_{\epsilon}(t)=C^{\star}(t)\mathbf{Z}_{\epsilon}(t)dt+\sqrt{\epsilon}\Gamma^{\star}(t)d\mathbf{W}(t),\\
	\mathbf{Z}_{\epsilon}(0)=\mathbf{x}_0-\mathbf{X}^{\star}(0).
	\end{cases}
	\end{equation}
	For convenience, let $A_{[1]}=\mathcal{G}_0\boldsymbol{\Phi}_{C^{\star}}(\theta)\mathcal{G}_0^{-1}$ and $\boldsymbol{\phi}=\{\phi_1,...,\phi_{\xi}\}$.
	\begin{theorem}\label{3-1}
		Under Assumption \ref{2;1}, for sufficiently small $\epsilon>0$, system (\ref{1.2}) approximately has a local SPS which follows the distribution $\mathbb{N}_n(\mathbf{X}^{\star}(t),\Sigma^c_{\epsilon}(t))$, where
		$$\Sigma^c_{\epsilon}(t)=\boldsymbol{\Phi}_{C^{\star}}(t)\Big[\Sigma^c_{\epsilon}(0)+\epsilon\int_{0}^{t}\big(\boldsymbol{\Phi}^{-1}_{C^{\star}}(s)\Gamma^{\star}(s)\big)\big(\boldsymbol{\Phi}^{-1}_{C^{\star}}(s)\Gamma^{\star}(s)\big)^{\top}ds\Big]\boldsymbol{\Phi}_{C^{\star}}^{\top}(t),$$
		and
		\begin{equation}\label{3.4}
		\Sigma^c_{\epsilon}(0)=\epsilon \mathcal{G}_0^{\top}\Bigl(\sum_{k=1}^{\xi}\lambda_k^{+}\Sigma_{[1]\phi_k}\Bigr)\mathcal{G}_0,
		\end{equation}
		with $\Sigma_{[1]\phi_k}$ obtained by Algorithm \ref{1}. In addition, for any constant vector $\mathbf{X}=(X_1,...,X_n)^{\top}\in\mathbb{R}^n$, there holds:
		\begin{equation}\label{3.5}
		\mathbf{X}^{\top}\Sigma^c_{\epsilon}(0)\mathbf{X}\ge\rho_{\epsilon}\sum_{k=1}^{\xi}\Bigl(Z_{\phi_k}^2+\sum_{j=2}^{\eta_k}(\mathbf{H}_{[1]\phi_k,j}^{(j)})^2\Bigr),
		\end{equation}
		where $\rho_{\epsilon}>0$ is defined in (\ref{3.33}), $\mathbf{Z}=\mathcal{G}_0\mathbf{X}:=(Z_1,...,Z_n)^{\top}$, and $$\mathbf{H}_{[1]\phi_k,j}=\Big(\prod_{i=0}^{j-1}(Q_{[1]\phi_k,i}^{-1})^{\top}P_{[1]\phi_k,i}\Big)J_{[1]\phi_k}\mathbf{Z},$$
		with $\eta_k$, $J_{[1]\phi_k}$, $P_{[1]\phi_k,i}$ and $Q_{[1]\phi_k,i}$ determined in Algorithm \ref{1}.
		\\
		\begin{threeparttable}
		\begin{algorithm}[H]\label{1}
			\caption{Algorithm for obtaining $\Sigma_{[1]\phi_k}\ (k\in\mathbb{S}_{\xi}^0)$}
			\LinesNumbered
			\KwIn{\rm{$A_{[1]}$, $\boldsymbol{\phi}$.}}
			\KwOut{\rm{$\eta_k$, $\Sigma_{[1]\phi_k}=(\prod_{i=0}^{\eta_k-1}\overline{a}_{\nu_{k(i)},i}^{\lfloor[1]\phi_k,i\rfloor})^2\times$ $[M_{[1]\phi_k,\eta_k}(\prod_{i=0}^{\eta_k-1}Q_{[1]\phi_k,i}P_{[1]\phi_k,i})J_{[1]\phi_k}]^{-1}\Delta_{[1]\phi_k,\eta_k} \{[M_{[1]\phi_k,\eta_k}(\prod_{i=0}^{\eta_k-1}Q_{[1]\phi_k,i}P_{[1]\phi_k,i})J_{[1]\phi_k}]^{-1}\}^{\top}$\tnote{\textcolor{blue}{$\ddag1$}}.}
			}		
			\rm{\textbf{(Initialization):} $\eta_k:=1$\;}
			\rm{\textbf{(Order transformation):}} $\overline{A}_{[1]\phi_k,1}:=J_{[1]\phi_k}A_{[1]}J_{[1]\phi_k}^{-1}$\;
			\For{$i=1:n-1$}{
			\eIf{$\sum_{j=i+1}^{n}(\overline{a}_{ji}^{\lfloor[1]\phi_k,i\rfloor})^2=0$}
			{
			$\eta_k=i$\;
			\rm{\textbf{break}}\;
			}
			{
			Choose a ``\underline{\textbf{suitable}}\tnote{\textcolor{blue}{$\ddag2$}}" $\nu_{k(i)}\in\mathbb{S}_n^i$ such that $\overline{a}_{[1]\nu_{k(i)},i}^{\lfloor\phi_k,i\rfloor}\neq 0$\;
			\rm{\textbf{(Rotation transformation):}} $\widehat{A}_{[1]\phi_k,i}:=P_{[1]\phi_k,i}\overline{A}_{[1]\phi_k,i}P_{[1]\phi_k,i}^{-1}$\;
			\rm{\textbf{(Elimination transformation):}} $\overline{A}_{[1]\phi_k,i+1}:=Q_{[1]\phi_k,i}\widehat{A}_{[1]\phi_k,i}Q_{[1]\phi_k,i}^{-1}$;
			}
			$\eta_k$++;
			}
			\rm{\textbf{(Standardized transformation):}} $A_{[1]s,\phi_k}:=M_{[1]\phi_k,\eta_k}\overline{A}_{[1]\phi_k,\eta_k}M_{[1]\phi_k,\eta_k}^{-1}$\;
			Get a standard $\mathbb{L}_0$-algebraic equation $\Im_d(\Xi_{[1]\phi_k,\eta_k},A_{[1]s,\phi_k}^{(\eta_k)},\amalg_{\eta_k,1})=\mathbb{O}$\tnote{\textcolor{blue}{$\ddag1$}}\;
			\rm{\textbf{return}} $\eta_k,\ \Xi_{[1]\phi_k,\eta_k},\ \Sigma_{[1]\phi_k}.$
		\end{algorithm}
		\begin{tablenotes}
			\footnotesize
			\item[\rm{\textcolor{blue}{$\ddag1$}}] \rm{$\overline{a}_{ji}^{\lfloor\cdot\rfloor}$ (or $\overline{a}_{j,i}^{\lfloor\cdot\rfloor}$) represents the $i$th element of the $j$th row of $\overline{A}_{(\cdot)}$, and the paraphrase of $\widehat{a}_{j,i}^{\lfloor\cdot\rfloor}$ is the same as $\overline{a}_{j,i}^{\lfloor\cdot\rfloor}$. In particular, $\overline{a}_{\nu_{k(0)},0}^{\lfloor[1]\phi_k,0\rfloor}=1$ and $P_{[1]\phi_k,j}=Q_{[1]\phi_k,j}=\mathbf{I}_n$, where $j\in\{0,n-1\}$. Moreover,
			\begin{equation*}
			J_{[1]\phi_k}=\left(\begin{array}{ccc}
			\mathbb{O} & 1 & \mathbb{O} \\ 
			\mathbf{I}_{\phi_k-1} & \mathbb{O} & \mathbb{O} \\ 
			\mathbb{O} & \mathbb{O} & \mathbf{I}_{n-\phi_k}
			\end{array}\right),\ \ M_{[1]\phi_k,\eta_k}=\left(\begin{array}{cc}
			\mathcal{M}_{[1]\eta_k} & \mathbb{O} \\ 
			\mathbb{O} & \mathbf{I}_{n-\eta_k}
			\end{array}\right),\ \  
			\mathcal{M}_{[1]\eta_k}=\left(\begin{array}{c}
			\boldsymbol{\beta}_{\eta_k}(\overline{A}_{[1]\phi_k,\eta_k}^{(\eta_k)})^{\eta_k-1}\\ 
			\boldsymbol{\beta}_{\eta_k}(\overline{A}_{[1]\phi_k,\eta_k}^{(\eta_k)})^{\eta_k-2}\\ 
			\cdots\\ 
			\boldsymbol{\beta}_{\eta_k}\end{array}\right),   
			\end{equation*}
			\begin{equation*}
			\Delta_{[1]\phi_k,\eta_k}=\left(\begin{array}{cc}
			\Xi_{[1]\phi_k,\eta_k} & \mathbb{O} \\ 
			\mathbb{O} & \mathbb{O} 
			\end{array}\right),\ \ P_{[1]\phi_k,i}=\left(\begin{array}{ccc}
			\mathbf{I}_i & \mathbb{O} & \mathbb{O} \\ 
			\mathbb{O} & \mathbb{O} & \mathbf{I}_{n+1-\nu_{k(i)}} \\ 
			\mathbb{O} & \mathbf{I}_{\nu_{k(i)}-1-i} & \mathbb{O}
			\end{array}\right),\ \ Q_{[1]\phi_k,i}=\left(\begin{array}{ccc}
			\mathbf{I}_i & \mathbb{O} & \mathbb{O} \\ 
			\mathbb{O} & 1 & \mathbb{O} \\ 
			\mathbb{O} & \boldsymbol{\ell}_{[1]k,n-1-i} & \mathbf{I}_{n-1-i}
			\end{array}\right),   
			\end{equation*}
			where $\boldsymbol{\ell}_{[1]k,n-1-i}=\frac{-1}{\widehat{a}_{i+1,i}^{\lfloor[1]\phi_k,i\rfloor}}(\widehat{a}_{i+2,i}^{\lfloor[1]\phi_k,i\rfloor},...,\widehat{a}_{n,i}^{\lfloor[1]\phi_k,i\rfloor})^{\top}$, and $\Xi_{[1]\phi_k,\eta_k}$ is shown in (\ref{3.21}).}
			\item[\rm{\textcolor{blue}{$\ddag2$}}] \rm{The choice of $\nu_{k(i)}$ should be helpful to prove $\Sigma^c_{\epsilon}(0)\succ\mathbb{O}$. The specific criteria of such choice can refer to \cite[Remark 9, Rules 1-3]{S39}.}
		\end{tablenotes}
	\end{threeparttable}
	\end{theorem}
	\begin{proof}
		By the definition of $\boldsymbol{\Phi}_{C^{\star}}(\cdot)$, we have
		\begin{equation*}
		\frac{d}{dt}\big(\boldsymbol{\Phi}_{C^{\star}}(t)\big)=C^{\star}(t)\boldsymbol{\Phi}_{C^{\star}}(t),\ \ \ \boldsymbol{\Phi}_{C^{\star}}(0)=\mathbf{I}_n.
		\end{equation*}
		Then the solution of (\ref{3.3}) can be explicitly written as
		\begin{equation}\label{3.6}
		\mathbf{Z}_{\epsilon}(t)=\boldsymbol{\Phi}_{C^{\star}}(t)\mathbf{Z}_{\epsilon}(0)+\sqrt{\epsilon}\int_{0}^{t}\boldsymbol{\Phi}_{C^{\star}}(t)\boldsymbol{\Phi}_{C^{\star}}^{-1}(s)\Gamma^{\star}(s)d\mathbf{W}(s).
		\end{equation}
		To proceed, we let 
		$$\boldsymbol{\varphi}(t)=\mathbb{E}\mathbf{Z}_{\epsilon}(t),\ \ \ \Sigma^c_{\epsilon}(t)=\mathbb{E}\big(\mathbf{Z}_{\epsilon}(t)\mathbf{Z}_{\epsilon}^{\top}(t)\big)-\boldsymbol{\varphi}(t)\boldsymbol{\varphi}^{\top}(t).$$
		Direct calculation shows that
		\begin{equation}\label{3.7}
		\begin{cases}
		\boldsymbol{\varphi}(t)=\boldsymbol{\Phi}_{C^{\star}}(t)\boldsymbol{\varphi}(0),\\
		\Sigma^c_{\epsilon}(t)=\boldsymbol{\Phi}_{C^{\star}}(t)\Big[\Sigma^c_{\epsilon}(0)+\epsilon\int_{0}^{t}\big(\boldsymbol{\Phi}^{-1}_{C^{\star}}(s)\Gamma^{\star}(s)\big)\big(\boldsymbol{\Phi}^{-1}_{C^{\star}}(s)\Gamma^{\star}(s)\big)^{\top}ds\Big]\boldsymbol{\Phi}_{C^{\star}}^{\top}(t).
		\end{cases}
		\end{equation}
		Since $\int_{0}^{t}\boldsymbol{\Phi}_{C^{\star}}(t)\boldsymbol{\Phi}_{C^{\star}}^{-1}(s)\Gamma^{\star}(s)d\mathbf{W}(s)$ is a local martingale, we determine that the solution process $\mathbf{Z}_{\epsilon}(t)$ follows a Gaussian distribution $\mathbb{N}_{n}(\boldsymbol{\varphi}(t),\Sigma^c_{\epsilon}(t))$ at any time $t$.

		To find the periodic solution of (\ref{3.3}), it is equivalent to ensuring the solution of Eq. (\ref{3.7}) is $\theta$-periodic, which means 
		\begin{equation}\label{3.8}
		\text{(i)}\ \boldsymbol{\varphi}(t+\theta)=\boldsymbol{\varphi}(t),\ \ \ \text{(ii)}\ \Sigma^c_{\epsilon}(t+\theta)=\Sigma^c_{\epsilon}(t),\ \ \ \forall\ t\ge 0.
		\end{equation}
		To obtain (i), the result $\boldsymbol{\varphi}(\theta)=\boldsymbol{\varphi}(0)$ is required. Using Assumption \ref{2;1}(c), we have $\boldsymbol{\varphi}(0)=\mathbf{0}$. Thus, $\boldsymbol{\varphi}(t)=\mathbf{0}$ for any $t\ge 0$.

		Next we expect to find a suitable $\Sigma^c_{\epsilon}(0)$ such that $\Sigma^c_{\epsilon}(t)$ is also $\theta$-periodic. By letting $\Sigma^c_{\epsilon}(\theta)=\Sigma^c_{\epsilon}(0)$, we get that $\Sigma^c_{\epsilon}(0)$ satisfies the Lyapunov equation
		\begin{equation}\label{3.9}
		\Im_d\bigg(\Sigma^c_{\epsilon}(0),\boldsymbol{\Phi}_{C^{\star}}(\theta),\epsilon\int_{0}^{\theta}\big(\boldsymbol{\Phi}_{C^{\star}}(\theta)\boldsymbol{\Phi}^{-1}_{C^{\star}}(t)\Gamma^{\star}(t)\big)\big(\boldsymbol{\Phi}_{C^{\star}}(\theta)\boldsymbol{\Phi}^{-1}_{C^{\star}}(t)\Gamma^{\star}(t)\big)^{\top}dt\bigg)=\mathbb{O}.
		\end{equation}
		Below we prove condition (ii) in (\ref{3.8}) if Eq. (\ref{3.9}) holds. Mainly, using the basic properties of $\boldsymbol{\Phi}_{C^{\star}}(\cdot)$, one has
		\begin{equation*}
		\boldsymbol{\Phi}_{C^{\star}}(t+\theta)=\boldsymbol{\Phi}_{C^{\star}}(t)\boldsymbol{\Phi}_{C^{\star}}(\theta),\ \ \ \forall\ t\ge 0.
		\end{equation*}
		This together with (\ref{3.7}) and (\ref{3.9}) yields
		\begin{align}\label{3.10}
		\Sigma^c_{\epsilon}(t+\theta)=&\boldsymbol{\Phi}_{C^{\star}}(t+\theta)\bigg[\Sigma_{\epsilon}(0)+\epsilon\int_{0}^{t+\theta}\big(\boldsymbol{\Phi}^{-1}_{C^{\star}}(s)\Gamma^{\star}(s)\big)\big(\boldsymbol{\Phi}^{-1}_{C^{\star}}(s)\Gamma^{\star}(s)\big)^{\top}ds\bigg]\boldsymbol{\Phi}_{C^{\star}}^{\top}(t+\theta)\notag\\
		=&\boldsymbol{\Phi}_{C^{\star}}(t)\bigg[\boldsymbol{\Phi}_{C^{\star}}(\theta)\Sigma_{\epsilon}(0)\boldsymbol{\Phi}_{C^{\star}}^{\top}(\theta)+\epsilon\int_{0}^{t+\theta}\big(\boldsymbol{\Phi}_{C^{\star}}(\theta)\boldsymbol{\Phi}^{-1}_{C^{\star}}(s)\Gamma^{\star}(s)\big)\big(\boldsymbol{\Phi}_{C^{\star}}(\theta)\boldsymbol{\Phi}^{-1}_{C^{\star}}(s)\Gamma^{\star}(s)\big)^{\top}ds\bigg]\boldsymbol{\Phi}_{C^{\star}}^{\top}(t)\notag\\
		=&\boldsymbol{\Phi}_{C^{\star}}(t)\bigg[\Sigma^c_{\epsilon}(0)-\epsilon\int_{0}^{\theta}\big(\boldsymbol{\Phi}_{C^{\star}}(\theta)\boldsymbol{\Phi}^{-1}_{C^{\star}}(t)\Gamma^{\star}(t)\big)\big(\boldsymbol{\Phi}_{C^{\star}}(\theta)\boldsymbol{\Phi}^{-1}_{C^{\star}}(t)\Gamma^{\star}(t)\big)^{\top}dt\notag\\
		&+\epsilon\int_{0}^{t+\theta}\big(\boldsymbol{\Phi}_{C^{\star}}(\theta)\boldsymbol{\Phi}^{-1}_{C^{\star}}(s)\Gamma^{\star}(s)\big)\big(\boldsymbol{\Phi}_{C^{\star}}(\theta)\boldsymbol{\Phi}^{-1}_{C^{\star}}(s)\Gamma^{\star}(s)\big)^{\top}ds\bigg]\boldsymbol{\Phi}_{C^{\star}}^{\top}(t)\notag\\
		=&\boldsymbol{\Phi}_{C^{\star}}(t)\bigg[\Sigma^c_{\epsilon}(0)+\epsilon\boldsymbol{\Phi}_{C^{\star}}(\theta)\bigg(\int_{\theta}^{t+\theta}\big(\boldsymbol{\Phi}^{-1}_{C^{\star}}(s)\Gamma^{\star}(s)\big)\big(\boldsymbol{\Phi}^{-1}_{C^{\star}}(s)\Gamma^{\star}(s)\big)^{\top}ds\bigg)\boldsymbol{\Phi}_{C^{\star}}^{\top}(\theta)\bigg]\boldsymbol{\Phi}_{C^{\star}}^{\top}(t)\notag\\
		=&\boldsymbol{\Phi}_{C^{\star}}(t)\bigg[\Sigma^c_{\epsilon}(0)+\epsilon\boldsymbol{\Phi}_{C^{\star}}(\theta)\bigg(\int_{0}^{t}\big(\boldsymbol{\Phi}^{-1}_{C^{\star}}(s+\theta)\Gamma^{\star}(s)\big)\big(\boldsymbol{\Phi}^{-1}_{C^{\star}}(s+\theta)\Gamma^{\star}(s)\big)^{\top}ds\bigg)\boldsymbol{\Phi}_{C^{\star}}^{\top}(\theta)\bigg]\boldsymbol{\Phi}_{C^{\star}}^{\top}(t)\notag\\
		=&\Sigma^c_{\epsilon}(t),\ \ \ \forall\ t\ge 0.
		\end{align}
		In another respect, under Assumption \ref{2;1}(c), it follows from Lemma \ref{2,1} that the solution of Eq. (\ref{3.9}) exists and is unique.

		To summarize, we determine that if $\Sigma^c_{\epsilon}(0)$ satisfies Eq. (\ref{3.9}), then the assertion (\ref{3.8}) holds, and system (\ref{3.3}) has a unique SPS, which admits the distribution $\mathbb{N}_{n}(\mathbf{0},\Sigma^c_{\epsilon}(t))$. According to the relationship between (\ref{3.3}) around the solution $\mathbf{X}^{\star}(t)$ and (\ref{1.2}), we conclude that $\mathbf{X}_{\epsilon}(t)-\mathbf{X}^{\star}(t)$ around $\mathbf{0}$ can be approximated by $\mathbf{Z}_{\epsilon}(t)$, i.e., system (\ref{1.2}) approximately admits a local periodic solution with the state distribution $\mathbb{N}_{n}(\mathbf{X}^{\star}(t),\Sigma^c_{\epsilon}(t))$, where
		$$\Sigma^c_{\epsilon}(t)=\boldsymbol{\Phi}_{C^{\star}}(t)\Big[\Sigma^c_{\epsilon}(0)+\epsilon\int_{0}^{t}\big(\boldsymbol{\Phi}^{-1}_{C^{\star}}(s)\Gamma^{\star}(s)\big)\big(\boldsymbol{\Phi}^{-1}_{C^{\star}}(s)\Gamma^{\star}(s)\big)^{\top}ds\Big]\boldsymbol{\Phi}_{C^{\star}}^{\top}(t).$$
		Hence, the first part of Theorem \ref{3-1} is verified.

		In what follows, we aim at obtaining the specific form and positive definiteness of $\Sigma^c_{\epsilon}(0)$. By virtue of (\ref{3.1}) and the definition of $A_{[1]}$, Eq. (\ref{3.9}) is equivalent to  
		\begin{equation}\label{3.11}
		\Im_d\Bigl(\frac{1}{\epsilon}\mathcal{G}_0\Sigma^c_{\epsilon}(0)\mathcal{G}_0^{\top},A_{[1]},\sum_{k=1}^{\xi}\lambda_k^{+}\amalg_{n,\phi_k}\Bigr)=\mathbb{O}.
		\end{equation}
		Let $\Sigma_{[1]\phi_k}\ (k\in\mathbb{S}_{\xi}^0)$ be the solutions of the following auxiliary equations, respectively:
		\begin{equation}\label{3.12}
		\Im_d(\Sigma_{[1]\phi_k},A_{[1]},\amalg_{n,\phi_k})=\mathbb{O},
		\end{equation}
		Using (\ref{3.11}), (\ref{3.12}) and the superposition principle, then $$\mathcal{G}_0\Sigma^c_{\epsilon}(0)\mathcal{G}_0^{\top}=\epsilon\sum_{k=1}^{\xi}\lambda_k^{+}\Sigma_{[1]\phi_k}.$$
		The desired result (\ref{3.4}) follows from the orthogonality of $\mathcal{G}_0$.

		We prove (\ref{3.5}) by the following two steps. First, we use the Gaussian-like elimination method (see \cite[Proof of Theorem 3.1]{S39}) to solve Eq. (\ref{3.12}) (i.e., the expression of $\Sigma_{[1]\phi_k}$ in Algorithm \ref{1}). Second, we combine the basic properties of $\Xi_{\phi_k,\eta_k}$ and an important property of $\mathbf{H}_{[1]\phi_k,i}$ to study the lower bound of $\mathbf{X}^{\top}\Sigma^c_{\epsilon}(0)\mathbf{X}$.

		\textbf{Step 1.} By proceeding the procedures 1 and 2 in Algorithm \ref{1}, for any $k\in\mathbb{S}_{\xi}^0$, let $\overline{A}_{[1]\phi_k,1}=J_{[1]\phi_k}A_{[1]}J_{[1]\phi_k}^{-1}$. In view of $(Q_{[1]\phi_k,0}P_{[1]\phi_k,0}J_{[1]\phi_k})\amalg_{n,\phi_k}(Q_{[1]\phi_k,0}P_{[1]\phi_k,0}J_{[1]\phi_k})^{\top}=\amalg_{n,1}$, then (\ref{3.12}) is equivalent to
		\begin{equation}\label{3.13}
		\Im_d\bigl((Q_{[1]\phi_k,0}P_{[1]\phi_k,0}J_{[1]\phi_k})\Sigma_{[1]\phi_k}(Q_{[1]\phi_k,0}P_{[1]\phi_k,0}J_{[1]\phi_k})^{\top},\overline{A}_{[1]\phi_k,1},\amalg_{n,1}\bigr)=\mathbb{O}.
		\end{equation}  
		Combining the Gaussian-like elimination method and the procedures 3-13 of Algorithm \ref{1} yields that $\eta_k\ge 1$, and
		\begin{equation}\label{3.14}
		\text{(i-1)}\ \sum_{j=i+1}^{n}(\overline{a}_{ji}^{\lfloor[1]\phi_k,i\rfloor})^2\neq 0,\ \ \forall\ i\in\mathbb{S}_{\eta_k-1}^0,\ \ \ \ \text{(i-2)}\ \overline{a}_{j,\eta_k}^{\lfloor[1]\phi_k,\eta_k\rfloor}=0,\ \ \forall\ j\in\mathbb{S}_{n}^{\eta_k+1},
		\end{equation}
		where each $\overline{a}_{ji}^{\lfloor [1]\phi_k,i\rfloor}$ is determined by the following iterative scheme:
		\begin{equation}\label{3.15}
		\begin{cases}
		\widehat{A}_{[1]\phi_k,i}:=P_{[1]\phi_k,i}\overline{A}_{[1]\phi_k,i}P_{[1]\phi_k,i}^{-1},\\ \overline{A}_{[1]\phi_k,i+1}:=Q_{[1]\phi_k,i}\widehat{A}_{[1]\phi_k,i}Q_{[1]\phi_k,i}^{-1},\ \ \ \forall\ i\in\mathbb{S}_{\eta_k}^0.
		\end{cases}
		\end{equation}
		Using (\ref{3.14}), (\ref{3.15}), and the forms of $P_{[1]\phi_k,i}$ and $Q_{[1]\phi_k,i}$, we determine
		 \begin{equation}\label{3.16}
		 \text{(i-3)}\ \overline{a}_{i+1,i}^{\lfloor[1]\phi_k,\eta_k\rfloor}=\overline{a}_{\nu_{k(i)},i}^{\lfloor[1]\phi_k,i\rfloor}(\neq 0),\ \ \text{and}\ \ \overline{a}_{j,i}^{\lfloor[1]\phi_k,\eta_k\rfloor}=0,\ \ \forall\ i\in\mathbb{S}_{\eta_k-1}^{-1};j\in\mathbb{S}_n^i.
		 \end{equation}
		 By Definition \ref{2:3}, we have $\overline{A}_{[1]\phi_k,\eta_k}^{(\eta_k)}\in\mathcal{U}_{cm}(\eta_k)$. As in Algorithm \ref{1}, we let
		 \begin{equation*}
		 M_{[1]\phi_k,\eta_k}=\left(\begin{array}{cc}
		 \mathcal{M}_{[1]\eta_k} & \mathbb{O} \\ 
		 \mathbb{O} & \mathbf{I}_{n-\eta_k}
		 \end{array}\right),\ \ A_{[1]s,\phi_k}=M_{[1]\phi_k,\eta_k}\overline{A}_{[1]\phi_k,\eta_k}M_{[1]\phi_k,\eta_k}^{-1}.
		 \end{equation*}
		 This together with the form of $\mathcal{M}_{[1]\eta_k}$ and Proposition \ref{2&2} implies
		 \begin{equation}\label{3.17}
		 A_{[1]s,\phi_k}^{(\eta_k)}=\mathcal{M}_{[1]\eta_k}\overline{A}_{[1]\phi_k,\eta_k}^{(\eta_k)}\mathcal{M}_{[1]\eta_k}^{-1}\in\mathscr{T}(\eta_k).
		 \end{equation}
		 A similar argument in (\ref{a.9})-(\ref{a.11}) coupled with (\ref{3.16}) leads to
		 \begin{align}\label{3.18}
		 M_{[1]\phi_k,\eta_k}\amalg_{n,1}M_{[1]\phi_k,\eta_k}^{\top}=\Big(\bigl(\boldsymbol{\beta}_{\eta_k}(\overline{A}_{[1]\phi_k,\eta_k}^{(\eta_k)})^{\eta_k-1}\bigr)^{(1)}\Big)^2\amalg_{n,1}=\Bigl(\prod_{i=0}^{\eta_k-1}\overline{a}_{\nu_{k(i)},i}^{\lfloor[1]\phi_k,i\rfloor}\Bigr)^2\amalg_{n,1}.
		 \end{align}
		 Denote $$\widetilde{\Sigma}_{[1]\phi_k}=\Bigl[\Bigl(\prod_{i=0}^{\eta_k-1}Q_{[1]\phi_k,i}P_{[1]\phi_k,i}\Bigr)J_{[1]\phi_k}\Bigr]\Sigma_{[1]\phi_k}\Bigl[\Bigl(\prod_{i=0}^{\eta_k-1}Q_{[1]\phi_k,i}P_{[1]\phi_k,i}\Bigr)J_{[1]\phi_k}\Bigr]^{\top}.$$
		 Then by (\ref{3.18}), Eq. (\ref{3.13}) is equivalent to
		 \begin{equation}\label{3.19}
		 \Im_d\biggl(M_{[1]\phi_k,\eta_k}\widetilde{\Sigma}_{[1]\phi_k}M_{[1]\phi_k,\eta_k}^{\top},A_{[1]s,\phi_k},\Bigl(\prod_{i=0}^{\eta_k-1}\overline{a}_{\nu_{k(i)},i}^{\lfloor[1]\phi_k,i\rfloor}\Bigr)^2\amalg_{n,1}\biggr)=\mathbb{O}.
		 \end{equation}
		 In the display above, we have used the fact
		 \begin{equation*}
		 \Bigl[\Bigl(\prod_{i=0}^{\eta_k-1}Q_{[1]\phi_k,i}P_{[1]\phi_k,i}\Bigr)J_{[1]\phi_k}\Bigr]\amalg_{n,1}\Bigl[\Bigl(\prod_{i=0}^{\eta_k-1}Q_{[1]\phi_k,i}P_{[1]\phi_k,i}\Bigr)J_{[1]\phi_k}\Bigr]^{\top}=\amalg_{n,1}.
		 \end{equation*}
		 To proceed, by (\ref{3.17}) and Algorithm \ref{1}, let $\Xi_{[1]\phi_k,\eta_k}$ be the following $\eta_k$-dimensional standard $\mathbb{L}_0$-algebraic equation
		 \begin{equation}\label{3.20}
		 \Im_d\bigl(\Xi_{[1]\phi_k,\eta_k},A_{[1]s,\phi_k}^{(\eta_k)},\amalg_{\eta_k,1}\bigr)=\mathbb{O}.
		 \end{equation} 
		 Using Proposition \ref{2&1}, we have $\Xi_{[1]\phi_k,\eta_k}\succ\mathbb{O}$ and
		 \begin{equation}\label{3.21}
		 \Xi_{[1]\phi_k,\eta_k}=\left(\begin{array}{cccccc}
		 \zeta_1 & \zeta_2 & \zeta_3 & \zeta_4 & \cdots & \zeta_{\eta_k} \\ 
		 \zeta_2 & \zeta_1 & \zeta_2 & \zeta_3 & \cdots & \zeta_{\eta_k-1} \\ 
		 \zeta_3 & \zeta_2 & \zeta_1 & \zeta_2 & \cdots & \zeta_{\eta_k-2} \\ 
		 \zeta_4 & \zeta_3 & \zeta_2 & \zeta_1 & \cdots & \zeta_{\eta_k-2} \\ 
		 \vdots & \vdots & \vdots & \vdots & \begin{sideways}$\ddots$\end{sideways} & \vdots \\ 
		 \zeta_{\eta_k} & \zeta_{\eta_k-1} & \zeta_{\eta_k-2} & \zeta_{\eta_k-3} & \cdots & \zeta_1
		 \end{array}\right),
		 \end{equation}
		 where $(\zeta_1,\zeta_2,...,\zeta_{\eta_k})^{\top}=\mathscr{G}_{c,A_{[1]s,\phi_k}^{(\eta_k)}}^{-1}\mathbf{e}_{l}$.
		 \\
		 Below we aim to study the relationship between $\widetilde{\Sigma}_{[1]\phi_k}$ and $\Xi_{[1]\phi_k,\eta_k}$. To begin, we consider the following two conditions:
		 \begin{equation*}
		 (\mathscr{A}_1)\ \eta_k=n,\ \ \ (\mathscr{A}_2)\ \eta_k\in\mathbb{S}_{n-1}^0.
		 \end{equation*}
		 Case 1. If $(\mathscr{A}_1)$ is satisfied, then $M_{[1]\phi_k,n}=\mathcal{M}_n$ and $A_{[1]s,\phi_k}\in\mathscr{T}(n)$. By virtue of Eq. (\ref{3.19}) and Lemma \ref{2,1}, we determine
		 \begin{equation*}
		 \Bigl(\prod_{i=0}^{n-1}\overline{a}_{\nu_{k(i)},i}^{\lfloor[1]\phi_k,i\rfloor}\Bigr)^{-2}M_{[1]\phi_k,\eta_k}\widetilde{\Sigma}_{[1]\phi_k}M_{[1]\phi_k,\eta_k}^{\top}=\Xi_{[1]\phi_k,n},
		 \end{equation*}
		 which implies
		 \begin{equation}\label{3.22}
		 \begin{split}
		 \Sigma_{[1]\phi_k}=&\Bigl(\prod_{i=0}^{n-1}\overline{a}_{\nu_{k(i)},i}^{\lfloor[1]\phi_k,i\rfloor}\Bigr)^2\Bigl(M_{[1]\phi_k,n}\Bigl(\prod_{i=0}^{n-1}Q_{[1]\phi_k,i}P_{[1]\phi_k,i}\Bigr)J_{[1]\phi_k}\Bigr)^{-1}\\
		 &\times\Xi_{\phi_k,n} \Bigl[\Bigl(M_{[1]\phi_k,n}\Bigl(\prod_{i=0}^{n-1}Q_{[1]\phi_k,i}P_{[1]\phi_k,i}\Bigr)J_{[1]\phi_k}\Bigr)^{-1}\Bigr]^{\top}.
		 \end{split}
		 \end{equation}
		 Case 2. If $(\mathscr{A}_2)$ is satisfied, then $\overline{A}_{[1]\phi_k,\eta_k}$ is of the following form:
		 \begin{equation}\label{3.23}
		 \overline{A}_{[1]\phi_k,\eta_k}=\left(\begin{array}{cc}
		 \overline{A}_{[1]\phi_k,\eta_k}^{(\eta_k)} & \underline{A}_{1,\eta_k} \\ 
		 \mathbb{O} & \underline{A}_{2,\eta_k}
		 \end{array}\right),
		 \end{equation}
		 where $\underline{A}_{2,\eta_k}\in\mathbb{R}^{(n-\eta_k)\times(n-\eta_k)}$. Without loss of generality, we let
		 \begin{equation}\label{3.24}
		 \widetilde{\Sigma}_{[1]\phi_k}:=\left(\begin{array}{cc}
		 \widetilde{\Sigma}_{[1]\phi_k}^{(\eta_k)} & \pounds_1 \\ 
		 \pounds_1^{\top} & \pounds_2
		 \end{array}\right),
		 \end{equation}
		 where $\pounds_2\in\mathbb{R}^{(n-\eta_k)\times(n-\eta_k)}$ is real symmetric.
		 \\
		 Inserting (\ref{3.23}) and (\ref{3.24}) into Eq. (\ref{3.19}) yields
		 \begin{equation}\label{3.25}
		 \begin{cases}
		 \Im_d\biggl(\mathcal{M}_{[1]\eta_k}\widetilde{\Sigma}_{[1]\phi_k}^{(\eta_k)}\mathcal{M}_{[1]\eta_k}^{\top},A_{[1]s,\phi_k}^{(\eta_k)},\Bigl(\prod\limits_{i=0}^{\eta_k-1}\overline{a}_{\nu_{k(i)},i}^{\lfloor[1]\phi_k,i\rfloor}\Bigr)^2\amalg_{\eta_k,1}+\mathcal{M}_{[1]\eta_k}\underline{A}_{1,\eta_k}\big(A_{[1]s,\phi_k}^{(\eta_k)}\mathcal{M}_{[1]\eta_k}\pounds_1\big)^{\top}\\
		 \ \ \ \ \ \ \ \ \ \ \ \ +A_{[1]s,\phi_k}^{(\eta_k)}\mathcal{M}_{[1]\eta_k}\pounds_1(\mathcal{M}_{[1]\eta_k}\underline{A}_{1,\eta_k})^{\top}+\mathcal{M}_{[1]\eta_k}\underline{A}_{1,\eta_k}\pounds_2(\mathcal{M}_{[1]\eta_k}\underline{A}_{1,\eta_k})^{\top}\biggr)=\mathbb{O},\\
		 \mathcal{M}_{[1]\eta_k}\pounds_1-\big(A_{[1]s,\phi_k}^{(\eta_k)}\mathcal{M}_{[1]\eta_k}\pounds_1+\mathcal{M}_{[1]\eta_k}\underline{A}_{1,\eta_k}\pounds_2\big)\underline{A}_{2,\eta_k}^{\top}=\mathbb{O},\\
		 \Im_d\bigl(\pounds_2,\underline{A}_{2,\eta_k},\mathbb{O}\bigr)=\mathbb{O}.
		 \end{cases}
		 \end{equation}
		 Using (\ref{3.23}) and Assumption \ref{2;1}(c), one has $\overline{A}_{[1]\phi_k,\eta_k}^{(\eta_k)}\in\overline{\mathbf{CM}}(\eta_k)$ and $\underline{A}_{2,\eta}\in\overline{\mathbf{CM}}(n-\eta_k)$. Combined with (\ref{a.3}) and the third equation of (\ref{3.25}), we obtain
		 $$\pounds_2=\sum_{k=0}^{\infty}\underline{A}_{2,\eta}^k\mathbb{O}(\underline{A}_{2,\eta}^k)^{\top}=\mathbb{O}.$$
		 Hence, (\ref{3.25}) is simplified as
		 \begin{equation}\label{3.26}
		 \begin{cases}
		 \Im_d\biggl(\mathcal{M}_{[1]\eta_k}\widetilde{\Sigma}_{[1]\phi_k}^{(\eta_k)}\mathcal{M}_{[1]\eta_k}^{\top},A_{[1]s,\phi_k}^{(\eta_k)},\Bigl(\prod\limits_{i=0}^{\eta_k-1}\overline{a}_{\nu_{k(i)},i}^{\lfloor[1]\phi_k,i\rfloor}\Bigr)^2\amalg_{\eta_k,1}\\
		 \ \ \ \ \ \ \ \ \ \ \ \ +\mathcal{M}_{[1]\eta_k}\underline{A}_{1,\eta_k}\big(A_{[1]s,\phi_k}^{(\eta_k)}\mathcal{M}_{[1]\eta_k}\pounds_1\big)^{\top}+A_{[1]s,\phi_k}^{(\eta_k)}\mathcal{M}_{[1]\eta_k}\pounds_1(\mathcal{M}_{[1]\eta_k}\underline{A}_{1,\eta_k})^{\top}\biggr)=\mathbb{O},\\
		 \mathcal{M}_{[1]\eta_k}\pounds_1-A_{[1]s,\phi_k}^{(\eta_k)}\mathcal{M}_{[1]\eta_k}\pounds_1\underline{A}_{2,\eta_k}^{\top}=\mathbb{O}.
		 \end{cases}
		 \end{equation}
		 By Lemma \ref{2,1} and (\ref{3.12}), then $\widetilde{\Sigma}_{[1]\phi_k}$ is unique, which implies that the solution $(\widetilde{\Sigma}_{[1]\phi_k}^{(\eta_k)},\pounds_1)$ of Eq. (\ref{3.26}) is also unique.

		 Let $\digamma=A_{[1]s,\phi_k}^{(\eta_k)}\mathcal{M}_{[1]\eta_k}\pounds_1$, given the second equation of (\ref{3.26}), we have
		 \begin{equation*}
		 \digamma-A_{[1]s,\phi_k}^{(\eta_k)}\digamma\underline{A}_{2,\eta_k}^{\top}=\mathbb{O}.
		 \end{equation*}
		 Next we verify $\digamma=\mathcal{M}_{[1]\eta_k}\pounds_1$ by using a contradiction argument. Suppose that $\digamma\neq\mathcal{M}_{[1]\eta_k}\pounds_1$, we can construct the following auxiliary Lyapunov equation:
		 \begin{equation*}
		 \begin{split}
		 \Im_d\biggl(\mathcal{M}_{[1]\eta_k}&\widehat{\Sigma}_{\phi_k}^{[\digamma]}\mathcal{M}_{[1]\eta_k}^{\top},A_{[1]s,\phi_k}^{(\eta_k)},\Bigl(\prod\limits_{i=0}^{\eta_k-1}\overline{a}_{\nu_{k(i)},i}^{\lfloor[1]\phi_k,i\rfloor}\Bigr)^2\amalg_{\eta_k,1}\\
		 &+\mathcal{M}_{[1]\eta_k}\underline{A}_{1,\eta_k}\big(A_{[1]s,\phi_k}^{(\eta_k)}\digamma\big)^{\top}+A_{[1]s,\phi_k}^{(\eta_k)}\digamma(\mathcal{M}_{[1]\eta_k}\underline{A}_{1,\eta_k})^{\top}\biggr)=\mathbb{O}.
		 \end{split}
		 \end{equation*}  
		 By (\ref{3.17}) and Lemma \ref{2,1}, $\widehat{\Sigma}_{\phi_k}^{[\digamma]}$ exists and is unique. By a similar argument in (\ref{a.9})-(\ref{a.11}), one gets
		 \begin{equation*}
		 |\mathcal{M}_{[1]\eta_k}|=\prod_{i=1}^{\eta_k-1}\bigl(\boldsymbol{\beta}_{\eta_k}(\overline{A}_{[1]\phi_k,\eta_k}^{(\eta_k)})^{\eta_k-i}\bigr)^{(i)}=\prod_{i=0}^{\eta_k-1}\big(\overline{a}_{\nu_{k(i)},i}^{\lfloor[1]\phi_k,i\rfloor}\big)^i\neq 0.
		 \end{equation*}
		 This leads to a contradiction that there are two different solutions $(\widehat{\Sigma}_{\phi_k}^{[\digamma]},\mathcal{M}_{[1]\eta_k}^{-1}\digamma)$ and $(\widetilde{\Sigma}_{[1]\phi_k}^{(\eta_k)},\pounds_1)$ satisfying Eq. (\ref{3.26}). Thus, $\pounds_1=\mathcal{M}_{[1]\eta_k}^{-1}\digamma$, which yields
		 \begin{equation}\label{3.27}
		 \big(A_{[1]s,\phi_k}^{(\eta_k)}-\mathbf{I}_{\eta_k}\big)\digamma=\mathbb{O}.
		 \end{equation}
		 Denote by $\{\lambda_i\}_{1\le i\le\eta_k}$ the eigenvalues of $A_{[1]s,\phi_k}^{(\eta_k)}$, we find
		 $$\bigl|A_{[1]s,\phi_k}^{(\eta_k)}-\mathbf{I}_{\eta_k}\bigr|=(-1)^{\eta_k}\psi_{A_{[1]s,\phi_k}^{(\eta_k)}}(1)=(-1)^{\eta_k}\sum_{i=1}^{\eta_k}(1-\lambda_i)\neq 0.$$
		 Combining (\ref{3.27}), we have $\digamma=\mathbb{O}$ and $\pounds_1=\mathbb{O}$. Then Eq. (\ref{3.19}) is equivalent to 
		 \begin{equation}\label{3.28}
		 \Im_d\biggl(\Bigl(\prod\limits_{i=0}^{\eta_k-1}\overline{a}_{\nu_{k(i)},i}^{\lfloor[1]\phi_k,i\rfloor}\Bigr)^{-2}\mathcal{M}_{[1]\eta_k}\widetilde{\Sigma}_{[1]\phi_k}^{(\eta_k)}\mathcal{M}_{[1]\eta_k}^{\top},A_{[1]s,\phi_k}^{(\eta_k)},\amalg_{\eta_k,1}\biggr)=\mathbb{O},
		 \end{equation}
		 with
		 $$\widetilde{\Sigma}_{[1]\phi_k}=\left(\begin{array}{cc}
		 \widetilde{\Sigma}_{[1]\phi_k}^{(\eta_k)} & \mathbb{O} \\ 
		 \mathbb{O} & \mathbb{O}
		 \end{array}\right).$$
		 Applying (\ref{3.20}) and Proposition \ref{2&1} to Eq. (\ref{3.28}) leads to $$\Bigl(\prod_{i=0}^{\eta_k-1}\overline{a}_{\nu_{k(i)},i}^{\lfloor[1]\phi_k,i\rfloor}\Bigr)^{-2}\mathcal{M}_{[1]\eta_k}\widetilde{\Sigma}_{[1]\phi_k}^{(\eta_k)}\mathcal{M}_{[1]\eta_k}^{\top}=\Xi_{[1]\phi_k,\eta_k}\succ\mathbb{O}.$$ 
		 Then
		 \begin{align}\label{3.29}
		 \Sigma_{[1]\phi_k}=&\Big(M_{[1]\phi_k,\eta_k}\Bigl(\prod_{i=0}^{\eta_k-1}Q_{[1]\phi_k,i}P_{[1]\phi_k,i}\Bigr)J_{[1]\phi_k}\Big)^{-1}\left(\begin{array}{cc}
		 \mathcal{M}_{[1]\eta_k}\widetilde{\Sigma}_{[1]\phi_k}^{(\eta_k)}\mathcal{M}_{[1]\eta_k}^{\top} & \mathbb{O} \\ 
		 \mathbb{O} & \mathbb{O}
		 \end{array}\right)\notag\\
		 &\times\Big[\Big(M_{[1]\phi_k,\eta_k}\Bigl(\prod_{i=0}^{\eta_k-1}Q_{[1]\phi_k,i}P_{[1]\phi_k,i}\Bigr)J_{[1]\phi_k}\Big)^{-1}\Big]^{\top}\notag\\
		 =&\Bigl(\prod_{i=0}^{\eta_k-1}\overline{a}_{\nu_{k(i)},i}^{\lfloor[1]\phi_k,i\rfloor}\Bigr)^2\Big(M_{[1]\phi_k,\eta_k}\Bigl(\prod_{i=0}^{\eta_k-1}Q_{[1]\phi_k,i}P_{[1]\phi_k,i}\Bigr)J_{[1]\phi_k}\Big)^{-1}\notag\\
		 &\times\Delta_{[1]\phi_k,\eta_k} \Big[\Big(M_{[1]\phi_k,\eta_k}\Bigl(\prod_{i=0}^{\eta_k-1}Q_{[1]\phi_k,i}P_{[1]\phi_k,i}\Bigr)J_{[1]\phi_k}\Big)^{-1}\Big]^{\top},
		 \end{align}   
		 where $\Delta_{[1]\phi_k,\eta_k}$ is the same as in Algorithm \ref{1}.

		 By (\ref{3.22}), (\ref{3.29}) and $\Delta_{[1]\phi_k,n}=\Xi_{[1]\phi_k,n}$, the explicit form of $\Sigma_{[1]\phi_k}$ in Algorithm \ref{1} is verified. An application of (\ref{3.21}) and Algorithm \ref{1} for $(\ref{3.4})$ implies that 
		 $$\text{(ii-1)}\ \Sigma_{[1]\phi_k}\succeq\mathbb{O},\ \ \forall\ k\in\mathbb{S}_{\xi}^0,\ \ \ \text{(ii-2)}\ \Sigma^c_{\epsilon}(0)\succeq\mathbb{O}.$$
		 
	    \textbf{Step 2.} In view of $J_{[1]\phi_k}^{-1}=J_{[1]\phi_k}^{\top}$ and $P_{[1]\phi_k,i}^{-1}=P_{[1]\phi_k,i}^{\top}$, we have
		\begin{equation*}
		\mathbf{H}_{[1]\phi_k,\eta_k}=\biggl(\Bigl(\Bigl(\prod_{i=0}^{\eta_k-1}Q_{[1]\phi_k,i}P_{[1]\phi_k,i}\Bigr)J_{[1]\phi_k}\Bigr)^{-1}\biggr)^{\top}\mathbf{Z}.
		\end{equation*}
	    Combining (\ref{3.4}), (\ref{3.22}) and (\ref{3.29}) leads to
		\begin{align}\label{3.30}
		\mathbf{X}^{\top}\Sigma^c_{\epsilon}(0)\mathbf{X}=&\epsilon\mathbf{Z}^{\top}\Bigl(\sum_{k=1}^{\xi}\lambda_k^{+}\Sigma_{[1]\phi_k}\Bigr)\mathbf{Z}\notag\\
		\ge&\epsilon\min_{k\in\mathbb{S}_{\xi}^0}\bigg\{\lambda_k^{+}\Bigl(\prod_{i=0}^{\eta_k-1}\overline{a}_{\nu_{k(i)},i}^{\lfloor[1]\phi_k,i\rfloor}\Bigr)^2\bigg\}\Biggl\{\sum_{k=1}^{\xi}\mathbf{Z}^{\top}\biggl(M_{[1]\phi_k,\eta_k}\Bigl(\prod_{i=0}^{\eta_k-1}Q_{[1]\phi_k,i}P_{[1]\phi_k,i}\Bigr)J_{[1]\phi_k}\biggr)^{-1}\notag\\
		&\times\Delta_{[1]\phi_k,\eta_k} \biggl[\biggl(M_{[1]\phi_k,\eta_k}\Bigl(\prod_{i=0}^{\eta_k-1}Q_{[1]\phi_k,i}P_{[1]\phi_k,i}\Bigr)J_{[1]\phi_k}\biggr)^{-1}\biggr]^{\top}\mathbf{Z}\Biggr\}\notag\\
		=&\epsilon\min_{k\in\mathbb{S}_{\xi}^0}\bigg\{\lambda_k^{+}\Bigl(\prod_{i=0}^{\eta_k-1}\overline{a}_{\nu_{k(i)},i}^{\lfloor[1]\phi_k,i\rfloor}\Bigr)^2\bigg\}\sum_{k=1}^{\xi}\Biggl\{\biggl[\biggl(\Bigl(\Bigl(\prod_{i=0}^{\eta_k-1}Q_{[1]\phi_k,i}P_{[1]\phi_k,i}\Bigr)J_{[1]\phi_k}\Bigr)^{-1}\biggr)^{\top}\mathbf{Z}\biggr]^{\top}\notag\\
		&\times\Bigl(M_{[1]\phi_k,\eta_k}^{-1}\Delta_{[1]\phi_k,\eta_k}(M_{[1]\phi_k,\eta_k}^{-1})^{\top}\Bigr)\biggl[\biggl(\Bigl(\Bigl(\prod_{i=0}^{\eta_k-1}Q_{[1]\phi_k,i}P_{[1]\phi_k,i}\Bigr)J_{[1]\phi_k}\Bigr)^{-1}\biggr)^{\top}\mathbf{Z}\biggr]\Biggr\}\notag\\
		=&\epsilon\min_{k\in\mathbb{S}_{\xi}^0}\bigg\{\lambda_k^{+}\Bigl(\prod_{i=0}^{\eta_k-1}\overline{a}_{\nu_{k(i)},i}^{\lfloor[1]\phi_k,i\rfloor}\Bigr)^2\bigg\}\sum_{k=1}^{\xi}\mathbf{H}_{[1]\phi_k,\eta_k}^{\top}\left(\begin{array}{cc}
		\mathcal{M}_{[1]\eta_k}^{-1}\Xi_{[1]\phi_k,\eta_k}(\mathcal{M}_{[1]\eta_k}^{-1})^{\top} & \mathbb{O} \\ 
		\mathbb{O} & \mathbb{O} 
		\end{array}\right)\mathbf{H}_{[1]\phi_k,\eta_k}.
		\end{align}
		Denote by $\widehat{\lambda}_k$ the minimal eigenvalue of $\mathcal{M}_{[1]\eta_k}^{-1}\Xi_{[1]\phi_k,\eta_k}(\mathcal{M}_{[1]\eta_k}^{-1})^{\top}$, we can obtain $\widehat{\lambda}_k>0$ and
		\begin{equation}\label{3.31}
		M_{[1]\phi_k,\eta_k}^{-1}\Xi_{[1]\phi_k,\eta_k}(M_{[1]\phi_k,\eta_k}^{-1})^{\top}\succeq\widehat{\lambda}_k\mathbf{I}_{\eta_k}.
		\end{equation}
		Using the forms of $J_{[1]\phi_k}$, $P_{[1]\phi_k,i}$ and $Q_{[1]\phi_k,i}$, for any $\eta_k\ge 2$,
		\begin{equation}\label{3.32}
		\mathbf{H}_{[1]\phi_k,l}^{\langle j \rangle}=\mathbf{H}_{[1]\phi_k,j}^{\langle j\rangle},\ \ \ \forall\ j\in\mathbb{S}_{\eta_k}^0;l\in\mathbb{S}_{\eta_k}^j,
		\end{equation}
		which can be proved by a slight modification in \cite[Proof of (3.29)]{S39}.
		\\
		Define
		\begin{equation}\label{3.33}
		\rho_{\epsilon}=\epsilon\min_{k\in\mathbb{S}_{\xi}^0}\Bigl\{\widehat{\lambda}_k\lambda_k^{+}\Bigl(\prod_{i=0}^{\eta_k-1}\overline{a}_{\nu_{k(i)},i}^{\lfloor[1]\phi_k,i\rfloor}\Bigr)^2\Bigr\}>0.
		\end{equation}
		Inserting (\ref{3.31})-(\ref{3.33}) into (\ref{3.30}) yields
		\begin{align}\label{3.34}
		\mathbf{X}^{\top}\Sigma^c_{\epsilon}(0)\mathbf{X}\ge&\epsilon\min_{k\in\mathbb{S}_{\xi}^0}\bigg\{\lambda_k^{+}\Bigl(\prod_{i=0}^{\eta_k-1}\overline{a}_{\nu_{k(i)},i}^{\lfloor[1]\phi_k,i\rfloor}\Bigr)^2\bigg\}\sum_{k=1}^{\xi}\widehat{\lambda}_k(\mathbf{H}_{[1]\phi_k,\eta_k}^{\langle\eta_k\rangle})^{\top}\mathbf{H}_{[1]\phi_k,\eta_k}^{\langle\eta_k\rangle}\notag\\
		\ge&\rho_{\epsilon}\sum_{k=1}^{\xi}\sum_{j=1}^{\eta_k}(\mathbf{H}_{[1]\phi_k,\eta_k}^{(j)})^2\notag\\
		=&\rho_{\epsilon}\sum_{k=1}^{\xi}\Bigl(Z_{\phi_k}^2+\sum_{j=2}^{\eta_k}(\mathbf{H}_{[1]\phi_k,j}^{(j)})^2\Bigr).
		\end{align} 
		In the display above, we have used
		\begin{equation*}
		\mathbf{H}_{[1]\phi_k,1}^{(1)}=Z_{\phi_k},\ \ \ \forall\ k\in\mathbb{S}_{\xi}^0.
		\end{equation*}
		Thus, we get the desired result (\ref{3.5}). The proof is complete. 
	\end{proof}
	\begin{remark}\label{22}
  {\rm{Under Assumption \ref{2&1}, the PNOA method is employed to yield a local approximate SPS with the distribution $\mathbb{N}_n(\mathbf{X}^{\star}(t),\Sigma^c_{\epsilon}(t))$ for (\ref{1.2}), as shown in Theorem \ref{3-1}. Notably, such $\Sigma^c_{\epsilon}(t)$ is required to be positive definite when using the PNOA method. This condition ensures the approximate form for the joint marginal density of some subpopulations can be explicitly obtained for further application.

		Note that $\Sigma^c_{\epsilon}(t)$ is determined by $\Sigma^c_{\epsilon}(0)$, $\forall\ t>0$, we first focus on the special form and positive definiteness of $\Sigma^c_{\epsilon}(0)$. In fact, by (\ref{3.9}), (\ref{a.3}) and Assumption \ref{2&1}(c), $\Sigma^c_{\epsilon}(0)$ can be explicitly written as
		\begin{equation}\label{3.35}
		\Sigma^c_{\epsilon}(0)=\epsilon\sum_{i=1}^{\infty}\int_{0}^{\theta}\big(\boldsymbol{\Phi}^i_{C^{\star}}(\theta)\boldsymbol{\Phi}^{-1}_{C^{\star}}(t)\Gamma^{\star}(t)\big)\big(\boldsymbol{\Phi}^i_{C^{\star}}(\theta)\boldsymbol{\Phi}^{-1}_{C^{\star}}(t)\Gamma^{\star}(t)\big)^{\top}dt.
		\end{equation}
	    However, (\ref{3.35}) is almost unavailable in practical terms. Mainly, it is difficult to compute such matrix integral as it requires the accurate results of $\boldsymbol{\Phi}^{-1}_{C^{\star}}(t)$ and $\boldsymbol{\Phi}^i_{C^{\star}}(\theta)$ for any $(t,i)\in[0,\infty)\times\mathbb{S}_{\infty}^0$. Moreover, only $\Sigma^c_{\epsilon}(0)\succeq\mathbb{O}$ can be obtained under degenerate diffusion, but $\Sigma^c_{\epsilon}(0)\succ\mathbb{O}$ is unknown. Thus, we need to study $\Sigma^c_{\epsilon}(0)$ from a matrix equation perspective (i.e., Eq. (\ref{3.9})).

	    A common approach for solving the discrete-type Lyapunov equation $\Im_d(\Sigma,\boldsymbol{\varpi},\aleph)$ is to derive a simple canonical form of $\boldsymbol{\varpi}$ ($\mathscr{C}_f(\boldsymbol{\varpi})$ for short) by matrix transformations, thereby constructing implementable iteration schemes. Bartels--Stewart method \cite{S43,S44} is currently the most popular algorithm along this line and enjoys considerable success, the associated command ``dlyap($\cdot,\cdot$)" in MATLAB software is thus developed. The main idea of such numerical method is Schur factorization, i.e., find an orthogonal matrix $Q$ such that  $\mathscr{C}_f(\boldsymbol{\varpi})=Q^{-1}\boldsymbol{\varpi}Q\in\mathcal{U}(\cdot)$, then $\Im_c(\Sigma,\boldsymbol{\varpi},\aleph)=\mathbb{O}$ is equivalent to $\Im_c(Q^{\top}\Sigma Q,\mathcal{U}(\cdot),Q^{\top}\aleph Q)=\mathbb{O}$, which is easily calculated. Inspired by this, a few numerical methods based on different matrix factorizations (including Hessenberg--Schur transformation and Cholesky factorization) have been also developed \cite{S42,S45,S46}. Although these algorithms have large potential to solve (\ref{3.9}), the positive definiteness of $\Sigma^c_{\epsilon}(0)$ cannot be verified.
     
     Combined with the superposition principle and Gaussian-like elimination method, we propose a novel numerical framework for solving the general Lyapunov equation $\Im_d(\Sigma,\boldsymbol{\varpi},\aleph)$; see Algorithm \ref{1}. A key idea is to introduce a new canonical form $\mathscr{T}(\cdot)$ and combine similarity transformations such that $\mathscr{C}_f(\boldsymbol{\varpi}^{(m)})\in\mathscr{T}(l)$ for any $\boldsymbol{\varpi}\in\overline{\mathbf{CM}}(l)$, where $m\in\mathbb{S}_l^0$. Then $\Im_d(\Sigma,\boldsymbol{\varpi},\aleph)=\mathbb{O}$ can be equivalently transformed into an $m$-dimensional standard $\mathbb{L}_0$ algebraic equation in the sense that excluding zero matrix equation. One of the distinct advantages of Algorithm \ref{1} is to obtain the expression of $\Sigma^c_{\epsilon}(t)$ and its positive definiteness simultaneously. To be specific, using Proposition \ref{2&1} and Gaussian-like elimination method, for any $k\in\mathbb{S}_{\xi}^0$, we construct a sequence $\{\mathbf{H}_{[1]\phi_k,j}\}_{j=1}^{\eta_k}$, which can record the form and minimal rank (i.e., $\eta_k$) of all column components of $\Sigma_{[1]\phi_k}$. Then (\ref{3.4}) and the expression of $\Sigma^c_{\epsilon}(t)$ are derived. Meanwhile (\ref{3.6}), a criterion for analyzing $\Sigma^c_{\epsilon}(t)\succ\mathbb{O}$, is obtained. Specifically, there is a vector $\boldsymbol{h}_{[1]\phi_k,i}$ such that $\mathbf{H}_{[1]\phi_k,i}^{(i)}=\boldsymbol{h}_{[1]\phi_k,i}\mathbf{Z}$, $\forall\ i\in\mathbb{S}_{\eta_k}^0$. Note that $\mathbf{Z}=\mathcal{G}_0\mathbf{X}$, then $\Sigma^c_{\epsilon}(0)\succ\mathbb{O}$ is equivalent to proving that $\mathbf{X}^{\top}\Sigma^c_{\epsilon}(0)\mathbf{X}=0$ holds if and only if $\mathbf{Z}=\mathbf{0}$, which is easily verified. Furthermore, by Theorem \ref{3-1}, the desired result $\Sigma^c_{\epsilon}(t)\succ\mathbb{O}$ follows from
     \begin{equation*}
     \mathbf{X}^{\top}\Sigma^c_{\epsilon}(t)\mathbf{X}\ge\big(\boldsymbol{\Phi}^{\top}_{C^{\star}}(t)\mathbf{X}\big)^{\top}\Sigma^c_{\epsilon}(0)\big(\boldsymbol{\Phi}^{\top}_{C^{\star}}(t)\mathbf{X}\big).
     \end{equation*} 
     
	 It should be further noted that when solving Eq. (\ref{3.9}) using the aforementioned numerical algorithms, some complex iterative schemes and high computational cost are required for the canonical factorizations of $\boldsymbol{\Phi}_{C^{\star}}(\theta)$ and $\int_{0}^{\theta}\big(\boldsymbol{\Phi}_{C^{\star}}(\theta)\boldsymbol{\Phi}^{-1}_{C^{\star}}(t)\Gamma^{\star}(t)\big)\big(\boldsymbol{\Phi}_{C^{\star}}(\theta)\boldsymbol{\Phi}^{-1}_{C^{\star}}(t)\Gamma^{\star}(t)\big)^{\top}dt$. Instead, our Algorithm \ref{1} only needs to obtain some large $\eta_k\in\mathbb{S}_n^0$ satisfying $\mathscr{C}_f(A_{[1]}^{(\eta_k)})\in\mathscr{T}(\eta_k)$, $\forall\ k\in\mathbb{S}_{\xi}^0$. This is implemented by a Gaussian-like elimination method (i.e., (\ref{3.14}) and (\ref{3.15})), which is easier than matrix factorization technique. Thus, the computational structure of Algorithm \ref{1} is simpler.}}
	\end{remark}
	
	In part I of the series \cite{S39}, we provided a normal approximation (called uNA) approach to approximate the invariant measure of (\ref{1.2}) under autonomous case, yielding a complete numerical framework for solving the general continuous-type Lyapunov equation $\Im_c(\cdot,\cdot,\cdot)=\mathbb{O}$; see \cite[Algorithm 3]{S39}. Below we expect to study the connections between the PNOA and uNA approaches.

	If $f(t,\cdot)$ is independent of $t$, it follows from Assumptions \ref{2;1}(a) that the deterministic system of (\ref{1.2}) has a unique equilibrium $\mathbf{X}^*$. In this case, $C^{\star}(t)$ is a constant matrix, denoted by $C^*$ for simplicity. 
	\begin{proposition}\label{3&1}
		Under Assumption \ref{2;1}, if $f(t,\cdot)$ is time-independent, then $\Sigma_{\epsilon}^c(t)$ satisfies:
		\begin{equation}\label{3.36}
		\Im_c\Bigl(\Sigma_{\epsilon}^c(t),C^*,\epsilon\Gamma^{\star}(t)(\Gamma^{\star}(t))^{\top}-\frac{d(\Sigma_{\epsilon}^c(t))}{dt}\Bigr)=\mathbb{O}.
		\end{equation} 
	\end{proposition}
	\begin{proof}
		In view of (\ref{3.6}), for any $t\ge 0$,
		\begin{equation*}
		\begin{split}
		\mathbf{Z}_{\epsilon}(t+\theta)=&\boldsymbol{\Phi}_{C^{\star}}(t+\theta)\bigg(\boldsymbol{\Phi}^{-1}_{C^{\star}}(t)\mathbf{Z}_{\epsilon}(t)+\sqrt{\epsilon}\int_{t}^{t+\theta}\boldsymbol{\Phi}_{C^{\star}}^{-1}(s)\Gamma^{\star}(s)d\mathbf{W}(s)\bigg)\\
		=&\boldsymbol{\Phi}_{C^{\star}}(\theta)\mathbf{Z}_{\epsilon}(t)+\sqrt{\epsilon}\boldsymbol{\Phi}_{C^{\star}}(t+\theta)\int_{0}^{\theta}\boldsymbol{\Phi}_{C^{\star}}^{-1}(t+s)\Gamma^{\star}(t+s)d\mathbf{W}(s).
		\end{split}
		\end{equation*}
		This together with the definition of $\Sigma_{\epsilon}^c(t)$ and (\ref{3.8}) implies
		\begin{equation}\label{3.37}
		\begin{split}
		\Sigma_{\epsilon}^c(t)=&\boldsymbol{\Phi}_{C^{\star}}(\theta)\bigg[\Sigma_{\epsilon}^c(t)+\epsilon\int_{0}^{\theta}\big(\boldsymbol{\Phi}_{C^{\star}}(t)\boldsymbol{\Phi}^{-1}_{C^{\star}}(t+s)\Gamma^{\star}(t+s)\big)\big(\boldsymbol{\Phi}_{C^{\star}}(t)\boldsymbol{\Phi}^{-1}_{C^{\star}}(t+s)\Gamma^{\star}(t+s)\big)^{\top}ds\bigg]\boldsymbol{\Phi}^{\top}_{C^{\star}}(\theta).
		\end{split}
		\end{equation}
		If $f(t,\cdot)$ is time-independent, we have $\boldsymbol{\Phi}_{C^{\star}}(t)=e^{C^*t}$. Then (\ref{3.37}) is simplified as
		\begin{equation}\label{3.38}
		\Im_d\bigg(\Sigma^c_{\epsilon}(t),e^{C^*\theta},\epsilon\int_{0}^{\theta}\big(e^{C^*(\theta-s)}\Gamma^{\star}(t+s)\big)\big(e^{C^*(\theta-s)}\Gamma^{\star}(t+s)\big)^{\top}ds\bigg)=\mathbb{O}.
		\end{equation}
		Taking the derivative on both sides of (\ref{3.38}) yields
		\begin{equation}\label{3.39}
		\Im_d\bigg(\frac{d(\Sigma^c_{\epsilon}(t))}{dt},e^{C^*\theta},\epsilon\int_{0}^{\theta}e^{C^*(\theta-s)}\frac{d(\Gamma^{\star}(t+s)(\Gamma^{\star}(t+s))^{\top})}{dt}e^{(C^*)^{\top}(\theta-s)}ds\bigg)=\mathbb{O}.
		\end{equation}
		By calculation,
		\begin{equation*}
		\begin{split}
		\big(e^{-C^*\theta}&\Gamma^{\star}(t)\big)\big(e^{-C^*\theta}\Gamma^{\star}(t)\big)^{\top}-\Gamma^{\star}(t)(\Gamma^{\star}(t))^{\top}\\
		=&\int_{0}^{\theta}\frac{d}{ds}\Big(\big(e^{-C^*s}\Gamma^{\star}(t+s)\big)\big(e^{-C^*s}\Gamma^{\star}(t+s)\big)^{\top}\Big)ds\\
		=&\int_{0}^{\theta}e^{-C^*s}\frac{d(\Gamma^{\star}(t+s)(\Gamma^{\star}(t+s))^{\top})}{dt}e^{-(C^*)^{\top}s}ds-C^*\int_{0}^{\theta}\big(e^{-C^*s}\Gamma^{\star}(t+s)\big)\big(e^{-C^*s}\Gamma^{\star}(t+s)\big)^{\top}ds\\
		&-\int_{0}^{\theta}\big(e^{-C^*s}\Gamma^{\star}(t+s)\big)\big(e^{-C^*s}\Gamma^{\star}(t+s)\big)^{\top}ds(C^*)^{\top}.
		\end{split}
		\end{equation*}
		Combining (\ref{3.38}), one gets
		\begin{equation*}
		\begin{split}
		C^*\Sigma_{\epsilon}^c(t)+\Sigma_{\epsilon}^c(t)(C^*)^{\top}=&e^{C^*\theta}\bigg[C^*\Sigma_{\epsilon}^c(t)+\Sigma_{\epsilon}^c(t)(C^*)^{\top}+\epsilon\int_{0}^{\theta}e^{-C^*s}\frac{d(\Gamma^{\star}(t+s)(\Gamma^{\star}(t+s))^{\top})}{dt}e^{-(C^*)^{\top}s}ds\notag\\
		&-\epsilon\big(e^{-C^*\theta}\Gamma^{\star}(t)\big)\big(e^{-C^*\theta}\Gamma^{\star}(t)\big)^{\top}+\epsilon\Gamma^{\star}(t)(\Gamma^{\star}(t))^{\top}\bigg]e^{(C^*)^{\top}\theta},
		\end{split}
		\end{equation*}
		which implies
		\begin{align}\label{3.40}
		\Im_d\bigg(C^*\Sigma_{\epsilon}^c(t)+&\Sigma_{\epsilon}^c(t)(C^*)^{\top}+\epsilon\Gamma^{\star}(t)(\Gamma^{\star}(t))^{\top},e^{C^*\theta},\notag\\
		&\epsilon\int_{0}^{\theta}e^{C^*(\theta-s)}\frac{d(\Gamma^{\star}(t+s)(\Gamma^{\star}(t+s))^{\top})}{dt}e^{(C^*)^{\top}(\theta-s)}ds\bigg)=\mathbb{O}.
		\end{align}
		As in Assumptions \ref{2;1}(c), we have $e^{C^*\theta}\in\overline{\mathbf{CM}}(n)$. It follows from (\ref{3.39}), (\ref{3.40}) and Lemma \ref{2,1} that
		\begin{equation*}
		C^*\Sigma_{\epsilon}^c(t)+\Sigma_{\epsilon}^c(t)(C^*)^{\top}+\epsilon\Gamma^{\star}(t)(\Gamma^{\star}(t))^{\top}-\frac{d(\Sigma^c_{\epsilon}(t))}{dt}=\mathbb{O},
		\end{equation*}
		i.e., (\ref{3.36}) holds. The proof is complete.
	\end{proof}
	\begin{remark}\label{33}
		{\rm{Under autonomous case, $\Gamma^{\star}(t)$ will degenerate into a constant matrix, denoted by $\Gamma^*$ for convenience. By (\ref{3.39}), one has $\frac{d(\Sigma^c_{\epsilon}(t))}{dt}=\mathbb{O}$ and
		$$\Im_c\bigl(\Sigma_{\epsilon}^c(t),C^*,\epsilon\Gamma^*(\Gamma^*)^{\top}\bigr)=\mathbb{O}.$$ By a similar argument in \cite[(4.5) and (4.6)]{S39}, we determine
		\begin{equation*}
		\Sigma_{\epsilon}^c(t)\equiv\epsilon\int_{0}^{\infty}\bigl(e^{C^*s}\Gamma^*\bigr)\bigl(e^{C^*s}\Gamma^*\bigr)^{\top}ds:=\Sigma_{[o]\epsilon},\ \ \ \forall\ t\ge 0,
		\end{equation*}
		which coincides with that of \cite[Theorem 4.1]{S39}. In this case, the SPS with the distribution $\mathbb{N}_n(\mathbf{X}^{\star}(t),\Sigma^c_{\epsilon}(t))$ is degenerated into a stationary solution which follows $\mathbb{N}_n(\mathbf{X}^*,\Sigma_{[o]\epsilon})$. As a result, the PNOA method is a generalization of the uNA method in \cite{S39} involving the periodicity.
		}}
	\end{remark}
	As was alluded to, we provide a criterion (\ref{3.5}) for verifying $\Sigma^c_{\epsilon}(t)\succ\mathbb{O}$. But for a special case, where 
	the form of $\int_{0}^{\theta}\big(\boldsymbol{\Phi}_{C^{\star}}(\theta)\boldsymbol{\Phi}^{-1}_{C^{\star}}(t)\Gamma^{\star}(t)\big)\big(\boldsymbol{\Phi}_{C^{\star}}(\theta)\boldsymbol{\Phi}^{-1}_{C^{\star}}(t)\Gamma^{\star}(t)\big)^{\top}dt$ is complicated and the fundamental matrix $\boldsymbol{\Phi}_{C^{\star}}(\theta)$ is ``simple" in the sense that approaching the canonical form $\mathscr{T}(\cdot)$, the relevant analysis becomes tedious. To simplify, a modified criterion for analyzing $\Sigma^c_{\epsilon}(t)\succ\mathbb{O}$ is presented below.

	By (\ref{3.1}) and the definition of $\xi$, there exists two constants $\lambda_{\vartriangle}^+>0$, $\overline{\xi}\in\mathbb{S}_{\xi}^{-1}$ and a set $\{\overline{\phi}_1,...,\overline{\phi}_{\overline{\xi}}\}:=\overline{\boldsymbol{\phi}}\subseteq\mathbb{S}_n^0$ such that
	\begin{equation}\label{3.41}
	\int_{0}^{\theta}\big(\boldsymbol{\Phi}_{C^{\star}}(\theta)\boldsymbol{\Phi}^{-1}_{C^{\star}}(t)\Gamma^{\star}(t)\big)\big(\boldsymbol{\Phi}_{C^{\star}}(\theta)\boldsymbol{\Phi}^{-1}_{C^{\star}}(t)\Gamma^{\star}(t)\big)^{\top}dt\succeq\lambda_{\vartriangle}^{+}\sum_{k=1}^{\overline{\xi}}\amalg_{n,\overline{\phi}_k},
	\end{equation}
	where $\overline{\phi}_j>\overline{\phi}_i$, $\forall\ j>i$. If $\overline{\boldsymbol{\phi}}=\emptyset$, i.e., $\overline{\xi}=0$, we stipulate $\lambda_{\vartriangle}^{+}=1$.
	\begin{theorem}\label{3-2}
		Let Assumption \ref{2;1} hold. Then
		\begin{equation}\label{3.42}
		\mathbf{X}^{\top}\Sigma^c_{\epsilon}(0)\mathbf{X}\ge\overline{\rho}_{\epsilon}\sum_{k=1}^{\overline{\xi}}\Bigl(X_{\overline{\phi}_k}^2+\sum_{j=2}^{\overline{\eta}_k}(\mathbf{H}_{[2]\overline{\phi}_k,j}^{(j)})^2\Bigr),
		\end{equation}
		where $\overline{\rho}_{\epsilon}>0$ is given in (\ref{3.44}), $\mathbf{X}$ is the same as in (\ref{3.5}), and $$\mathbf{H}_{[2]\overline{\phi}_k,l}=\Big(\prod_{i=0}^{l-1}(Q_{[2]\overline{\phi}_k,i}^{-1})^{\top}P_{[2]\overline{\phi}_k,i}\Big)J_{[2]\overline{\phi}_k}\mathbf{X},\ \ \ \forall\ l\in\mathbb{S}_{\overline{\eta}_k}^0,$$
		with $\overline{\eta}_k,\ J_{[2]\overline{\phi}_k},\ P_{[2]\overline{\phi}_k,i}$ and $ Q_{[2]\overline{\phi}_k,i}$ shown in Algorithm \ref{2}.
		\\
		\begin{threeparttable}
			\begin{algorithm}[H]\label{2}
				\caption{Algorithm for obtaining $\mathbf{H}_{[2]\overline{\phi}_k,l}$}
				\LinesNumbered
				\KwIn{\rm{$\boldsymbol{\Phi}_{C^{\star}}(\theta)$, $\overline{\boldsymbol{\phi}}$.}}
				\KwOut{\rm{$\mathbf{H}_{[2]\overline{\phi}_k,l},\ \forall\ l\in\mathbb{S}_{\overline{\eta}_k}^0$.}
				}			
				\rm{\textbf{(Initialization):} $\overline{\eta}_k:=1,\ \overline{A}_{[2]\overline{\phi}_k,1}:=J_{[2]\overline{\phi}_k}\boldsymbol{\Phi}_{C^{\star}}(\theta)J_{[2]\overline{\phi}_k}^{-1}$\;}
				\rm{\textbf{(Recursive framework):} Using Gaussian-like elimination method, we can construct a FOR loop similar to Algorithm \ref{1}, to obtain the values of $\overline{\eta}_k$ and some suitable $\overline{\nu}_{k(i)}\in\mathbb{S}_n^i$, which yields 
				\\
				\ \ \ \ \ \ \ \ \ \ \ \ \ \ \ \ \ \ $\text{(2-i)}\ \overline{a}_{\overline{\nu}_{k(i)},i}^{\lfloor[2]\overline{\phi}_k,i\rfloor}\neq 0$\tnote{\textcolor{blue}{$\ddag$}},\ $\forall\ i\in\mathbb{S}_{\overline{\eta}_k-1}^0,\ \ \ \text{(2-ii)}\ \overline{a}_{j,\overline{\eta}_k}^{\lfloor[2]\overline{\phi}_k,\overline{\eta}_k\rfloor}=0,\ \forall\ j\in\mathbb{S}_n^{\overline{\eta}_k+1}$,\\
				where each $\overline{a}_{ji}^{\lfloor[2]\overline{\phi}_k,i\rfloor}$ is determined by the iterative scheme:
				\\
				\ \ \ \ \ \ \ \ \ \ \ \ $\text{(2-iii)}\ \widehat{A}_{[2]\overline{\phi}_k,i}:=P_{[2]\overline{\phi}_k,i}\overline{A}_{[2]\overline{\phi}_k,i}P_{[2]\overline{\phi}_k,i}^{-1}$\tnote{\textcolor{blue}{$\ddag$}},\ \ $\text{and}\ \ \overline{A}_{[2]\overline{\phi}_k,i+1}:=Q_{[2]\overline{\phi}_k,i}\widehat{A}_{[2]\overline{\phi}_k,i}Q_{[2]\overline{\phi}_k,i}^{-1}$\tnote{\textcolor{blue}{$\ddag$}},\ \ \ $\forall\ i\in\mathbb{S}_{\overline{\eta}_k}^0$}\;
				\rm{\textbf{return}} $\overline{\eta}_k$, $J_{[2]\overline{\phi}_k}$, $P_{[2]\overline{\phi}_k,i}$, $Q_{[2]\overline{\phi}_k,i}\ (i\in\mathbb{S}_{\overline{\eta}_k-1}^0).$
			\end{algorithm}
			\begin{tablenotes}
				\footnotesize
				\item[\rm{\textcolor{blue}{$\ddag$}}] \rm{$J_{[2]\overline{\phi}_k},\ P_{[2]\overline{\phi}_k,i}$ and $ Q_{[2]\overline{\phi}_k,i}$ have the same form as $J_{[1]\phi_k},\ P_{[1]\phi_k,i}$ and $Q_{[1]\phi_k,i}$ by only replacing $(\phi_k,\nu_{k(i)},\boldsymbol{\ell}_{[1]k,n-1-i})$ with $(\overline{\phi}_k,\overline{\nu}_{k(i)},\overline{\boldsymbol{\ell}}_{[2]k,n-1-i})$, where $\overline{\boldsymbol{\ell}}_{[2]k,n-1-i}=\frac{-1}{\widehat{a}_{i+1,i}^{\lfloor[2]\overline{\phi}_k,i\rfloor}}(\widehat{a}_{i+2,i}^{\lfloor[2]\overline{\phi}_k,i\rfloor},...,\widehat{a}_{n,i}^{\lfloor[2]\overline{\phi}_k,i\rfloor})^{\top}$. Analogously, we stipulate $\overline{a}_{\overline{\nu}_{k(0)},0}^{\lfloor[2]\overline{\phi}_k,0\rfloor}=1$ and $P_{[2]\overline{\phi}_k,m}=Q_{[2]\overline{\phi}_k,m}=\mathbf{I}_n$, $\forall\ k\in\mathbb{S}_{\overline{\xi}}^0;m\in\{0,n-1\}$.}
			\end{tablenotes}
		\end{threeparttable}
	\end{theorem}
	\begin{proof}
		Let $\widehat{\Sigma}_{[2]\overline{\phi}_k}\ (k\in\mathbb{S}_{\overline{\xi}}^0)$ be the solutions of the following algebraic equations, respectively:
		\begin{equation}\label{3.43}
		\Im_d\bigl(\widehat{\Sigma}_{[2]\overline{\phi}_k},\boldsymbol{\Phi}_{C^{\star}}(\theta),\amalg_{n,\overline{\phi}_k}\bigr)=\mathbb{O}.
		\end{equation}
		Combining Algorithm \ref{2} and a similar argument in (\ref{3.12})-(\ref{3.34}) yields that there is a $\widehat{\lambda}_{\vartriangle}>0$ such that
		\begin{align}\label{3.44}
		\mathbf{X}^{\top}\bigg(\sum_{k=1}^{\overline{\xi}}\widehat{\Sigma}_{[2]\overline{\phi}_k,\epsilon}\bigg)\mathbf{X}\ge&\widehat{\lambda}_{\vartriangle}\min_{k\in\mathbb{S}_{\overline{\xi}}^0}\bigg\{\Bigl(\prod_{i=0}^{\overline{\eta}_k-1}\overline{a}_{\overline{\nu}_{k(i)},i}^{\lfloor[2]\overline{\phi}_k,i\rfloor}\Bigr)^2\bigg\}\sum_{k=1}^{\overline{\xi}}\sum_{j=1}^{\overline{\eta}_k}\bigl(\mathbf{H}_{[2]\overline{\phi}_k,j}^{(j)}\bigr)^2\notag\\
		:=&\frac{\overline{\rho}_{\epsilon}}{\epsilon\lambda_{\vartriangle}^{+}}\sum_{k=1}^{\overline{\xi}}\sum_{j=1}^{\overline{\eta}_k}\bigl(\mathbf{H}_{[2]\overline{\phi}_k,j}^{(j)}\bigr)^2\notag\\
		=&\frac{\overline{\rho}_{\epsilon}}{\epsilon\lambda_{\vartriangle}^{+}}\sum_{k=1}^{\overline{\xi}}\Bigl(X_{\overline{\phi}_k}^2+\sum_{j=2}^{\overline{\eta}_k}(\mathbf{H}_{[2]\overline{\phi}_k,j}^{(j)})^2\Bigr),
		\end{align}
		where in the last equality, we have used the fact $\mathbf{H}_{[2]\overline{\phi}_k,1}^{(1)}=X_{\overline{\phi}_k}$, $\forall\ k\in\mathbb{S}_{\overline{\xi}}^0$.

		Below we consider an auxiliary Lyapunov equation:
		\begin{equation}\label{3.45}
		\Im_d\biggl(\Sigma_{\text{aux}},\boldsymbol{\Phi}_{C^{\star}}(\theta),\epsilon\Big(\int_{0}^{\theta}\big(\boldsymbol{\Phi}_{C^{\star}}(\theta)\boldsymbol{\Phi}^{-1}_{C^{\star}}(t)\Gamma^{\star}(t)\big)\big(\boldsymbol{\Phi}_{C^{\star}}(\theta)\boldsymbol{\Phi}^{-1}_{C^{\star}}(t)\Gamma^{\star}(t)\big)^{\top}dt-\lambda_{\vartriangle}^{+}\sum_{k=1}^{\overline{\xi}}\amalg_{n,\overline{\phi}_k}\Big)\biggr)=\mathbb{O}.
		\end{equation}
		Under Assumption \ref{2;1}(c), by (\ref{3.41}) and a standard argument in (\ref{a.2})-(\ref{a.4}), one has $\lim_{k\rightarrow\infty}\boldsymbol{\Phi}_{C^{\star}}^k(\theta)=\mathbb{O}$, and
		\begin{equation*}
		\Sigma_{\text{aux}}=\epsilon\sum_{k=0}^{\infty}\boldsymbol{\Phi}_{C^{\star}}^k(\theta)\Big(\int_{0}^{\theta}\big(\boldsymbol{\Phi}_{C^{\star}}(\theta)\boldsymbol{\Phi}^{-1}_{C^{\star}}(t)\Gamma^{\star}(t)\big)\big(\boldsymbol{\Phi}_{C^{\star}}(\theta)\boldsymbol{\Phi}^{-1}_{C^{\star}}(t)\Gamma^{\star}(t)\big)^{\top}dt-\lambda_{\vartriangle}^{+}\sum_{k=1}^{\overline{\xi}}\amalg_{n,\overline{\phi}_k}\Big)(\boldsymbol{\Phi}_{C^{\star}}^k(\theta))^{\top}\succeq\mathbb{O}.
		\end{equation*}
	   Both (\ref{3.43}) and (\ref{3.45}) imply
	\begin{equation}\label{3.46}
	\Sigma^c_{\epsilon}(0)=\epsilon\lambda_{\vartriangle}^{+}\sum_{k=1}^{\overline{\xi}}\widehat{\Sigma}_{[2]\overline{\phi}_k}+\Sigma_{\text{aux}}.
	\end{equation}
	Thus, the desired result follows from (\ref{3.44}) and (\ref{3.46}).
	\end{proof}
	Using Theorems \ref{3-1} and \ref{3-2}, we provide some sufficient conditions for $\Sigma^c_{\epsilon}(t)\succ\mathbb{O}$, which are stated as in the following Corollary.  
	\begin{corollary}\label{3/1}
		Under Assumption \ref{2;1}, if one of the following four conditions holds:
		$${\rm{(i)}}\ \xi=n,\ \  {\rm{(ii)}}\ \eta_{k_1}=n,\ \exists\ k_1\in\mathbb{S}_{\xi}^0,\ \ {\rm{(iii)}}\ \overline{\xi}=n,\ \  {\rm{(iv)}}\ \overline{\eta}_{k_2}=n,\ \exists\ k_2\in\mathbb{S}_{\overline{\xi}}^0.$$
		Then $\Sigma^c_{\epsilon}(t)\succ\mathbb{O}$, $\forall\ t\ge 0$.
	\end{corollary}
	\begin{proof}
		As shown before in Remark \ref{22}, it is sufficient to verify that $\mathbf{X}^{\top}\Sigma^c_{\epsilon}(0)\mathbf{X}=0$ holds if and only if $\mathbf{Z}=\mathbf{0}$ or $\mathbf{X}=\mathbf{0}$. The related analysis can be divided into the following three cases:

		Case 1. If condition (i) is satisfied, we have $\boldsymbol{\phi}=\mathbb{S}_n^0$, i.e., $\phi_i=i,\ \forall\ i\in\mathbb{S}_n^0$. In view of (\ref{3.5}),
		\begin{equation*}
		0=\mathbf{X}^{\top}\Sigma^c_{\epsilon}(0)\mathbf{X}\ge\rho_{\epsilon}\sum_{k=1}^{n}Z_k^2=\rho_{\epsilon}|\mathbf{Z}|^2.
		\end{equation*}
		Intuitively, $\mathbf{Z}=\mathbf{0}$ is required.

		Case 2. If condition (ii) is satisfied, using (\ref{3.5}) and (\ref{3.32}),
		\begin{equation*}
		0=\mathbf{X}^{\top}\Sigma^c_{\epsilon}(0)\mathbf{X}\ge\rho_{\epsilon}\biggl(Z_{\phi_{k_1}}^2+\sum_{j=2}^{\eta_{k_1}}(\mathbf{H}_{[1]\phi_{k_1},n}^{(j)})^2\biggr)=\rho_{\epsilon}|\mathbf{H}_{[1]\phi_{k_1},n}|^2.
		\end{equation*}
		It then follows from $\mathbf{H}_{[1]\phi_{k_1},n}=\bigl[\prod_{i=0}^{n-1}(Q_{[1]\phi_{k_1},i}^{-1})^{\top}P_{[1]\phi_{k_1},i}\bigr]J_{[1]\phi_{k_1}}\mathbf{Z}$ that $\mathbf{Z}=\mathbf{0}$.

		Case 3. Under condition (iii) (resp., (iv)), the desired results can be obtained by (\ref{3.42}) together with a similar argument in case 1 (resp., case 2), and is thus omitted.
	\end{proof}
	\begin{remark}\label{44}
	{\rm{We highlight a special diffusion of (\ref{1.2}):
	\begin{equation}\label{3.47}
	G_c(t,\mathbf{X}_{\epsilon})(G_c(t,\mathbf{X}_{\epsilon}))^{\top}=\text{diag}\bigl\{\sigma_1^2(t)X_{\epsilon,1}^2,...,\sigma_n^2(t)X_{\epsilon,n}^2\bigr\},
	\end{equation}
	where $\sigma_i(t)$ is a positive $\theta$-periodic function, $\forall\ i\in\mathbb{S}_n^0$. In fact, (\ref{3.47}) is known as periodic linear diffusion \cite{S47,S48}, and is a common way to introduce stochasticity and periodicity into
	biologically realistic dynamic models. If further $X_{\epsilon,i}(t)\neq 0,\ \forall\ (t,i)\in[0,\infty)\times\mathbb{S}_n^0$, then $\Gamma^{\star}(t)(\Gamma^{\star}(t))^{\top}\succ\mathbb{O}$, which falls into case (i). Thus, $\Sigma^c_{\epsilon}(t)\succ\mathbb{O}$.}}
	\end{remark}
\section{Periodic log-normal approximation (PLNA)}
In this section, we aim to provide a PLNA method for explicit approximation of the SPSD of (\ref{1.2}), with right-skewed probability distributions.

Under Assumption \ref{2;2}, it follows from (\ref{2.3}) and the It\^{o}'s formula that 
\begin{equation}\label{4.1}
d(\ln X_{\epsilon,i}(t)-\Psi^{\star}_{\epsilon,i}(t))=\bigl(F_i(t,\mathbf{X}_{\epsilon}(t))-F_i(t,e^{\boldsymbol{\Psi}^{\star}_{\epsilon}(t)})\bigr)dt+\sqrt{\epsilon}\sum\limits_{j=1}^{N}g_{ij}(t,\mathbf{X}_{\epsilon}(t))dW_j(t),\ \ \ \ i\in\mathbb{S}_n^0.
\end{equation}
Then by Taylor expansion, the linearized equations of (\ref{1.2}) near $\mathbf{X}^*$ is
\begin{equation}\label{4.2}
\begin{cases}
d\mathbf{Y}_{\epsilon}(t)=\Bigl(\dfrac{\partial F_i(t,\mathbf{x})}{\partial(\ln x_j)}\Bigr)_{n\times n}\bigl|_{\mathbf{x}=e^{\boldsymbol{\Psi}^{\star}_{\epsilon}(t)}}\mathbf{Y}_{\epsilon}(t)dt+\sqrt{\epsilon}(g_{ij}(t,e^{\boldsymbol{\Psi}^{\star}_{\epsilon}(t)}))_{n\times N}d\mathbf{W}(t)\\
\ \ \ \ \ \ \ \ \ :=D^{\star}(t)\mathbf{Y}_{\epsilon}(t)dt+\sqrt{\epsilon}\Lambda^{\star}(t)\mathbf{W}(t),\\
\mathbf{Y}_{\epsilon}(0)=\ln\mathbf{x}_0-\boldsymbol{\Psi}^{\star}_{\epsilon}(0).
\end{cases}
\end{equation}
For simplicity, we define $$\Upsilon(t)=\int_{0}^{t}\big(\boldsymbol{\Phi}_{D^{\star}}(t)\boldsymbol{\Phi}^{-1}_{D^{\star}}(s)\Lambda^{\star}(s)\big)\big(\boldsymbol{\Phi}_{D^{\star}}(t)\boldsymbol{\Phi}^{-1}_{D^{\star}}(s)\Lambda^{\star}(s)\big)^{\top}ds.$$
Obviously, $\Upsilon(t)\succeq\mathbb{O}$, $\forall\ t\ge 0$. Analogous to (\ref{3.1}), denote by $\lambda_k^{\oplus}\ (k\in\mathbb{S}_{\omega}^0)$ all positive eigenvalues of $\Upsilon(\theta)$, where
$\omega=\text{rank}(\Upsilon(\theta))$. Then we can determine a set $\{\mu_1,...,\mu_{\omega}\}:=\boldsymbol{\mu}\subseteq\mathbb{S}_n^0$ and an orthogonal matrix $\mathcal{H}_{\epsilon}$ satisfying
\begin{equation}\label{4.3}
\mathcal{H}_{\epsilon}\Upsilon(\theta)\mathcal{H}_{\epsilon}^{\top}=\sum_{k=1}^{\omega}\lambda_k^{\oplus}\amalg_{n,\mu_k},
\end{equation}
where $\mu_j>\mu_i$, $\forall\ j>i$.

By a standard argument in (\ref{3.3}) and (\ref{3.6})-(\ref{3.10}), we obtain that system (\ref{4.2}) admits a unique SPS $(\mathbf{Y}_{\epsilon,\theta}(t))$ which follows the distribution $\mathbb{N}_n(\mathbf{0},\Sigma^d_{\epsilon}(t))$, where  
\begin{equation*}
\Sigma^d_{\epsilon}(t)=\mathbb{V}\text{ar}(\mathbf{Y}_{\epsilon,\theta}(t))=\boldsymbol{\Phi}_{D^{\star}}(t)\Sigma^d_{\epsilon}(0)\boldsymbol{\Phi}_{D^{\star}}^{\top}(t)+\epsilon\Upsilon(t),
\end{equation*}
and $\Sigma^d_{\epsilon}(0)$ is determined by
\begin{equation}\label{4.4}
\Im_d\big(\Sigma^d_{\epsilon}(0),\boldsymbol{\Phi}_{D^{\star}}(\theta),\epsilon\Upsilon(\theta)\big)=\mathbb{O}.
\end{equation}
In view of the relationship between (\ref{4.1}) near the solution $\boldsymbol{\Psi}^{\star}_{\epsilon}(t)$ and (\ref{4.2}), the process $\ln\mathbf{X}_{\epsilon}(t)$ near $\boldsymbol{\Psi}^{\star}_{\epsilon}(t)$ can be approximated by $\mathbf{Y}_{\epsilon}(t)+\boldsymbol{\Psi}^{\star}_{\epsilon}(t)$. This implies that system (\ref{1.2}) approximately has a local periodic solution which follows the distribution $\mathbb{LN}_n(\boldsymbol{\Psi}^{\star}_{\epsilon}(t),\Sigma^d_{\epsilon}(t))$.

To study the positive definiteness of $\Sigma^d_{\epsilon}(t)$, a natural approach is to obtain the specific form of $\Sigma^d_{\epsilon}(0)$. To this end, let $$A_{[3]}=\mathcal{H}_{\epsilon}\boldsymbol{\Phi}_{D^{\star}}(\theta)\mathcal{H}_{\epsilon}^{-1}.$$
Under Assumption \ref{2;2}(3), we have $A_{[3]}\in\overline{\mathbf{CM}}(n)$. Consider the following auxiliary Lyapunov equations
\begin{equation}\label{4.5}
\Im_d(\Sigma_{[3]\mu_k,\epsilon},A_{[3]},\amalg_{n,\mu_k})=\mathbb{O},\ \ \ \forall\ k\in\mathbb{S}_{\omega}^0.
\end{equation}
Below we need to study the minimal rank of the column components of $\Sigma_{[3]\mu_k,\epsilon}$.
\begin{theorem}\label{4-1}
	Let Assumption \ref{2;2} hold. For sufficiently small $\epsilon$, system (\ref{1.2}) approximately admits a local periodic solution with the distribution $\mathbb{LN}_n(\boldsymbol{\Psi}^{\star}_{\epsilon}(t),\Sigma^d_{\epsilon}(t))$, where
	$\Sigma^d_{\epsilon}(t)=\boldsymbol{\Phi}_{D^{\star}}(t)\Sigma^d_{\epsilon}(0)\boldsymbol{\Phi}_{D^{\star}}^{\top}(t)+\epsilon\Upsilon(t)$, and
	\begin{equation}\label{4.6}
	\Sigma^d_{\epsilon}(0)=\epsilon\mathcal{H}_{\epsilon}^{\top}\Bigl(\sum_{k=1}^{\omega}\lambda_k^{\oplus}\Sigma_{[3]\mu_k,\epsilon}\Bigr)\mathcal{H}_{\epsilon},
	\end{equation}
	with $\Sigma_{[3]\mu_k,\epsilon}$ shown in Algorithm \ref{3}. Moreover, for any constant vector $\mathbf{X}\in\mathbb{R}^n$, let $\mathbf{Y}=\mathcal{H}_{\epsilon}\mathbf{X}$ and $$\mathbf{H}_{[3]\mu_k,j}=\Big(\sum_{i=0}^{j-1}\big(Q_{[3]\mu_k,i}^{-1}\big)^{\top}P_{[3]\mu_k,i}\Big)J_{[3]\mu_k}\mathbf{Y},$$
	one has
	\begin{equation}\label{4.7}
	\mathbf{X}^{\top}\Sigma^d_{\epsilon}(0)\mathbf{X}\ge\varrho_{\epsilon}\sum_{k=1}^{\omega}\Bigl(Y_{\mu_k}^2+\sum_{j=2}^{\delta_k}(\mathbf{H}_{[3]\mu_k,j}^{(j)})^2\Bigr),
	\end{equation}
	where $\varrho_{\epsilon}>0$, $\delta_k,\ J_{[3]\mu_k},\ P_{[3]\mu_k,i}$ and $Q_{[3]\mu_k,i}\ (\forall\ i\in\mathbb{S}_{\delta_k}^0)$ are given in Algorithm \ref{3}.
	\\
	\begin{threeparttable}
		\begin{algorithm}[H]\label{3}
			\caption{Algorithm for obtaining $\Sigma_{[3]\mu_k,\epsilon}\ (k\in\mathbb{S}_{\omega}^0)$}
			\LinesNumbered
			\KwIn{\rm{$A_{[3]}$, $\boldsymbol{\mu}$.}}
			\KwOut{\rm{$\delta_k$, $\Sigma_{[3]\mu_k,\epsilon}=(\prod_{i=0}^{\delta_k-1}\overline{a}_{\varsigma_{k(i)},i}^{\lfloor[3]\mu_k,i\rfloor})^2$ $\times[M_{[3]\mu_k,\delta_k}(\prod_{i=0}^{\delta_k-1}Q_{[3]\mu_k,i}P_{[3]\mu_k,i})J_{[3]\mu_k}]^{-1}\Delta_{[3]\mu_k,\delta_k} \{[M_{[3]\mu_k,\delta_k}(\prod_{i=0}^{\delta_k-1}Q_{[3]\mu_k,i}P_{[3]\mu_k,i})J_{[3]\mu_k}]^{-1}\}^{\top}$\tnote{\textcolor{blue}{$\sharp$}}.}
			}			
			\rm{\textbf{(Initialization):} $A_{[3]\mu_k,1}=J_{[3]\mu_k}A_{[3]}J_{[3]\mu_k}^{-1}$, $\delta_k=1$\;}
			\For{$i=1:n-1$}{
				\eIf{$\sum_{j=i+1}^{n}(\overline{a}_{ji}^{\lfloor[3]\mu_k,i\rfloor})^2=0$}
				{
					$\delta_k=i$\;
					\rm{\textbf{break}}\;
				}
				{
					Choose a ``\underline{\textbf{suitable}} \tnote{\textcolor{blue}{$\ddag$}}" $\varsigma_{k(i)}\in\mathbb{S}_n^i$ such that $\overline{a}_{\varsigma_{k(i)},i}^{\lfloor[3]\mu_k,i\rfloor}\neq 0$. Then, $\widehat{A}_{[3]\mu_k,i}=P_{[3]\mu_k,i}\overline{A}_{[3]\mu_k,i}P_{[3]\mu_k,i}^{-1}$ and $\overline{A}_{[3]\mu_k,i+1}:=Q_{[3]\mu_k,i}\widehat{A}_{[3]\mu_k,i}Q_{[3]\mu_k,i}^{-1}$\;
				}
				$\delta_k$++;
			}
			Obtain a standard $\mathbb{L}_0$-algebraic equation $\Im_d\bigl(\Xi_{[3]\mu_k,\delta_k},A_{[3]s,\mu_k}^{(\delta_k)},\amalg_{\delta_k,1}\bigr)=\mathbb{O}$, where $A_{[3]s,\mu_k}=M_{[3]\mu_k,\delta_k}\overline{A}_{[3]\mu_k,\delta_k}M_{[3]\mu_k,\delta_k}^{-1}$\;
			\rm{\textbf{return}} $\delta_k,\ \varrho_{\epsilon},\ \Xi_{[3]\mu_k,\delta_k},\ \Sigma_{[3]\mu_k,\epsilon}.$
		\end{algorithm}
		\begin{tablenotes}
			\footnotesize
			\item[\rm{\textcolor{blue}{$\sharp$}}] \rm{We stipulate $P_{[3]\mu_k,l}=Q_{[3]\mu_k,l}=\mathbf{I}_n$ and $\overline{a}_{\varsigma_{k(0)},0}^{\lfloor[3]\mu_k,0\rfloor}=1$ for any $l\in\{0,n-1\}$ and $k\in\mathbb{S}_{\omega}^0$. Moreover, $J_{[3]\mu_k},\ P_{[3]\mu_k,i}$, $Q_{[3]\mu_k,i}$ and $\mathcal{M}_{[3]\delta_k}$ have the same form as $J_{[1]\phi_k},\ P_{[1]\phi_k,i}$, $ Q_{[1]\phi_k,i}$ and $\mathcal{M}_{[3]\delta_k}$ by replacing $(\phi_k,\eta_k,\nu_{k(i)},\boldsymbol{\ell}_{[1]k,n-1-i},\overline{A}_{[1]\phi_k,\eta_k})$ with $(\mu_k,\delta_{k(i)},\varsigma_{k(i)},\boldsymbol{\ell}_{[3]k,n-1-i},\overline{A}_{[3]\mu_k,\delta_k})$, where $\boldsymbol{\ell}_{[3]k,n-1-i}=\frac{-1}{\widehat{a}_{i+1,i}^{\lfloor[3]\overline{\phi}_k,i\rfloor}}(\widehat{a}_{i+2,i}^{\lfloor[3]\overline{\phi}_k,i\rfloor},...,\widehat{a}_{n,i}^{\lfloor[3]\overline{\phi}_k,i\rfloor})^{\top}$, and
			\begin{equation*}
			M_{[3]\mu_k,\delta_k}=\left(\begin{array}{cc}
			\mathcal{M}_{[3]\delta_k} & \mathbb{O} \\ 
			\mathbb{O} & \mathbf{I}_{n-\delta_k}
			\end{array}\right),\ \ \ \Delta_{[3]\mu_k,\delta_k}=\left(\begin{array}{cc}
			\Xi_{[3]\mu_k,\delta_k} & \mathbb{O} \\ 
			\mathbb{O} & \mathbb{O} 
			\end{array}\right),\ \ \ \varrho_{\epsilon}=\epsilon\min_{k\in\mathbb{S}_{\omega}^0}\bigg\{\widehat{\lambda}_k^{\diamond}\lambda_k^{\oplus}\Bigl(\prod_{i=0}^{\delta_k-1}\overline{a}_{\varsigma_{k(i)},i}^{\lfloor[3]\phi_k,i\rfloor}\Bigr)^2\bigg\},
			\end{equation*}
			with $\widehat{\lambda}_k^{\diamond}$ being the minimal eigenvalue of $\mathcal{M}_{[3]\delta_k}^{-1}\Xi_{[3]\mu_k,\delta_k}(\mathcal{M}_{[3]\delta_k}^{-1})^{\top}$.}
			\item[\rm{\textcolor{blue}{$\ddag$}}] \rm{The choice of $\varsigma_{k(i)}$ is conducive to verifying $\Sigma^d_{\epsilon}(0)\succ\mathbb{O}$.}
		\end{tablenotes}
	\end{threeparttable}
\end{theorem}
\begin{proof}
	Using (\ref{4.3})-(\ref{4.5}), the superposition principle and the orthogonality of $\mathcal{H}_{\epsilon}$, we obtain (\ref{4.6}).

	Next by proceeding the procedures 1-10 in Algorithm \ref{3} and mimicking the proof of Theorem \ref{3-1} (mainly, (\ref{3.12})-(\ref{3.29})), we can determine that (i) $\overline{A}_{[3]\mu_k,\delta_k}^{(\delta_k)}\in\mathcal{U}_{cm}(\delta_k)$, (ii) $\delta_k$ is the minimal rank of all column components of $\Sigma_{[3]\mu_k,\epsilon}$, $\forall\ k\in\mathbb{S}_{\omega}^0$, and (iii) 
	\begin{equation*}
	\widetilde{\Sigma}_{[3]\mu_k,\epsilon}=\left(\begin{array}{cc}
	\widetilde{\Sigma}_{[3]\mu_k,\epsilon}^{(\delta_k)} & \mathbb{O} \\ 
	\mathbb{O} & \mathbb{O}
	\end{array}\right).
	\end{equation*} 
	where
	\begin{equation}\label{4.8}
	\widetilde{\Sigma}_{[3]\mu_k,\epsilon}:=\bigg(\Bigl(\prod_{i=0}^{\delta_k-1}Q_{[3]\mu_k,i}P_{[3]\mu_k,i}\Bigr)J_{[3]\mu_k}\bigg)\Sigma_{[3]\mu_k,\epsilon}\bigg(\Bigl(\prod_{i=0}^{\delta_k-1}Q_{[3]\mu_k,i}P_{[3]\mu_k,i}\Bigr)J_{[3]\mu_k}\bigg)^{\top}.
	\end{equation}
	Combined with Proposition \ref{2&2} and the definition of $M_{[3]\mu_k,\delta_k}$, we have $A_{[3]s,\mu_k}^{(\delta_k)}\in\mathscr{T}(\delta_k)$ and 
	\begin{equation*}
	M_{[3]\mu_k,\delta_k}\amalg_{n,\mu_k}M_{[3]\mu_k,\delta_k}^{\top}=\left(\begin{array}{cc}
	\Bigl(\prod\limits_{i=0}^{\delta_k-1}\overline{a}_{\varsigma_{k(i)},i}^{\lfloor[3]\mu_k,i\rfloor}\Bigr)^2\amalg_{\delta_k,1} & \mathbb{O} \\ 
	\mathbb{O} & \mathbb{O}
	\end{array}\right).
	\end{equation*}
	Thus, Eqs. (\ref{4.5}) can be equivalent to a standard $\mathbb{L}_0$-algebraic equation
	\begin{equation*}
	\Im_d\biggl(\Bigl(\prod\limits_{i=0}^{\delta_k-1}\overline{a}_{\varsigma_{k(i)},i}^{\lfloor[3]\mu_k,i\rfloor}\Bigr)^{-2}M_{[3]\mu_k,\delta_k}\widetilde{\Sigma}_{[3]\mu_k,\epsilon}M_{[3]\mu_k,\delta_k}^{\top},A_{[3]s,\mu_k}^{(\delta_k)},\amalg_{\delta_k,1}\biggr)=\mathbb{O}.
	\end{equation*}
	This together with (\ref{4.8}), Proposition \ref{2&1} and the procedure 11 of Algorithm \ref{3} implies
	\begin{align}\label{4.9}
	\Sigma_{[3]\mu_k,\epsilon}=&\bigg(M_{[3]\mu_k,\delta_k}\Bigl(\prod_{i=0}^{\delta_k-1}Q_{[3]\mu_k,i}P_{[3]\mu_k,i}\Bigr)J_{[3]\mu_k}\bigg)^{-1}\left(\begin{array}{cc}
	\mathcal{M}_{[3]\delta_k}\widetilde{\Sigma}_{[3]\mu_k}^{(\delta_k)}\mathcal{M}_{[3]\delta_k}^{\top} & \mathbb{O} \\ 
	\mathbb{O} & \mathbb{O}
	\end{array}\right)\notag\\
	&\times\bigg[\bigg(M_{[3]\mu_k,\delta_k}\Bigl(\prod_{i=0}^{\delta_k-1}Q_{[3]\mu_k,i}P_{[3]\mu_k,i}\Bigr)J_{[3]\mu_k}\bigg)^{-1}\bigg]^{\top}\notag\\
	=&\Bigl(\prod\limits_{i=0}^{\delta_k-1}\overline{a}_{\varsigma_{k(i)},i}^{\lfloor[3]\mu_k,i\rfloor}\Bigr)^2\bigg(M_{[3]\mu_k,\delta_k}\Bigl(\prod_{i=0}^{\delta_k-1}Q_{[3]\mu_k,i}P_{[3]\mu_k,i}\Bigr)J_{[3]\mu_k}\bigg)^{-1}\notag\\
	&\times\Delta_{[1]\phi_k,\eta_k} \bigg[\bigg(M_{[3]\mu_k,\delta_k}\Bigl(\prod_{i=0}^{\delta_k-1}Q_{[3]\mu_k,i}P_{[3]\mu_k,i}\Bigr)J_{[3]\mu_k}\bigg)^{-1}\bigg]^{\top}.
	\end{align}
	Then by (\ref{4.9}) and the definitions of $\mathbf{H}_{[3]\mu_k,i}$ and $\varrho_{\epsilon}$, we obtain 
	\begin{align}\label{4.10}
	\mathbf{X}^{\top}\Sigma^d_{\epsilon}(0)\mathbf{X}=&\epsilon\mathbf{Y}^{\top}\Bigl(\sum_{k=1}^{\omega}\lambda_k^{\oplus}\Sigma_{[3]\mu_k,\epsilon}\Bigr)\mathbf{Y}\notag\\
	\ge&\epsilon\min_{k\in\mathbb{S}_{\omega}^0}\bigg\{\lambda_k^{\oplus}\Bigl(\prod\limits_{i=0}^{\delta_k-1}\overline{a}_{\varsigma_{k(i)},i}^{\lfloor[3]\mu_k,i\rfloor}\Bigr)^2\bigg\}\sum_{k=1}^{\omega}\Biggl\{\biggl[\biggl(\Bigl(\Bigl(\prod_{i=0}^{\delta_k-1}Q_{[3]\mu_k,i}P_{[3]\mu_k,i}\Bigr)J_{[3]\mu_k}\Bigr)^{-1}\biggr)^{\top}\mathbf{Y}\biggr]^{\top}\notag\\
	&\times\left(\begin{array}{cc}
	\mathcal{M}_{[3]\delta_k}^{-1}\Xi_{[3]\mu_k,\delta_k}(\mathcal{M}_{[3]\delta_k}^{-1})^{\top} & \mathbb{O} \\ 
	\mathbb{O} & \mathbb{O} 
	\end{array}\right)\biggl[\biggl(\Bigl(\Bigl(\prod_{i=0}^{\delta_k-1}Q_{[3]\mu_k,i}P_{[3]\mu_k,i}\Bigr)J_{[3]\mu_k}\Bigr)^{-1}\biggr)^{\top}\mathbf{Y}\biggr]\Biggr\}\notag\\
	\ge&\epsilon\min_{k\in\mathbb{S}_{\omega}^0}\bigg\{\lambda_k^{\oplus}\Bigl(\prod\limits_{i=0}^{\delta_k-1}\overline{a}_{\varsigma_{k(i)},i}^{\lfloor[3]\mu_k,i\rfloor}\Bigr)^2\bigg\}\sum_{k=1}^{\omega}\widehat{\lambda}_k^{\diamond}(\mathbf{H}_{[3]\mu_k,\delta_k}^{\langle\delta_k\rangle})^{\top}\mathbf{H}_{[3]\mu_k,\delta_k}^{\langle\delta_k\rangle}\notag\\
	\ge&\varrho_{\epsilon}\sum_{k=1}^{\omega}\sum_{i=1}^{\delta_k}(\mathbf{H}_{[3]\mu_k,i}^{(i)})^2.
	\end{align}
	In the display above, we have used
	\begin{equation*}
	\biggl(\Bigl(\Bigl(\prod_{i=0}^{\delta_k-1}Q_{[3]\mu_k,i}P_{[3]\mu_k,i}\Bigr)J_{[3]\mu_k}\Bigr)^{-1}\biggr)^{\top}\mathbf{Y}=\mathbf{H}_{[3]\mu_k,\delta_k},\ \ \text{and}\ \ \mathbf{H}_{[3]\mu_k,j}^{\langle l\rangle}=\mathbf{H}_{[3]\mu_k,l}^{\langle l\rangle},\ \ \ \forall\ l\in\mathbb{S}_{\delta_k}^0;j\in\mathbb{S}_{\delta_k}^l.
	\end{equation*}
	Then the desired result (\ref{4.7}) follows from (\ref{4.10}) and $\mathbf{H}_{[3]\mu_k,1}^{(1)}=Y_{\mu_k}$, $\forall\ k\in\mathbb{S}_{\omega}^0$. This completes the proof.
\end{proof}
It should be mentioned that although a criterion (\ref{4.7}) is provided for verifying $\Sigma^d_{\epsilon}(t)\succ\mathbb{O}$, the relevant analysis may be laborious if $\Upsilon(\theta)$ is highly complex and $\boldsymbol{\Phi}_{D^{\star}}(\theta)$ is ``simple" in the sense that approaching the canonical form $\mathscr{T}(\cdot)$. In this case, we introduce another available criterion to simplify calculations.

Under Assumption \ref{2;2}, there exists two constants $\overline{\omega}\in\mathbb{S}_{\omega}^0$, $\lambda_{\vartriangle}^{\oplus}>0$, and a set $\{\overline{\mu}_1,...,\overline{\mu}_{\overline{\omega}}\}:=\overline{\boldsymbol{\mu}}\subseteq\mathbb{S}_n^0$ (including $\overline{\boldsymbol{\mu}}=\emptyset$) such that
\begin{equation}
\Upsilon(\theta)\succeq\lambda_{\vartriangle}^{\oplus}\sum_{k=1}^{\overline{\omega}}\amalg_{n,\overline{\mu}_k},
\end{equation}
where $\overline{\mu}_i<\overline{\mu}_j$, $\forall\ i<j$.
\begin{theorem}\label{4-2}
	Under Assumption \ref{2;2}, the following assertion holds: 
	\begin{equation}
	\mathbf{X}^{\top}\Sigma^d_{\epsilon}(0)\mathbf{X}\ge\overline{\varrho}_{\epsilon}\sum_{k=1}^{\overline{\omega}}\Bigl(X_{\overline{\mu}_k}^2+\sum_{j=2}^{\overline{\delta}_k}(\mathbf{H}_{[4]\overline{\mu}_k,j}^{(j)})^2\Bigr),
	\end{equation}
	where $\mathbf{X}$ is the same as in (\ref{3.5}), and $$\mathbf{H}_{[4]\overline{\mu}_k,j}=\Big(\sum_{i=0}^{j-1}(Q_{[4]\overline{\mu}_k,i}^{-1})^{\top}P_{[4]\overline{\mu}_k,i}\Big)J_{[4]\overline{\mu}_k}\mathbf{X},\ \ \ \forall\ j\in\mathbb{S}_{\overline{\delta}_k}^0,$$
	with $\overline{\delta}_k,\ J_{[4]\overline{\mu}_k},\ P_{[4]\overline{\mu}_k,i}$ and $Q_{[4]\overline{\mu}_k,i}$ defined in Algorithm \ref{4}. Moreover, $\overline{\varrho}_{\epsilon}\propto\epsilon\lambda_{\vartriangle}^{\oplus}\min_{k\in\mathbb{S}_{\overline{\omega}}^0}\big\{(\prod_{i=0}^{\overline{\delta}_k-1}\overline{a}_{\overline{\varsigma}_{k(i)},i}^{\lfloor[4]\overline{\mu}_k,i\rfloor})^2\big\}$. 
	\\
	\begin{threeparttable}
		\begin{algorithm}[H]\label{4}
			\caption{Algorithm for obtaining $\mathbf{H}_{[4]\overline{\mu}_k,j}\ (j\in\mathbb{S}_{\overline{\delta}_k}^0)$}
			\LinesNumbered
			\KwIn{\rm{$\boldsymbol{\Phi}_{D^{\star}}(\theta)$, $\overline{\boldsymbol{\mu}}$.}}
			\KwOut{\rm{$\mathbf{H}_{[4]\overline{\mu}_k,j}$\tnote{\textcolor{blue}{$\ddag$}}.}
			}			
			\rm{\textbf{(Initialization):} $\overline{A}_{[4]\overline{\mu}_k,1}=J_{[4]\overline{\mu}_k}\boldsymbol{\Phi}_{D^{\star}}(\theta)J_{[4]\overline{\mu}_k}^{-1}$, $\overline{\delta}_k:=1$\;}
			\rm{\textbf{(Recursive framework):} Construct a FOR loop similar to Algorithm \ref{3} to determine $\overline{\delta}_k$ and some $\overline{\varsigma}_{k(i)}\in\mathbb{S}_n^i\ (i\in\mathbb{S}_{\overline{\delta}_k-1}^0)$, which ensures
			\\
			\ \ \ \ \ \ \ \ \ \ \ \ \ \ \ \ \ \ $\text{(4-i)}\ \overline{a}_{\overline{\varsigma}_{k(i)},i}^{\lfloor[4]\overline{\mu}_k,i\rfloor}\neq 0,\ \ \forall\ i\in\mathbb{S}_{\overline{\delta}_k-1}^0,\ \ \ \text{(4-ii)}\ \overline{a}_{j,\overline{\delta}_k}^{\lfloor[4]\overline{\mu}_k,\overline{\delta}_k\rfloor}=0,\ \forall\ j\in\mathbb{S}_n^{\overline{\delta}_k+1}$\tnote{\textcolor{blue}{$\ddag$}},\\
			where each $\overline{a}_{ji}^{\lfloor[4]\overline{\mu}_k,i\rfloor}$ is obtained by the iterative scheme:
			\\
			\ \ \ \ \ \ \ \ \ \ \ \ \ \ \ \ \ \ $\widehat{A}_{[4]\overline{\mu}_k,i}:=P_{[4]\overline{\mu}_k,i}\overline{A}_{[4]\overline{\mu}_k,i}P_{[4]\overline{\mu}_k,i}^{-1}$\tnote{\textcolor{blue}{$\ddag$}},\ \ $\overline{A}_{[4]\overline{\mu}_k,i+1}:=Q_{[4]\overline{\mu}_k,i}\widehat{A}_{[4]\overline{\mu}_k,i}Q_{[4]\overline{\mu}_k,i}^{-1}$}\;
			\rm{\textbf{return}} $\overline{\delta}_k$, $J_{[4]\overline{\mu}_k}$, $P_{[4]\overline{\mu}_k,i}$, $Q_{[4]\overline{\mu}_k,i}.$
		\end{algorithm}
		\begin{tablenotes}
			\footnotesize
			\item[\rm{\textcolor{blue}{$\ddag$}}] \rm{$J_{[4]\overline{\mu}_k},\ P_{[4]\overline{\mu}_k,i}$ and $ Q_{[4]\overline{\mu}_k,i}$ have the same form as $J_{[1]\phi_k},\ P_{[1]\phi_k,i}$ and $ Q_{[1]\phi_k,i}$ by replacing $(\phi_k,\nu_{k(i)},\boldsymbol{\ell}_{[1]k,n-1-i})$ with $(\overline{\mu}_k,\overline{\varsigma}_{k(i)},\boldsymbol{\ell}_{[4]k,n-1-i})$, where $\boldsymbol{\ell}_{[4]k,n-1-i}=\frac{-1}{\widehat{a}_{i+1,i}^{\lfloor[4]\overline{\mu}_k,i\rfloor}}(\widehat{a}_{i+2,i}^{\lfloor[4]\overline{\mu}_k,i\rfloor},...,\widehat{a}_{n,i}^{\lfloor[4]\overline{\mu}_k,i\rfloor})^{\top}$. Similarly, we stipulate $\overline{a}_{\overline{\varsigma}_{k(0)},0}^{\lfloor[4]\overline{\mu}_k,0\rfloor}=1$ and $P_{\overline{\mu}_k,l}=Q_{\overline{\mu}_k,l}=\mathbf{I}_n$, $\forall\ l\in\{0,n-1\}$.}
		\end{tablenotes}
	\end{threeparttable}
\end{theorem}
\begin{proof}
By a standard argument in (\ref{3.43}) and (\ref{3.45}), there holds
\begin{equation*}
\Sigma^d_{\epsilon}(0)\succeq\epsilon\lambda_{\vartriangle}^{\oplus}\sum_{k=1}^{\overline{\omega}}\widehat{\Sigma}_{[4]\overline{\mu}_k,\epsilon},
\end{equation*} 
where $\widehat{\Sigma}_{[4]\overline{\mu}_k,\epsilon}\ (k\in\mathbb{S}_{\overline{\omega}}^0)$ are the solutions of the following Lyapunov equations, respectively:
\begin{equation*}
\Im_d\bigl(\widehat{\Sigma}_{[4]\overline{\mu}_k,\epsilon},\boldsymbol{\Phi}_{D^{\star}}(\theta),\amalg_{n,\overline{\mu}_k}\bigr)=\mathbb{O}.
\end{equation*}
The rest of the proof is similar to that of (\ref{3.12})-(\ref{3.34}), and is thus omitted. 	
\end{proof}
Finally, by Theorems \ref{4-1} and \ref{4-2}, we can obtain the following two results analogous to Proposition \ref{3&1} and Corollary \ref{3/1}. 
\begin{proposition}\label{4&1}
	Let Assumption \ref{2;2} holds. If (\ref{1.2}) is independent of $t$ (i.e., $f(t,\cdot)=f(\cdot)$ and $g_{ij}(t,\cdot)=g_{ij}(\cdot)$), then
	\begin{equation}
	\Im_c\Bigl(\Sigma_{\epsilon}^d(t),D^*,\epsilon\Lambda^*(\Lambda^*)^{\top}-\frac{d(\Sigma_{\epsilon}^d(t))}{dt}\Bigr)=\mathbb{O},
	\end{equation} 
	where 
	$$D^*=\bigg(\frac{\partial}{\partial(\ln x_j)}\Big(\frac{f_i(\mathbf{x})}{x_i}-\frac{\epsilon}{2}\sum_{j=1}^{N}g_{ij}^2(\mathbf{x})\Big)\bigg)_{n\times n}\bigl|_{\mathbf{x}=e^{\boldsymbol{\Psi}^*_{\epsilon}}},\ \ \ \Lambda^*=(g_{ij}(e^{\boldsymbol{\Psi}^*_{\epsilon}}))_{n\times N},$$
	with $\boldsymbol{\Psi}^*_{\epsilon}$ being the unique equilibrium of (\ref{2.4}).
\end{proposition}  
\begin{corollary}\label{4/1}
	Under Assumption \ref{2;2}, if one of the following four conditions holds: 
	$${\rm{(1)}}\ \omega=n,\ \ {\rm{(2)}}\ \overline{\delta}_{k_3}=n,\ \ \exists\ k_3\in\mathbb{S}_{\omega}^0,\ \ {\rm{(3)}}\ \overline{\omega}=n,\ \ {\rm{(4)}}\ \overline{\delta}_{k_4}=n,\ \ \exists\ k_4\in\mathbb{S}_{\overline{\omega}}^0.$$
	Then $\Sigma^d_{\epsilon}(t)\succ\mathbb{O}$.
\end{corollary}
\section{Applications}
This section presents some numerical examples to verify our theoretical results. A more biologically reasonable stochastic modeling assumption is provided simultaneously.

Consider a classic Logistic model
\begin{equation}\label{5.1}
     dx(t)=x(t)(\overline{r}-bx(t))dt,
\end{equation}
where $x(t)$ is the population level at time $t$, $b>0$ denotes the individual competition strength within species, and $\overline{r}$ represents the intrinsic growth rate. In practical terms, the growth rate is largely affected by various environmental variations such as soil erosion, predation, population migration, and epidemics. Therefore,  it is better to consider $\overline{r}$ as a random variable $r(t)$ that fluctuates around the mean $\overline{r}$, potentially even taking negative values. A common approach to model such continuous fluctuation is linear perturbation \cite{S49}, i.e., a linear function of Gaussian white noise:
\begin{equation}\label{5.2}
    r(t)=\overline{r}+\frac{\sigma_0dW(t)}{dt},
\end{equation}
where $\sigma_0^2>0$ represents the noise density of $W(t)$. Currently, the choice of linear perturbation has been extensively used by biologists for exploring the impact of stochastic mechanisms on population dynamics or virus infection; see \cite{S50,S51,S52,S53} and references therein. However, a notable limitation of such approach is that by defining $\langle r(T)\rangle$ as the time averages of $r(t)$ on time interval $[0,T]$, we obtain
\begin{equation*}
\langle r(T)\rangle=\frac{1}{T}\int_{0}^{T}r(t)dt=\overline{r}+\frac{\sigma_0W(T)}{T}\thicksim\mathbb{N}\Bigl(\overline{r},\frac{\sigma_0^2}{T}\Bigr).
\end{equation*}
Intuitively, the variance $\frac{\sigma_0^2}{T}$ tends to infinity as $T\rightarrow 0$. This implies that the fluctuation of $r(t)$ will become very large in a sufficiently small time interval, which is inconsistent with facts. More precisely, successive averages of the process of linear perturbation experience large random oscillations \cite{S54}. To eliminate the conceptual and practical difficulties associated with linear perturbation, a possible way is to consider the key parameters in a randomly varying environment as mean-reverting processes \cite{S55}. A classic mean-reverting process is the Ornstein--Uhlenbeck process which, for the growth rate $r(t)$, has the form: 
\begin{equation}\label{5.3}
dr(t)=m_0\bigl(\overline{r}-r(t)\bigr)dt+\sigma_0dW(t),
\end{equation}
where $W(t)$ and $\sigma_0$ are the same as in (\ref{5.2}). $m_0>0$ denotes the speed of reversion, which satisfies $m_0\neq\overline{r}$. By a standard argument in \cite{S54}, we have $r(t)\thicksim\mathbb{N}\bigl(\mathbb{E}(r(t)),\mathbb{V}\text{ar}(r(t))\bigr)$, where $\mathbb{E}(r(t))=\overline{r}+(r(0)-\overline{r})e^{-m_0t}$ and $\mathbb{V}\text{ar}(r(t))=\frac{\sigma_0^2}{2m_0}(1-e^{-2m_0t})$. By calculation, 
\begin{equation*}
\lim_{t\rightarrow 0}\mathbb{E}(r(t))=\overline{r},\ \ \ \lim_{t\rightarrow 0}\mathbb{V}\text{ar}(r(t))=0,
\end{equation*}
and 
\begin{equation*}
\mathbb{V}\text{ar}(\langle r(T)\rangle)=\frac{\sigma_0^2}{3}T+O(T^2),\ \ \ \text{for\ some\ small}\ T.
\end{equation*}
Clearly, the fluctuation of $r(t)$ in (\ref{5.3}) is sufficiently small in a small time interval. The study of biological models motivated by Ornstein--Uhlenbeck processes have attracted much attention in recent two years. For related interesting work, we refer to the treatises \cite{S56,S57,S58,S59}, and references therein. Along this line, we further assume that (\ref{5.3}) is perturbed by small diffusion and periodically varying environment. Inspired by the above, we consider the following stochastic population model: 
\begin{equation}\label{5.4}
\begin{cases}
dx_{\epsilon}(t)=x_{\epsilon}(t)\big(r_{\epsilon}(t)-bx_{\epsilon}(t)\big)dt,\\
dr_{\epsilon}(t)=m_0\bigl(\overline{r}-r_{\epsilon}(t)\bigr)dt+\sqrt{\epsilon}\sigma(t)dW(t),
\end{cases}
\end{equation}
where $\sigma(t)$ is a periodic function with period $\theta$.

Below we aim to derive the approximate SPSD using our main results and assess its approximation effect on the realistic one. Direct calculation shows that $\Gamma^{\star}(t)=\text{diag}\{0,\sigma(t)\}$, $\mathbf{X}^{\star}(t)=(\frac{\overline{r}}{b},\overline{r})^{\top}$, and 
$$C^{\star}(t)=\left(\begin{array}{cc}
-\overline{r} & \frac{\overline{r}}{b} \\ 
0 & -m_0
\end{array}\right):=C^*,\ \ \ \ (\text{constant\ matrix})$$
which means $\boldsymbol{\Phi}_{C^{\star}}(t)=e^{C^*t}$ (see Proposition \ref{3&1}). Note that $C^*$ has two different eigenvalues $-\overline{r}$ and $-m_0$, Assumption \ref{2;1} is then satisfied by (\ref{5.4}). Moreover, an application of Cayley--Hamilton theorem yields
\begin{equation*}
\boldsymbol{\Phi}_{C^{\star}}(t)=\left(\begin{array}{cc}
e^{-\overline{r}t} & \frac{\overline{r}(e^{-m_0t}-e^{-\overline{r}t})}{b(\overline{r}-m_0)} \\ 
0 & e^{-m_0t}
\end{array}\right),\ \ \ \boldsymbol{\Phi}^{-1}_{C^{\star}}(t)=\left(\begin{array}{cc}
e^{\overline{r}t} & \frac{\overline{r}(e^{m_0t}-e^{\overline{r}t})}{b(\overline{r}-m_0)} \\ 
0 & e^{m_0t}
\end{array}\right).
\end{equation*}
In what follows, for simplicity, we assume $\sigma(t)=\sigma_0\sin(\frac{2\pi t}{\theta})$, where $\sigma_0>0$ is a constant.
Now let $$\nabla_0(t)=\int_{0}^{t}\big(\boldsymbol{\Phi}^{-1}_{C^{\star}}(s)\Gamma^{\star}(s)\big)\big(\boldsymbol{\Phi}^{-1}_{C^{\star}}(s)\Gamma^{\star}(s)\big)^{\top}ds:=\left(\begin{array}{cc}
\hbar_{11}(t) & \hbar_{12}(t) \\ 
\hbar_{12}(t) & \hbar_{22}(t)
\end{array}\right).$$
By complex calculation, one gets
\begin{equation*}
\begin{cases}
\hbar_{11}(t)=\dfrac{\sigma_0^2\overline{r}^2}{4b^2(\overline{r}-m_0)^2}\bigg\{\Big[e^{2m_0t}\Big(\dfrac{1}{m_0}-\dfrac{m_0\theta^2\cos(\frac{4\pi t}{\theta})+2\pi\theta\sin(\frac{4\pi t}{\theta})}{m_0^2\theta^2+4\pi^2}\Big)-\dfrac{4\pi^2}{m_0(m_0^2\theta^2+4\pi^2)}\Big]\\
\ \ \ \ \ \ \ \ \ \ \ +\Big[e^{2\overline{r}t}\Big(\dfrac{1}{\overline{r}}-\dfrac{\overline{r}\theta^2\cos(\frac{4\pi t}{\theta})+2\pi\theta\sin(\frac{4\pi t}{\theta})}{\overline{r}^2\theta^2+4\pi^2}\Big)-\dfrac{4\pi^2}{\overline{r}(\overline{r}^2\theta^2+4\pi^2)}\Big]\\
\ \ \ \ \ \ \ \ \ \ \  -4\Big[e^{(\overline{r}+m_0)t}\Big(\dfrac{1}{\overline{r}+m_0}-\dfrac{(\overline{r}+m_0)\theta^2\cos(\frac{4\pi t}{\theta})+4\pi\theta\sin(\frac{4\pi t}{\theta})}{(\overline{r}+m_0)^2\theta^2+16\pi^2}\Big)-\dfrac{16\pi^2}{(\overline{r}+m_0)((\overline{r}+m_0)^2\theta^2+16\pi^2)}\Big]\bigg\},\\
\hbar_{12}(t)=\dfrac{\sigma_0^2\overline{r}}{4b(\overline{r}-m_0)}\bigg\{\Big[e^{2m_0t}\Big(\dfrac{1}{m_0}-\dfrac{m_0\theta^2\cos(\frac{4\pi t}{\theta})+2\pi\theta\sin(\frac{4\pi t}{\theta})}{m_0^2\theta^2+4\pi^2}\Big)-\dfrac{4\pi^2}{m_0(m_0^2\theta^2+4\pi^2)}\Big]\\
\ \ \ \ \ \ \ \ \ \ \ -2\Big[e^{(\overline{r}+m_0)t}\Big(\dfrac{1}{\overline{r}+m_0}-\dfrac{(\overline{r}+m_0)\theta^2\cos(\frac{4\pi t}{\theta})+4\pi\theta\sin(\frac{4\pi t}{\theta})}{(\overline{r}+m_0)^2\theta^2+16\pi^2}\Big)-\dfrac{16\pi^2}{(\overline{r}+m_0)((\overline{r}+m_0)^2\theta^2+16\pi^2)}\Big]\bigg\},\\
\hbar_{22}(t)=\dfrac{\sigma_0^2}{4}\Big[e^{2m_0t}\Big(\dfrac{1}{m_0}-\dfrac{m_0\theta^2\cos(\frac{4\pi t}{\theta})+2\pi\theta\sin(\frac{4\pi t}{\theta})}{m_0^2\theta^2+4\pi^2}\Big)-\dfrac{4\pi^2}{m_0(m_0^2\theta^2+4\pi^2)}\Big].
\end{cases}
\end{equation*}
In particular, if $t=\theta$, then
\begin{equation*}
\begin{cases}
\hbar_{11}(\theta)=\dfrac{(\sigma_0\overline{r}\pi)^2}{b^2(\overline{r}-m_0)^2}\Big[\dfrac{e^{2m_0\theta}-1}{m_0(m_0^2\theta^2+4\pi^2)}+\dfrac{e^{2\overline{r}\theta}-1}{\overline{r}(\overline{r}^2\theta^2+4\pi^2)}-\dfrac{16(e^{(\overline{r}+m_0)\theta}-1)}{(\overline{r}+m_0)((\overline{r}+m_0)^2\theta^2+16\pi^2)}\Big],\\
\hbar_{12}(\theta)=\dfrac{\sigma_0^2\pi^2\overline{r}}{b(\overline{r}-m_0)}\Big[\dfrac{e^{2m_0\theta}-1}{m_0(m_0^2\theta^2+4\pi^2)}-\dfrac{8(e^{(\overline{r}+m_0)\theta}-1)}{(\overline{r}+m_0)((\overline{r}+m_0)^2\theta^2+16\pi^2)}\Big],\\
\hbar_{22}(\theta)=\dfrac{\sigma_0^2\pi^2(e^{2m_0\theta}-1)}{m_0(m_0^2\theta^2+4\pi^2)}.
\end{cases}
\end{equation*}
As shown in (\ref{3.9}), $\Sigma^c_{\epsilon}(0)$ is determined by
\begin{equation}\label{5.5}
\Im_d\Big(\Sigma^c_{\epsilon}(0),\boldsymbol{\Phi}_{C^{\star}}(\theta),\epsilon\boldsymbol{\Phi}_{C^{\star}}(\theta)\nabla_0(\theta)\big(\boldsymbol{\Phi}_{C^{\star}}(\theta)\big)^{\top}dt\Big)=\mathbb{O}.
\end{equation}
Using Algorithm \ref{1}, we obtain that $\eta_1=2$ (i.e., $\Sigma^c_{\epsilon}(0)\succ\mathbb{O}$) holds if and only if $\nabla_0(\theta)\succ\mathbb{O}$. In fact,
by virtue of the H$\ddot{\text{o}}$lder's inequality, we determine
\begin{equation*}
\begin{split}
|\nabla_0(\theta)|=&(\sigma_0\pi)^2\bigg[\frac{e^{2m_0\theta}-1}{m_0(m_0^2\theta^2+4\pi^2)}+\frac{e^{2\overline{r}\theta}-1}{\overline{r}(\overline{r}^2\theta^2+4\pi^2)}-\Big(\frac{e^{(\overline{r}+m_0)\theta}-1}{(\overline{r}+m_0)((\overline{r}+m_0)^2\theta^2+16\pi^2)}\Big)^2\bigg]\\
=&(\sigma_0\pi)^2\bigg[\int_{0}^{\theta}e^{2m_0t}\sin^2\Big(\frac{2\pi t}{\theta}\Big)dt+\int_{0}^{\theta}e^{2\overline{r}t}\sin^2\Big(\frac{2\pi t}{\theta}\Big)dt\\
&\ \ \ \ \ \ \ \ \ \ -\Big(\int_{0}^{\theta}e^{(\overline{r}+m_0)t}\sin^2\Big(\frac{2\pi t}{\theta}\Big)dt\Big)^2\bigg]>0.
\end{split}
\end{equation*}
This combined with Corollary \ref{3/1} yields $\Sigma^c_{\epsilon}(t)\succ\mathbb{O}$, $\forall\ t\ge 0$.

Below we derive the explicit form of $\Sigma^c_{\epsilon}(\cdot)$. Let $\Sigma^c_{\epsilon}(\cdot)_{(i,j)}$ be the $i$th element of the $j$th row of $\Sigma^c_{\epsilon}(\cdot)$ for convenience, $\forall\ i,j\in\mathbb{S}_2^0$. By Algorithm \ref{1}, a complex calculation for the definition of $\mathcal{G}_0$ and (\ref{5.5}) leads to
\begin{equation}\label{5.6}
\begin{cases}
\Sigma^c_{\epsilon}(0)_{(1,1)}=\dfrac{\epsilon}{(1-e^{-2\overline{r}\theta})(1-e^{-(\overline{r}+m_0)\theta})}\bigg[\dfrac{e^{2m_0\theta}(\sigma_0\overline{r}\pi)^2(1+e^{-(\overline{r}+m_0)\theta})(e^{-m_0\theta}-e^{-\overline{r}\theta})^2}{m_0b^2(\overline{r}-m_0)^2(m_0^2\theta^2+4\pi^2)}\\
\ \ \ \ \ \ \ \ \ \ \ \ \ \ \ \ \ +e^{-\overline{r}\theta}\Big(e^{-\overline{r}\theta}\big(1-e^{-(\overline{r}+m_0)\theta}\big)\hbar_{11}(\theta)+\dfrac{2\overline{r}(e^{-m_0\theta}-e^{-\overline{r}\theta})\hbar_{12}(\theta)}{b(\overline{r}-m_0)}\Big)\bigg],\\
\Sigma^c_{\epsilon}(0)_{(1,2)}=\dfrac{\epsilon}{(1-e^{-(\overline{r}+m_0)\theta})}\bigg[e^{-(\overline{r}+m_0)\theta}\hbar_{12}(\theta)+\dfrac{\overline{r}\sigma_0^2\pi^2e^{m_0\theta}(e^{-m_0\theta}-e^{-\overline{r}\theta})}{bm_0(\overline{r}-m_0)(m_0^2\theta^2+4\pi^2)}\bigg],\\
\Sigma^c_{\epsilon}(0)_{(2,2)}=\dfrac{\epsilon\sigma_0^2\pi^2}{m_0(m_0^2\theta^2+4\pi^2)}.
\end{cases}
\end{equation}
Combining Theorem \ref{3-1} yields
\begin{equation}\label{5.7}
\begin{cases}
\Sigma^c_{\epsilon}(t)_{(1,1)}=e^{-2\overline{r}t}\Big(\Sigma^c_{\epsilon}(0)_{(1,1)}+\epsilon\hbar_{11}(t)\Big)+\dfrac{2\overline{r}e^{-\overline{r}t}(e^{-m_0t}-e^{-\overline{r}t})}{b(\overline{r}-m_0)}\Big(\Sigma^c_{\epsilon}(0)_{(1,2)}+\epsilon\hbar_{12}(t)\Big)\\
\ \ \ \ \ \ \ \ \ \ \ \ \ \ \ +\dfrac{\overline{r}^2(e^{-m_0t}-e^{-\overline{r}t})^2}{b^2(\overline{r}-m_0)^2}\Big(\Sigma^c_{\epsilon}(0)_{(2,2)}+\epsilon\hbar_{22}(t)\Big),\\
\Sigma^c_{\epsilon}(t)_{(1,2)}=e^{-m_0t}\Big[e^{-\overline{r}t}\Big(\Sigma^c_{\epsilon}(0)_{(1,2)}+\epsilon\hbar_{12}(t)\Big)+\dfrac{\overline{r}(e^{-m_0t}-e^{-\overline{r}t})}{b(\overline{r}-m_0)}\Big(\Sigma^c_{\epsilon}(0)_{(2,2)}+\epsilon\hbar_{22}(t)\Big)\Big],\\
\Sigma^c_{\epsilon}(t)_{(2,2)}=e^{-2m_0t}\Big(\Sigma^c_{\epsilon}(0)_{(2,2)}+\epsilon\hbar_{22}(t)\Big).
\end{cases}
\end{equation}
As a summary, by Theorem \ref{3-1} we obtain
\begin{itemize}
	\item[($\boldsymbol{\otimes}$-1)] For sufficiently small $\epsilon>0$, system (\ref{5.4}) approximately has a local SPS $(\mathbf{X}^{\triangleright}_{\epsilon}(t))$ which follows the distribution $\mathbb{N}_2((\frac{\overline{r}}{b},\overline{r})^{\top},\Sigma^c_{\epsilon}(t))$, where $\Sigma^c_{\epsilon}(t)$ is defined in (\ref{5.7}).
\end{itemize}
\begin{remark}\label{55}
	{\rm{Before proceeding further, we have some comments:
	\begin{itemize}
	\item In fact, we can mimic the proofs of \cite[Theorem 3.1]{S60} and \cite[Theorems 4.1 and 5.1]{S20} to obtain that system (\ref{5.4}) has a unique stochastic $\theta$-periodic solution if $\overline{r}>0$. Moreover, such $\theta$-periodic solution is globally attractive.
 \item SPS denotes a long-time, steady state of a stochastic process involving the periodicity, and is an intuitive reflection of the tendency that the distribution of system states gradually presents periodic changes. However, the existence and form of such a periodic solution cannot be directly verified due to the features of a computer simulation, such as finite iterations and single sample paths the finite number of iterations and a single sample path. To address it and study the approximation effect of our algorithms explicitly, we use enough computer simulations and sufficient iterations as a viable alternative for the SPS of (\ref{5.4}). Moreover, the Monte Carlo numerical method is carried out. The main idea is listed below:
	\begin{itemize}
	\item[(i)] Let $\mathbf{Num}$ be the total number of simulation, for each simulation test $i\in\mathbb{S}_{\mathbf{Num}}^0$, based on Milstein's higher-order method \cite{S61}, we consider the discretization equations of (\ref{5.4}) on time interval $[0,T]$ with one step size $\Delta t$. Denote by $(x^{(i)}_{\epsilon}(k\Delta t),r^{(i)}_{\epsilon}(k\Delta t))^{\top}$ the value of the $k$th iteration of the simulation test $i$. According to ($\boldsymbol{\otimes}$-1), the initial value  $(x_{\epsilon}^{(i)}(0),r^{(i)}_{\epsilon}(0))^{\top}\thicksim\mathbb{N}_2((\frac{\overline{r}}{b},\overline{r})^{\top},\Sigma^c_{\epsilon}(0))$, $\forall\ i\in\mathbb{S}_{\mathbf{Num}}^0$.
	\item[(ii)] For convenience, we only choose the iteration values at the unit time, i.e., $(x^{(i)}_{\epsilon}(j),r^{(i)}_{\epsilon}(j))^{\top}$, $\forall\ j\in\mathbb{S}_T^0;i\in\mathbb{S}_{\mathbf{Num}}^0$. In this sense, the statistical mean and covariance of the iteration value at the integral time are
	\begin{equation*}
	\begin{cases}
	\overline{\text{M}}_x(j,\mathbf{Num})=\dfrac{\sum_{i\in\mathbb{S}_{\mathbf{Num}}^0}x^{(i)}_{\epsilon}(j)}{\mathbf{Num}},\ \ \overline{\text{Cov}}_{11}(j,\mathbf{Num})=\dfrac{\sum_{i\in\mathbb{S}_{\mathbf{Num}}^0}(x^{(i)}_{\epsilon}(j)-\overline{\text{M}}_x(j,\mathbf{Num}))^2}{\mathbf{Num}-1},\\ \overline{\text{M}}_r(j,\mathbf{Num})=\dfrac{\sum_{i\in\mathbb{S}_{\mathbf{Num}}^0}r^{(i)}_{\epsilon}(j)}{\mathbf{Num}},\ \ 
	\overline{\text{Cov}}_{22}(j,\mathbf{Num})=\dfrac{\sum_{i\in\mathbb{S}_{\mathbf{Num}}^0}(r^{(i)}_{\epsilon}(j)-\overline{\text{M}}_r(j,\mathbf{Num}))^2}{\mathbf{Num}-1},\\ \overline{\text{Cov}}_{12}(j,\mathbf{Num})=\dfrac{\sum_{i\in\mathbb{S}_{\mathbf{Num}}^0}(x^{(i)}_{\epsilon}(j)-\overline{\text{M}}_x(j,\mathbf{Num}))(r^{(i)}_{\epsilon}(j)-\overline{\text{M}}_r(j,\mathbf{Num}))}{\mathbf{Num}-1},
	\end{cases}
	\end{equation*}
	where $j\in\mathbb{S}_T^0$. Then the index $(\overline{\text{M}}_x(\cdot),\overline{\text{M}}_r(\cdot),\overline{\text{Cov}}_{11}(\cdot),\overline{\text{Cov}}_{12}(\cdot),\overline{\text{Cov}}_{22}(\cdot))$ is a good alternative for the mean and covariance of the underlying SPSD of (\ref{5.4}) at the unit time for some large $\mathbf{Num}$. Certainly, such alternative is more viable as $\mathbf{Num}$ increases.
	\item[(iii)] We define $(\overline{\mathbb{A}\text{ee}}_x,\overline{\mathbb{A}\text{ee}}_r)$ and $(\overline{\mathbb{A}\text{ev}}_{11},\overline{\mathbb{A}\text{ev}}_{12},\overline{\mathbb{A}\text{ev}}_{22})$ as the average relative error between the SPS of (\ref{5.4}) and the solution $(\mathbf{X}^{\triangleright}_{\epsilon}(t))$ regarding the mean and covariance. In view of (ii), then
	\begin{equation*}
	\overline{\mathbb{A}\text{ee}}_x=\frac{1}{T}\sum_{j\in\mathbb{S}_T^0}\frac{|\overline{\text{M}}_x(j,\mathbf{Num})-\frac{\overline{r}}{b}|}{\overline{\text{M}}_x(j,\mathbf{Num})},\ \ \overline{\mathbb{A}\text{ee}}_r=\frac{1}{T}\sum_{j\in\mathbb{S}_T^0}\frac{|\overline{\text{M}}_r(j,\mathbf{Num})-\overline{r}|}{\overline{\text{M}}_r(j,\mathbf{Num})},
	\end{equation*}
	\begin{equation*}
	\overline{\mathbb{A}\text{ev}}_{11}=\frac{1}{T}\sum_{j\in\mathbb{S}_T^0}\frac{|\overline{\text{Cov}}_{11}(j,\mathbf{Num})-\Sigma^c_{\epsilon}(j)_{(1,1)}|}{\overline{\text{Cov}}_{11}(j,\mathbf{Num})},\ \ \overline{\mathbb{A}\text{ev}}_{22}=\frac{1}{T}\sum_{j\in\mathbb{S}_T^0}\frac{|\overline{\text{Cov}}_{22}(j,\mathbf{Num})-\Sigma^c_{\epsilon}(j)_{(2,2)}|}{\overline{\text{Cov}}_{22}(j,\mathbf{Num})},
	\end{equation*}
	\begin{equation*}
	\overline{\mathbb{A}\text{ev}}_{12}=\frac{1}{T}\sum_{j\in\mathbb{S}_T^0}\frac{|\overline{\text{Cov}}_{12}(j,\mathbf{Num})-\Sigma^c_{\epsilon}(j)_{(1,2)}|}{\overline{\text{Cov}}_{12}(j,\mathbf{Num})},
	\end{equation*}
	which are clearly established on time interval $[0,T]$.
	\end{itemize}
	\end{itemize}}}
	\end{remark}
Below we provide a numerical example for illustration. We choose $T=400$ and $\Delta t=10^{-3}$. 
\begin{example}\label{5:1}
	{\rm{Consider (\ref{5.4}) with parameters $\overline{r}=0.5$, $m_0=0.3$, $b=0.5$ and $\sigma_0=0.1$. We focus on the following different combinations of noise intensity $\epsilon$ and periodic $\theta$:
	\begin{equation*}
		(\text{I})\ (\epsilon,\theta)=(0.01,50),\ \ (\text{II})\ (\epsilon,\theta)=(0.01,100),\ \ (\text{III})\ (\epsilon,\theta)=(0.05,100),\ \ (\text{IV})\ (\epsilon,\theta)=(0.1,100). 	
	\end{equation*}
	For case (I), a direct calculation for (\ref{5.6}) shows
	\begin{equation*}
	\Sigma^c_{\epsilon}(0)=10^{-4}\times\left(\begin{array}{cc}
	1.04524 & 0.34123 \\ 
	0.34123 & 0.12439
	\end{array} \right). 
	\end{equation*}
	Using (\ref{5.7}), we first plot the functions $x^{\star}(t)(=\frac{\overline{r}}{b})$, $r^{\star}(t)(=\overline{r})$ and $\Sigma^c_{\epsilon}(t)_{(i,j)}\ (i,j\in\mathbb{S}_2^0)$ on $t\in[0,200]$, as shown in the purple dotted lines of Fig. 2. By Remark \ref{55}, Fig. 2 further presents the variation trends of $\overline{\text{M}}_x(\cdot,\cdot)$, $\overline{\text{M}}_r(\cdot,\cdot)$ and $\overline{\text{Cov}}_{ij}(\cdot,\cdot)\ (i,j\in\mathbb{S}_2^0)$ at the simulation number $\mathbf{Num}=10^3$, $10^4$ and $10^5$, each in a different color. Obviously,
	the above functions $\Sigma^c_{\epsilon}(t)_{(i,j)}\ (i,j\in\mathbb{S}_2^0)$ all almost coincide with the corresponding three fitting curves. In addition, the function $\overline{\text{M}}_x(\cdot,\mathbf{Num})$ (resp., $\overline{\text{M}}_r(\cdot,\mathbf{Num})$) gradually approaches $x^{\star}(\cdot)$ (resp., $r^{\star}(\cdot)$) as $\mathbf{Num}$ increases. Thus, ($\boldsymbol{\otimes}$-1) is well verified, i.e., the SPSD of (\ref{5.4}) can be globally approximated by $(\mathbf{X}^{\triangleright}_{\epsilon}(t))$. To measure the similarity quantitatively, Table 2 shows the corresponding values of $\overline{\mathbb{A}\text{ee}}_x$, $\overline{\mathbb{A}\text{ee}}_r$ and $\overline{\mathbb{A}\text{ev}}_{ij}\ (i,j\in\mathbb{S}_2^0)$ at different simulation numbers. Intuitively, all the average relative errors inspected when $\mathbf{Num}\ge 10^4$ are less than 2\%. In this sense, we further use the Kolmogorov--Smirnov test \cite{S62} to test the alternative hypothesis that the numerical probability distribution of (\ref{5.4}) under $\mathbf{Num}=10^5$ and the distribution $\mathbb{N}_2((\frac{\overline{r}}{b},\overline{r})^{\top},\Sigma^c_{\epsilon}(t))$ are from different distributions against the null hypothesis that they are from the same distribution for each component. With 2\% significance level, the relevant tests imply that we cannot reject the null hypothesis. Hence, the similarity between the solution $(\mathbf{X}^{\triangleright}_{\epsilon}(t))$ and the underlying exact one of (\ref{5.4}) is significant.}}
\begin{figure}[H]
	\begin{center}
		\resizebox{18cm}{9.6cm}{\includegraphics{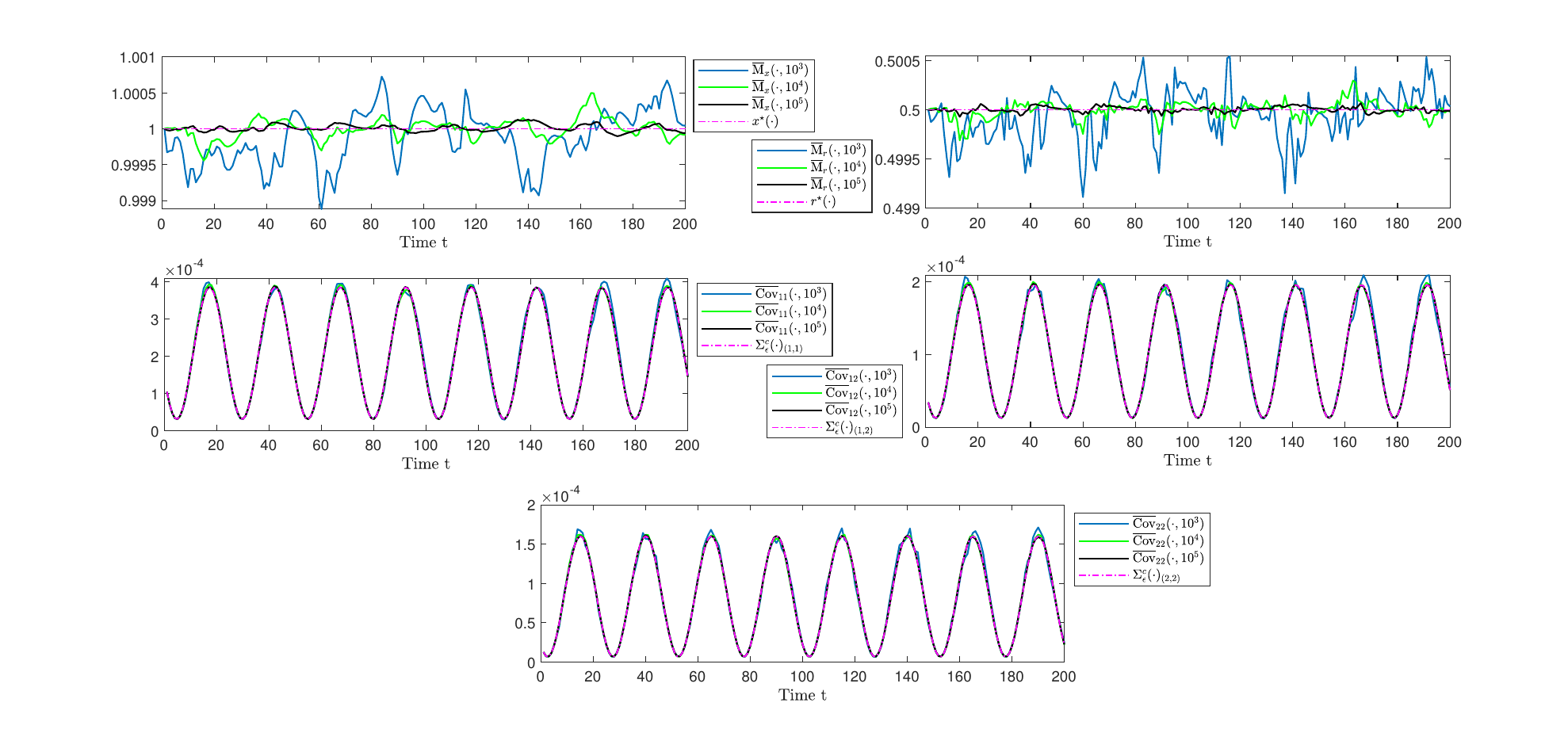}}
	\end{center}
	\makeatletter\def\@captype{figure}\makeatother \caption{The blue, green and black lines represent the functions $\overline{\text{M}}_x(\cdot)$, $\overline{\text{M}}_r(\cdot)$, $\overline{\text{Cov}}_{ij}(\cdot)\ (i,j\in\mathbb{S}_2^0)$ of (\ref{5.4}) at simulation number equals to $10^3$, $10^4$ and $10^5$, respectively. The functions $x^{\star}(\cdot)$, $r^{\star}(\cdot)$ and $\Sigma^c_{\epsilon}(\cdot)_{(i,j)}\ (i,j\in\mathbb{S}_2^0)$ are plotted by the corresponding purple dotted lines. Fixed parameters: $\epsilon=0.01$, $\theta=50$.}
\end{figure}
\begin{table}[hthp]
	\caption{List of values of $\overline{\mathbb{A}\text{ee}}_x$, $\overline{\mathbb{A}\text{ee}}_r$ and $\overline{\mathbb{A}\text{ev}}_{ij}\ (i,j\in\mathbb{S}_2^0)$ at different simulation numbers and parameters $(\epsilon,\theta)$}
	\label{table:para}
	\begin{small}
		\begin{center}
			\begin{tabular}{llllll}
				\hline\noalign{\smallskip}
				$((\epsilon,\theta),\mathbf{Num})$ & $\overline{\mathbb{A}\text{ee}}_x$ & $\overline{\mathbb{A}\text{ee}}_r$ & $\overline{\mathbb{A}\text{ev}}_{11}$ & $\overline{\mathbb{A}\text{ev}}_{12}$ & $\overline{\mathbb{A}\text{ev}}_{22}$\\
				\noalign{\smallskip}\hline\noalign{\smallskip}
				$(\text{case\ (I)},10^3)$ & $0.0341042\%$ & $0.0403638\%$ & $3.448\%$ & $4.067\%$ & $3.484\%$\\
				$(\text{case\ (I)},10^4)$ & $0.00933391\%$ & $0.01162\%$ & $1.031\%$ & $1.162\%$ & $1.087\%$\\
				$(\text{case\ (I)},10^5)$ & $0.00360542\%$ & $0.00391509\%$ & $0.289\%$ & $0.343\%$ & $0.319\%$\\
				$(\text{case\ (II)},10^3)$ & $0.0277649\%$ & $0.0358883\%$ & $3.683\%$ & $4.147\%$ & $3.543\%$\\
				$(\text{case\ (II)},10^4)$ & $0.0111913\%$ & $0.0133041\%$ & $1.351\%$ & $1.485\%$ & $1.232\%$\\
				$(\text{case\ (II)},10^5)$ & $0.00320987\%$ & $0.00386201\%$ & $0.370\%$ & $0.407\%$ & $0.362\%$\\
				\noalign{\smallskip}\hline
			\end{tabular}
		\end{center}
	\end{small}
\end{table}
{\rm{For case (II), we compute
		\begin{equation*}
		\Sigma^c_{\epsilon}(0)=10^{-5}\times\left(\begin{array}{cc}
		 3.17278 & 0.99841 \\ 
		 0.99841 & 0.35018
		\end{array} \right). 
		\end{equation*}
		Combined with (\ref{5.7}) and Remark \ref{55}, Fig. 3 shows the graphs of the key functions on $t\in[0,200]$, including: (i) $x^{\star}(\cdot)$, $r^{\star}(\cdot)$ and $\Sigma^c_{\epsilon}(\cdot)_{(i,j)}\ (i,j\in\mathbb{S}_2^0)$, all in the red dotted lines; (ii) $\overline{\text{M}}_x(\cdot,\cdot)$, $\overline{\text{M}}_r(\cdot,\cdot)$ and $\overline{\text{Cov}}_{ij}(\cdot,\cdot)\ (i,j\in\mathbb{S}_2^0)$ at $\mathbf{Num}=10^3$, $10^4$ and $10^5$, each in a different color. It is clear that the function $\overline{\text{M}}_x(\cdot,\mathbf{Num})$ (resp., $\overline{\text{M}}_r(\cdot,\mathbf{Num})$) gradually approaches $x^{\star}(\cdot)$ (resp., $r^{\star}(\cdot)$) as $\mathbf{Num}$ increases. Furthermore, the functions $\Sigma^c_{\epsilon}(t)_{(i,j)}\ (i,j\in\mathbb{S}_2^0)$ all almost coincide with the corresponding three fitting curves. These verifies ($\boldsymbol{\otimes}$-1) and Theorem \ref{3-1} well. To support the theoretical results deeply, we further provide in Table 2 the corresponding values of $\overline{\mathbb{A}\text{ee}}_x$, $\overline{\mathbb{A}\text{ee}}_r$ and $\overline{\mathbb{A}\text{ev}}_{ij}\ (i,j\in\mathbb{S}_2^0)$ under different $\mathbf{Num}$. Clearly, all the average relative errors inspected at $\mathbf{Num}=10^5$ (resp., $10^4$) are less than $0.5\%$ (resp., $2\%$). Based on the Kolmogorov--Smirnov test, we further determine that the null hypothesis that the distribution $\mathbb{N}_2((\frac{\overline{r}}{b},\overline{r})^{\top},\Sigma^c_{\epsilon}(t))$ and the numerical probability distribution of (\ref{5.4}) under $\mathbf{Num}=10^5$ are from the same distribution cannot be rejected with 2\% significance level. Hence, the solution $(\mathbf{X}^{\triangleright}_{\epsilon}(t))$ approximates the exact stochastic $\theta$-periodic solution of (\ref{5.4}) well, which validates our approach.}}
\begin{figure}[H]
	\begin{center}
		\resizebox{18cm}{9.6cm}{\includegraphics{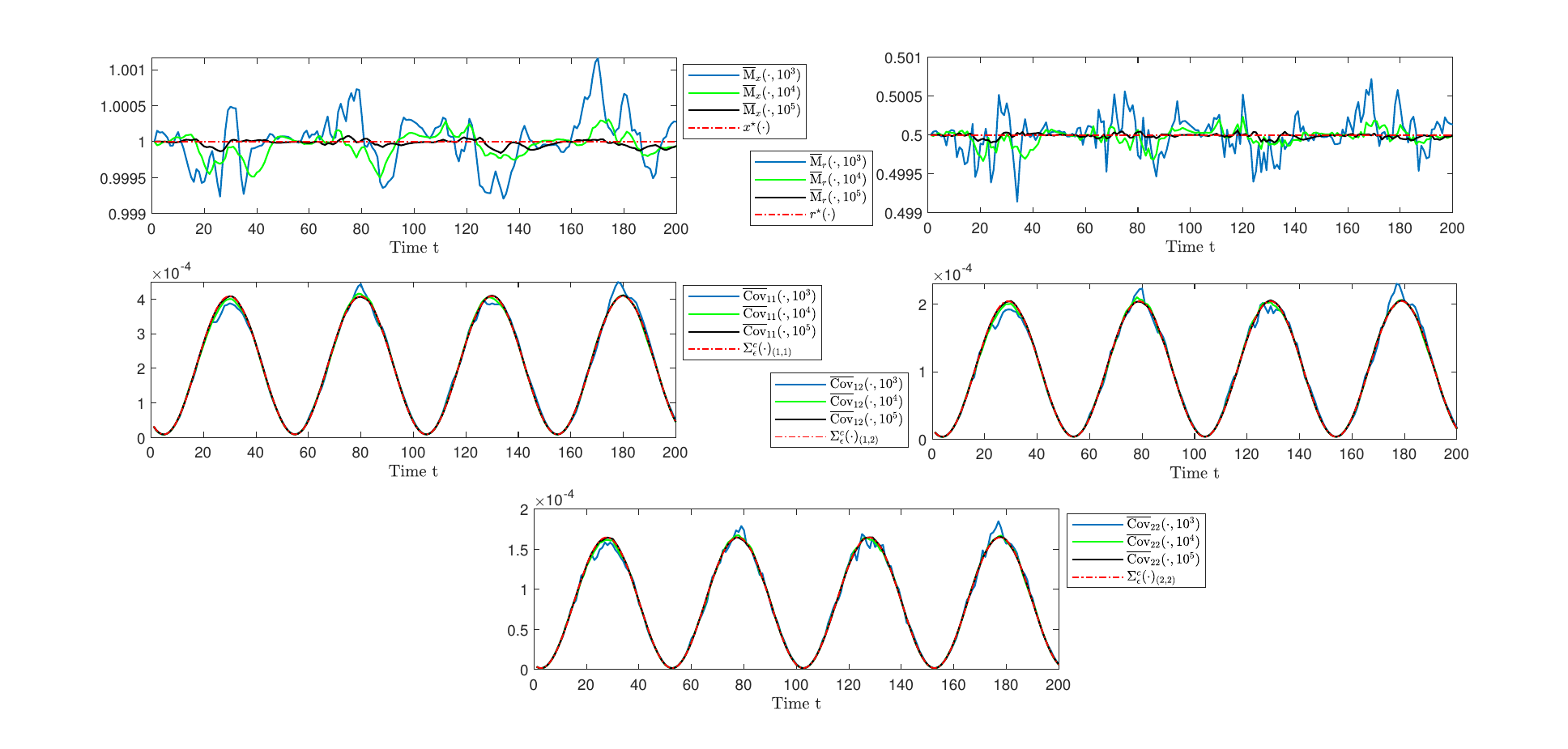}}
	\end{center}
	\makeatletter\def\@captype{figure}\makeatother \caption{The blue, green and black lines denote the functions $(\overline{\text{M}}_x(\cdot),\overline{\text{M}}_r(\cdot),\overline{\text{Cov}}_{11}(\cdot),\overline{\text{Cov}}_{12}(\cdot),\overline{\text{Cov}}_{22}(\cdot))$ of (\ref{5.4}) at $\mathbf{Num}=10^3$, $10^4$ and $10^5$, respectively. The red dotted lines depict the functions $x^{\star}(\cdot)$, $r^{\star}(\cdot)$ and $\Sigma^c_{\epsilon}(\cdot)_{(i,j)}\ (i,j\in\mathbb{S}_2^0)$ on $t\in[0,200]$. Fixed parameters: $\epsilon=0.01$, $\theta=100$.}
\end{figure}
{\rm{For case (III) (resp., (IV)), by a similar argument in case (I), Fig. 4 (resp., Fig. 5) presents the relationship between the functions $(\overline{\text{M}}_x(\cdot),\overline{\text{M}}_r(\cdot),\overline{\text{Cov}}_{11}(\cdot),\overline{\text{Cov}}_{12}(\cdot),\overline{\text{Cov}}_{22}(\cdot))$ and $(x^{\star}(\cdot),r^{\star}(\cdot),\Sigma^c_{\epsilon}(\cdot)_{(1,1)},\Sigma^c_{\epsilon}(\cdot)_{(1,2)},\Sigma^c_{\epsilon}(\cdot)_{(2,2)})$ on $t\in[0,200]$ at different simulation numbers. Table 3 shows the corresponding values of $\overline{\mathbb{A}\text{ee}}_x$, $\overline{\mathbb{A}\text{ee}}_r$ and $\overline{\mathbb{A}\text{ev}}_{ij}\ (i,j\in\mathbb{S}_2^0)$ under different $\mathbf{Num}$ and noise intensities. Evidently, all the average relative errors inspected when $\mathbf{Num}=10^5$ are still less than $0.5\%$. Moreover, using the Kolmogorov--Smirnov test, we consider the hypothesis testing problem with its null hypothesis that the distribution $\mathbb{N}_2((\frac{\overline{r}}{b},\overline{r})^{\top},\Sigma^c_{\epsilon}(t))$ and the numerical probability distribution of (\ref{5.4}) under $\mathbf{Num}=10^5$ and case (III) (or (IV)) are from the same distribution. It turns out that the hypothesis will be accepted at 2\% significance level, $\forall\ i\in\mathbb{S}_3^0$. This together with Figs. 3 and 4 implies that the solution $(\mathbf{X}^{\triangleright}_{\epsilon}(t))$ have a good global approximation effect for the underlying one of (\ref{5.4}).

Summing up cases (I)-(IV), Theorem \ref{3-1} and ($\boldsymbol{\otimes}$-1) are well demonstrated through the numerical experiments.}}
\begin{table}[hthp]
	\caption{List of values of $\overline{\mathbb{A}\text{ee}}_x$, $\overline{\mathbb{A}\text{ee}}_r$ and $\overline{\mathbb{A}\text{ev}}_{ij}\ (i,j\in\mathbb{S}_2^0)$ at different $\epsilon$ and simulation numbers (Fixed parameter: $\theta=100$)}
	\label{table:para}
	\begin{small}
		\begin{center}
			\begin{tabular}{llllll}
				\hline\noalign{\smallskip}
				$(\epsilon,\mathbf{Num})$ & $\overline{\mathbb{A}\text{ee}}_x$ & $\overline{\mathbb{A}\text{ee}}_r$ & $\overline{\mathbb{A}\text{ev}}_{11}$ & $\overline{\mathbb{A}\text{ev}}_{12}$ & $\overline{\mathbb{A}\text{ev}}_{22}$\\
				\noalign{\smallskip}\hline\noalign{\smallskip}
				$(\text{case\ (III)},10^3)$ & $0.0838015\%$ & $0.102289\%$ & $3.762\%$ & $4.112\%$ & $3.583\%$\\
				$(\text{case\ (III)},10^4)$ & $0.0239435\%$ & $0.0284901\%$ & $1.144\%$ & $1.296\%$ & $1.132\%$\\
				$(\text{case\ (III)},10^5)$ & $0.0108989\%$ & $0.00950172\%$ & $0.375\%$ & $0.420\%$ & $0.367\%$\\
				$(\text{case\ (IV)},10^3)$ & $0.119937\%$ & $0.148777\%$ & $3.116\%$ & $3.691\%$ & $3.448\%$\\
				$(\text{case\ (IV)},10^4)$ & $0.0376384\%$ & $0.0458015\%$ & $1.108\%$ & $1.269\%$ & $1.077\%$\\
				$(\text{case\ (IV)},10^5)$ & $0.0201679\%$ & $0.0134185\%$ & $0.366\%$ & $0.429\%$ & $0.361\%$\\
				\noalign{\smallskip}\hline
			\end{tabular}
		\end{center}
	\end{small}
\end{table}
\begin{figure}[H]
	\begin{center}
		\resizebox{18cm}{9.6cm}{\includegraphics{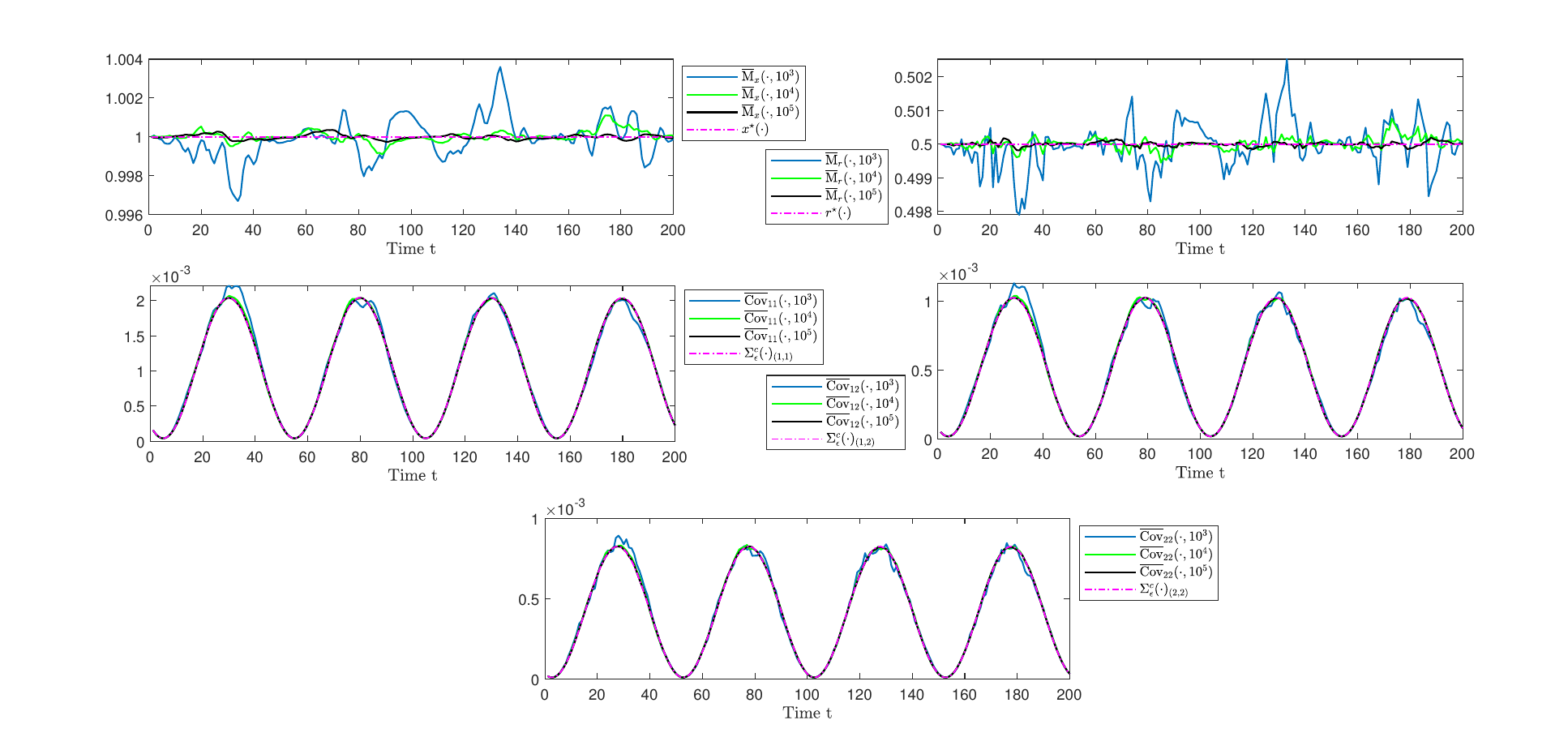}}
	\end{center}
	\makeatletter\def\@captype{figure}\makeatother \caption{The blue, green and black lines represent the functions $\overline{\text{M}}_x(\cdot)$, $\overline{\text{M}}_r(\cdot)$, $\overline{\text{Cov}}_{ij}(\cdot)\ (i,j\in\mathbb{S}_2^0)$ of (\ref{5.4}) at simulation number equals to $10^3$, $10^4$ and $10^5$, respectively. All the functions $x^{\star}(\cdot)$, $r^{\star}(\cdot)$ and $\Sigma^c_{\epsilon}(\cdot)_{(i,j)}\ (i,j\in\mathbb{S}_2^0)$ are depicted by the purple dotted lines. Fixed parameters: $\epsilon=0.05$, $\theta=100$.}
\end{figure}
\begin{figure}[H]
	\begin{center}
		\resizebox{18cm}{9.6cm}{\includegraphics{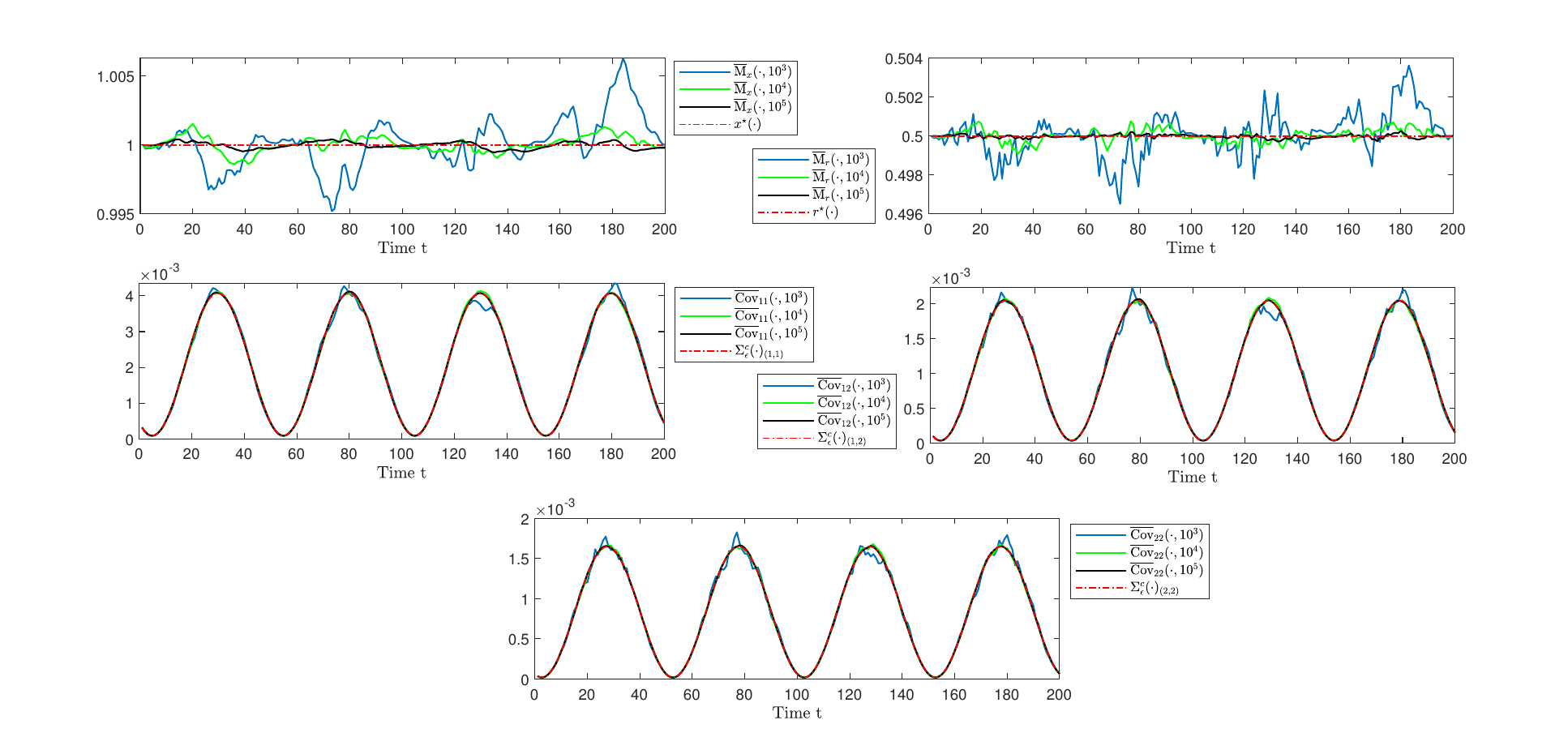}}
	\end{center}
	\makeatletter\def\@captype{figure}\makeatother \caption{The blue, green and black lines denote the functions $\overline{\text{M}}_x(\cdot)$, $\overline{\text{M}}_r(\cdot)$, $\overline{\text{Cov}}_{ij}(\cdot)\ (i,j\in\mathbb{S}_2^0)$ of (\ref{5.4}) at $\mathbf{Num}=10^3$, $10^4$ and $10^5$, respectively. All the functions $x^{\star}(\cdot)$, $r^{\star}(\cdot)$ and $\Sigma^c_{\epsilon}(\cdot)_{(i,j)}\ (i,j\in\mathbb{S}_2^0)$ are plotted by the corresponding red dotted lines. Fixed parameters: $\epsilon=0.1$, $\theta=100$.}
\end{figure} 
\end{example}
\begin{remark}\label{88}
	{\rm{Finally, we provide some concluding remarks:
	\begin{itemize}
	\item In the literature, the long-term properties of the Logistic models (e.g., (\ref{5.1})) can be well characterized by studying the reciprocal of solutions \cite{S63,S64}. However, such idea cannot be applicable to (\ref{5.4}). Mainly, the equation
	\begin{equation*}
	d\Big(\frac{1}{x_{\epsilon}(t)}\Big)=\Big(b-r_{\epsilon}(t)\cdot\frac{1}{x_{\epsilon}(t)}\Big)dt.
	\end{equation*} 
	is not linear due to the Ornstein--Uhlenbeck process. As a result, the solution of (\ref{5.4}) are not explicitly obtained as desired. This forces us to consider the numerical solution with enough computer simulations and sufficient iterations as a viable alternative for the SPS of (\ref{5.4}).
	\item Although the studies of (\ref{5.4}) and Example \ref{5:1} are mainly established under $m_0\neq\overline{r}$, the relevant analysis under $m_0=\overline{r}$ can be similarly carried out by Theorem \ref{3-1}. For example, when $m_0=\overline{r}$, we have
	\begin{equation*}
	\boldsymbol{\Phi}_{C^{\star}}(t)=\left(\begin{array}{cc}
	e^{-\overline{r}t} & \frac{\overline{r}te^{-\overline{r}t}}{b} \\ 
	0 & e^{-\overline{r}t}
	\end{array}\right),\ \ \ \boldsymbol{\Phi}^{-1}_{C^{\star}}(t)=\left(\begin{array}{cc}
	e^{\overline{r}t} & -\frac{\overline{r}te^{\overline{r}t}}{b} \\ 
	0 & e^{\overline{r}t}
	\end{array}\right).
	\end{equation*} 
	\item As was mentioned, Ornstein--Uhlenbeck process is a both biologically and mathematically reasonable assumption involved in the stochasticity and periodicity. However, only few studies have focused on the impact of such stochastic modeling approach on biological and ecological processes. As parts of our future work, we would like to explore effective techniques for analyzing the threshold dynamics of stochastic models motivated by Ornstein--Uhlenbeck process. 
	\item Although our main motivation of this section comes from the Logistic model, the ideas (i)-(iii) in Remark \ref{55} together with the Kolmogorov--Smirnov test is a general routine for studying the approximation effect of our approximate SPS on the underlying one of (\ref{1.2}). We expect to present more applications in ecology and biology in future work.
	\end{itemize}}}
		\end{remark}	
\section*{Appendix A} 
 \setcounter{equation}{0}
 \renewcommand\theequation{A.\arabic{equation}}
 \renewcommand\thelemma{A.\arabic{lemma}}
 \renewcommand\thetheorem{A.\arabic{theorem}}
 \textbf{(I) (Proof of Proposition \ref{2&1}):} Let $\psi_A(\lambda)=\sum_{i=0}^{l}a_i\lambda^{l-i}\ (a_0=1)$ for convenience. In view of $A\in\mathscr{T}(l)$ and Definition \ref{2:2}, there is a vector $(a_1,...,a_l)^{\top}:=\mathbf{a}^{\top}$ such that
 \begin{equation}\label{a.1}
 A=\left(\begin{array}{cc}
 -\mathbf{a}^{\langle l-1\rangle} & -a_l\\ 
 \mathbf{I}_{l-1} & \mathbb{O}
 \end{array}\right).
 \end{equation}
 Let $\lambda_i\ (i\in\mathbb{S}_{\ell_0}^0)$ be all the different roots of $\psi_A(\lambda)=0$, where $\ell_0\le l$. Below we divide the proof of Proposition \ref{2&1} into two steps.

 \textbf{Step 1.} (Proof for result (i) in Proposition \ref{2&1}): Using the definition of $\overline{\mathbf{CM}}(l)$, we have $|\lambda_i\lambda_j|\in(0,1)$ for any $i,j\in\mathbb{S}_{\ell_0}^0$. Then by Lemma \ref{2,1}, $\Xi_l$ is unique. In addition, it follows from Jordan theorem that there exists an invertible matrix $P_0$ such that
 $A=P_0J_{\text{ord}}P_0^{-1}$, where
 $$J_{\text{ord}}=\text{diag}\{J_{d(1)}(\lambda_1),...,J_{d(\ell_0)}(\lambda_{\ell_0})\},$$
 with $d(i):=\text{dim}(J_{d(i)}(\lambda_i))$, $\forall\ i\in\mathbb{S}_{\ell_0}^0$. 
 \\
 To proceed, let $m_1=\max_{i\in\mathbb{S}_{\ell_0}^0}\{d(i)\}$. In view of $J_{d(i)}(\lambda_i)=\lambda_i\mathbf{I}_{d(i)}+J_{d(i)}(0)$, where $(J_{d(i)}(0))^{d(i)-1}=\mathbb{O}$, then
 $$(J_{d(i)}(\lambda_i))^r=\sum_{k=0}^{d(i)-1}\mathbb{C}_r^k\lambda_i^{r-k}\mathbf{I}_{d(i)}(J_{d(i)}(0))^k,\ \ \ \forall\ r\ge m_1,$$
 where $\mathbb{C}$ is the combinatorial number. By calculation, we obtain $J_{\text{ord}}^m=\text{diag}\{(J_{d(1)}(\lambda_1))^m,...,(J_{d({\ell_0})}(\lambda_{\ell_0}))^m\}$, and
 \begin{equation}\label{a.2}
 \lim_{m\rightarrow\infty}(J_{d(i)}(\lambda_i))^m=\lim_{\substack{m\ge m_1\\ m\rightarrow\infty}}\sum_{k=0}^{d(i)-1}\mathbb{C}_r^k\lambda_i^{r-k}\mathbf{I}_{d(i)}(J_{d(i)}(0))^k=\mathbb{O},\ \ \ \forall\ i\in\mathbb{S}_{\ell_0}^0,
 \end{equation}
 which means $\lim_{m\rightarrow\infty}A^m=\mathbb{O}$. In this sense, by complex calculation and a standard argument in \cite{S65}, $\Xi_l$ has an explicit form:
 \begin{equation}\label{a.3}
 \Xi_l=\sum_{k=0}^{\infty}A^k\amalg_{l,1}(A^k)^{\top},
 \end{equation}
 which can be verified by the fact
 \begin{equation*}
 \begin{split}
 \Xi_l-A\Xi_lA^{\top}=&\sum_{k=0}^{\infty}A^k\amalg_{l,1}(A^k)^{\top}-A\Big(\sum_{k=0}^{\infty}A^k\amalg_{l,1}(A^k)^{\top}\Big)A^{\top}\\
 =&\sum_{k=0}^{\infty}A^k\amalg_{l,1}(A^k)^{\top}-\sum_{k=1}^{\infty}A^k\amalg_{l,1}(A^k)^{\top}\\
 =&\amalg_{l,1}.
 \end{split}
 \end{equation*}
 It readily follows from (\ref{a.3}) that for any $\mathbf{X}\in\mathbb{R}^l$,
 \begin{align}\label{a.4}
 \mathbf{X}^{\top}\Xi_l\mathbf{X}=\sum_{k=0}^{\infty}\big|\amalg_{l,1}(A^k)^{\top}\mathbf{X}\big|^2\ge 0.
 \end{align}
 Thus, $\Xi_l\succeq\mathbb{O}$. Below we verify $\Xi_l\succ\mathbb{O}$ by using a contradiction argument. Suppose that there exists a $\mathbf{X}_{\phi}\in\mathbb{R}^l\setminus\{\mathbf{0}\}$ satisfying $\mathbf{X}_{\phi}^{\top}\Xi_l\mathbf{X}_{\phi}=0$, then
 \begin{equation}\label{a.5}
 \amalg_{l,1}(A^k)^{\top}\mathbf{X}_{\phi}=\mathbf{0},\ \ \forall\ k\in \{0\}\cup\mathbb{S}_{\infty}^0.
 \end{equation}
 Using Cayley--Hamilton theorem, (\ref{a.5}) is equivalent to
 \begin{equation}\label{a.6}
 \mathbf{X}_{\phi}^{\top}C_0=\mathbf{0}^{\top},
 \end{equation}
 where $C_0=\left(\amalg_{l,1},A\amalg_{l,1},...,A^{l-1}\amalg_{l,1}\right)$.

 Direct calculation shows that
 $C_0=\left(\boldsymbol{\xi}_1,\mathbb{O}_{l,l-1},\boldsymbol{\xi}_2,\mathbb{O}_{l,l-1},...,\boldsymbol{\xi}_l,\mathbb{O}_{l,l-1}\right)$, where the $\mathbb{R}^{l\times 1}$-valued vectors  $\boldsymbol{\xi}_j\ (j\in\mathbb{S}_l^0)$ is:
 \begin{equation*}
 \left(\boldsymbol{\xi}_1,\boldsymbol{\xi}_2,...,\boldsymbol{\xi}_l\right)=\left(\begin{array}{cccccc}
 1 & \alpha_1 & \alpha_2 & \alpha_3 & \cdots & \alpha_{l-1} \\ 
 0 & 1 & \alpha_1 & \alpha_2 & \cdots & \alpha_{l-2} \\ 
 0 & 0 & 1 & \alpha_1 & \cdots & \alpha_{l-3} \\ 
 0 & 0 & 0 & 1 & \cdots & \alpha_{l-4} \\ 
 \vdots & \vdots & \vdots & \vdots & \begin{sideways}$\ddots$\end{sideways} & \vdots \\ 
 0 & 0 & 0 & 0 & \cdots & 1
 \end{array}\right), 
 \end{equation*}
 with $\alpha_k\ (k\in\mathbb{S}_{l-1}^0)$ determined by the iterative scheme 
 $\alpha_k=-\sum_{i=1}^{k}a_i\alpha_{k-i}\ (\alpha_0=1)$. 
 \\
 Intuitively, we have $|\left(\boldsymbol{\xi}_1,\boldsymbol{\xi}_2,...,\boldsymbol{\xi}_l\right)|=1$, then rank$(C_0)=l$. This together with (\ref{a.6}) implies 
 $$\mathbf{X}_{\phi}=(C_0^{-1})^{\top}\mathbf{0}=\mathbf{0},$$
 which leads to a contradiction. Therefore, $\Xi_l\succ\mathbb{O}$.

 \textbf{Step 2.} (Proof for result (ii) in Proposition \ref{2&1}): Let $\Xi_l:=(q_{ij})_{l\times l}$ for convenience. We first define some constants by
 $$\gamma_i=-a_iq_{11}-\sum_{k=2}^{l}(a_{i+1-k}+a_{i+k-1})q_{1k},\ \ \ \forall\ i\in\mathbb{S}_l^0,$$
 where $a_j=0$ for any $j\notin\mathbb{S}_l^0$. In view of (\ref{a.1}), we obtain
 \begin{equation}\label{a.7}
 A\Xi_lA^{\top}=\left(\begin{array}{cccccc}
 -\sum_{i=1}^{n}a_i\gamma_i & \gamma_1 & \gamma_2 & \cdots & \gamma_{l-1} \\ 
 \gamma_1 & q_{11} & q_{12} & \cdots & q_{1,n-1} \\ 
 \gamma_2 & q_{12} & q_{22} & \cdots & q_{2,l-1} \\ 
 \vdots & \vdots & \vdots & \begin{sideways}$\ddots$\end{sideways} & \vdots \\ 
 \gamma_{l-1} & q_{1,l-1} & q_{2,l-1} & \cdots & q_{l-1,l-1}
 \end{array}\right).
 \end{equation}
 Inserting (\ref{a.7}) into Eq. (\ref{2.1}) yields
 \begin{equation*}
 \Xi_l=\left(\begin{array}{cccccc}
 q_{11} & q_{12} & q_{13} & q_{14} & \cdots & q_{1l} \\ 
 q_{12} & q_{11} & q_{12} & q_{13} & \cdots & q_{1,l-1} \\ 
 q_{13} & q_{12} & q_{11} & q_{12} & \cdots & q_{1,l-2} \\ 
 q_{14} & q_{13} & q_{12} & q_{11} & \cdots & q_{1,l-2} \\ 
 \vdots & \vdots & \vdots & \vdots & \begin{sideways}$\ddots$\end{sideways} & \vdots \\ 
 q_{1l} & q_{1,l-1} & q_{1,l-2} & q_{1,l-3} & \cdots & q_{11}
 \end{array}\right),
 \end{equation*} 
 with $q_{1k}\ (k\in\mathbb{S}_l^0)$ satisfying
 \begin{equation*}
 \begin{cases}
 q_{11}+\sum\limits_{i=1}^{l}a_i\gamma_i=1,\\
 q_{1i}-\gamma_{i-1}=0,\ \ \forall\ i\in\mathbb{S}_l^1.
 \end{cases}
 \end{equation*}
 It then follows from the definition of $\mathscr{G}_{c,A}$ that $(q_{11},q_{12},...,q_{1l})^{\top}:=\boldsymbol{q}$ is the solution of equation $\mathscr{G}_{c,A}\boldsymbol{x}=\mathbf{e}_l$. Combined with the uniqueness of $\Xi_l$, we have $\boldsymbol{q}=\boldsymbol{\zeta}$. Thus, result (ii) in Proposition \ref{2&1} holds.   
 \\
 \\
 \textbf{(I) (Proof of Proposition \ref{2&2}):} In fact, Proposition \ref{2&2} is evidently true when $l=1$. Thus we discuss the case of $l\ge 2$.

 Consider the following vector $\mathbf{Y}(t)=(Y_1(t),...,Y_l(t))^{\top}$:
 \begin{equation}\label{a.8}
 Y_l(t)=X_l(t),\ \ \text{and}\ \ Y_k(t)=Y_{k+1}^{'}(t),\ \ \forall\ k\in\mathbb{S}_{l-1}^0,
 \end{equation}
 where $(X_1(t),...,X_l(t))^{\top}:=\mathbf{X}(t)$ is the solution of equation $d\mathbf{X}=C\mathbf{X}dt$.

 Now we define $C=(c_{ij})_{l\times l}$ and stipulate that $c_{1,0}=1$. An application of recursion method coupled with (\ref{a.8}), Definition \ref{2:3} and $C\in\mathcal{U}_{cm}(l)$ yields that
 \begin{equation}\label{a.9}
 \mathbf{Y}(t)=\left(\begin{array}{c}
 \boldsymbol{\beta}_lC^{l-1}\\ 
 \boldsymbol{\beta}_lC^{l-2}\\ 
 \cdots\\ 
 \boldsymbol{\beta}_l
 \end{array}\right)\mathbf{X}(t)=\mathcal{D}\mathbf{X}(t).
 \end{equation}
 and
 \begin{equation}\label{a.10}
 \begin{cases}
 (\boldsymbol{\beta}_lC^j)^{\langle l-j-1\rangle}=\mathbf{0},\ \ \ \forall\ j\in\mathbb{S}_{l-2}^0,\\
 (\boldsymbol{\beta}_lC^k)^{(l-k)}=\prod_{i=l-k}^{l-1}c_{i+1,i}\neq 0,\ \ \ \forall\ k\in\mathbb{S}_{l-1}^0.
 \end{cases}
 \end{equation}
 Thus, $\mathcal{D}$ is an upper triangular matrix. In view of
 \begin{equation}\label{a.11}
 |\mathcal{D}|=\prod_{i=1}^{l-1}(\boldsymbol{\beta}_lC^{l-i})^{(i)}=\prod_{i=1}^{l-1}c_{i+1,i}^i\neq 0,
 \end{equation}
 then $\mathcal{D}\in\mathcal{U}(l)$, which implies that $d\mathbf{Y}=\mathcal{D}C\mathcal{D}^{-1}\mathbf{Y}dt$. Combining (\ref{a.8}) and Definition \ref{2:2} yields
 \begin{equation*}
 \mathcal{D}C\mathcal{D}^{-1}\in\mathscr{T}(l).
 \end{equation*}
 The proof is complete.
 \section*{Data Availability Statements}
 No data was used for the research described in the article.
 \section*{Conflict of interest}
 The authors declare that they have no conflict of interest.
 \section*{Acknowledgments}
 The research of Baoquan Zhou was partially supported by the Key Laboratory of Applied Statistics of MOE at Northeast Normal University. The research of Hao Wang was partially supported by the Natural Sciences and Engineering Research Council of Canada (Individual Discovery Grant RGPIN-2020-03911 and Discovery Accelerator Supplement Award RGPAS-2020-00090) and the Canada Research Chairs program (Tier 1 Canada Research Chair Award). The research of Tianxu Wang was partially supported by the Interdisciplinary Lab for Mathematical Ecology \& Epidemiology at the University of Alberta. The research of Daqing Jiang was partially supported by the National Natural Science Foundation of China (No. 11871473) and the Fundamental Research Funds for the Central Universities (No. 22CX03030A).
 \section*{References}
	\bibliographystyle{model1-num-names}
		
\end{document}